\RequirePackage{etex}
\documentclass[12pt,reqno]{amsart}
\usepackage[centering,margin=1in]{geometry}
\usepackage{amssymb,amsfonts,amsmath,mathrsfs,etoolbox,mathtools,accents,bm,bbm,booktabs,textcomp,xspace}
\usepackage[utf8]{inputenc}
\usepackage[UKenglish,iso]{isodate}
\usepackage{enumerate}
\usepackage[dvipsnames]{xcolor}
\definecolor{dark-blue}{rgb}{0.15,0.15,0.4}
\definecolor{dark-red}{rgb}{0.4,0.15,0.15}
\definecolor{medium-blue}{rgb}{0,0,0.5}
\usepackage[breaklinks=true,linktocpage=true,colorlinks=true,linktoc=all,linkcolor=dark-blue,citecolor=dark-blue,filecolor=dark-blue,urlcolor=medium-blue,backref=page,hypertexnames=false]{hyperref}
\usepackage[nameinlink,capitalize]{cleveref}
\usepackage{autonum}
\usepackage{caption,subcaption}
\usepackage{constants}
\newconstantfamily{constant}{symbol=c}
\usepackage{csquotes}
\usepackage{orcidlink}
\allowdisplaybreaks

\crefname{conjecture}{Conjecture}{Conjectures}
\Crefname{conjecture}{Conjecture}{Conjectures}
\crefname{figure}{Figure}{Figures}
\Crefname{figure}{Figure}{Figures}
\crefname{lemma}{Lemma}{Lemmata}
\Crefname{lemma}{Lemma}{Lemmata}

\crefmultiformat{theorem}{Theorems~#2#1#3}{ and~#2#1#3}{,~#2#1#3}{,~and~#2#1#3}
\crefmultiformat{lemma}{Lemmata~#2#1#3}{ and~#2#1#3}{,~#2#1#3}{,~and~#2#1#3}

\usepackage{cite}
\makeatletter
\newcommand{\citecomment}[2][]{\citen{#2}#1\citevar}
\newcommand{\citeone}[1]{\citecomment{#1}}
\newcommand{\citetwo}[2][]{\citecomment[, #1]{#2}}
\newcommand{\citevar}{\@ifnextchar\bgroup{; \citeone}{\@ifnextchar[{; \citetwo}{]\xspace}}}
\newcommand{\citefirst}{\@ifnextchar\bgroup{\citeone}{\@ifnextchar[{\citetwo}{]}}}
\newcommand{\cites}{[\citefirst}
\makeatother

\usepackage{footnote}
\makesavenoteenv{tabular}
\makesavenoteenv{table}
\usepackage{makecell}

\newtheorem{theorem}{Theorem}[section]
\newtheorem{conjecture}[theorem]{Conjecture}
\newtheorem{corollary}[theorem]{Corollary}
\newtheorem{definition}[theorem]{Definition}
\newtheorem{lemma}[theorem]{Lemma}
\newtheorem{proposition}[theorem]{Proposition}
\theoremstyle{definition}
\newtheorem{remark}[theorem]{Remark}
\numberwithin{equation}{section}

\renewcommand{\Re}{\mathrm{Re}}
\renewcommand{\Im}{\mathrm{Im}}
\renewcommand{\hat}{\widehat}
\renewcommand{\tilde}{\widetilde}

\DeclareMathOperator{\arctanh}{arctanh}
\renewcommand{\epsilon}{\varepsilon}
\patchcmd{\section}{\scshape}{\bfseries}{}{}
\makeatletter
\renewcommand{\@secnumfont}{\bfseries}
\makeatother
\makeatletter\newcommand{\tpmod}[1]{{\@displayfalse \pmod{#1}}}
\renewcommand*\backref[1]{$\uparrow$\thinspace\ifx#1\relax \else #1 \fi}
\newcommand{\MYhref}[3][dark-red]{\href{#2}{\color{#1}{#3}}}

\usepackage{tikz}
\usepackage{xparse}

\usepackage{mwe}

\tikzset{
annotatedImage/x/.initial = 0.7,
annotatedImage/y/.initial = 0.7,
annotatedImage/width/.initial = 1,
annotatedImage/.unknown/.code = {
\edef\tikzappend{\noexpand\tikzset{annotatedImage/.append style = {\pgfkeyscurrentname=\pgfkeyscurrentvalue}}}
\tikzappend
},
annotatedImage/.style = {
draw=red, ultra thick, rounded corners, rectangle,
}
}

\newsavebox\annotatedImageBox

\newcommand\AnnotatedImageVal[1]{\pgfkeysvalueof{/tikz/annotatedImage/#1}}
\newcommand\SetUpAnnotatedImage[2]{
\tikzset{annotatedImage/.cd, #1}
\sbox\annotatedImageBox{\includegraphics[width=\AnnotatedImageVal{width}\textwidth,keepaspectratio]{#2}}
\pgfmathsetmacro\annotatedHeight{\ht\annotatedImageBox/28.453}
\pgfmathsetmacro\annotatedWidth{\wd\annotatedImageBox/28.453}
}

\NewDocumentCommand\annotatedImage{ O{} m m}{
\bgroup
\SetUpAnnotatedImage{#1}{#2}
\begin{tikzpicture}[xscale=\annotatedWidth, yscale=\annotatedHeight]
\node[inner sep=0, anchor=south west] (image) at (0,0) {\usebox{\annotatedImageBox}};
\node[annotatedImage] at (\AnnotatedImageVal{x},\AnnotatedImageVal{y}) {#3};
\end{tikzpicture}
\egroup
}

\newcommand\annotate[1][]{\node[annotatedImage,#1]}

\newenvironment{AnnotatedImage}[2][1]{
\SetUpAnnotatedImage{#1}{#2}
\tikzpicture[xscale=\annotatedWidth, yscale=\annotatedHeight]
\node[inner sep=0, anchor=south west] at (0,0) {\usebox{\annotatedImageBox}};
}{\endtikzpicture}

\begin{document}

\title[The Prime Geodesic Theorem for the Picard Orbifold]{The Prime Geodesic Theorem for the Picard Orbifold}

\author[Ikuya Kaneko]{Ikuya Kaneko\,\orcidlink{0000-0003-4518-1805}}
\address[Ikuya Kaneko]{The Division of Physics, Mathematics and Astronomy, California Institute of Technology, 1200 East California Boulevard, Pasadena, CA 91125, United States of America}
\email{\href{mailto:ikuyak@icloud.com}{ikuyak@icloud.com}}
\urladdr{\href{https://sites.google.com/view/ikuyakaneko/}{https://sites.google.com/view/ikuyakaneko/}}
\thanks{This research was partially supported by the Masason Foundation until the conclusion of the scholarship in June 2024. We hereby take this opportunity to extend our sincere gratitude for their hospitality and~patience throughout the duration of their esteemed support.}

\subjclass[2020]{Primary 11F72, 11L40, 11R42; Secondary 11F30, 11L05, 11M26}

\keywords{Prime geodesic theorem, Picard orbifold, mean Lindel\"{o}f hypothesis, subconvexity, Brun--Titchmarsh theorem, character sums, Zagier $L$-series, zero density theorem, spectral exponential sum}

\cleanlookdateon

\date{\today~(ISO 8601)}

\begin{abstract}
We establish the prime geodesic theorem for the Picard orbifold $\mathrm{PSL}_{2}(\mathbb{Z}[i]) \backslash \mathbb{H}^{3}$, wherein the error term shrinks proportionally to improvements in the subconvex exponent for quadratic Dirichlet $L$-functions over $\mathbb{Q}(i)$. Our result sheds light on a venerable conjecture by attaining an unconditional exponent of $1.483$ and a conditionally superior exponent of $1.425$ under the generalised Lindel\"{o}f hypothesis. The argument synthesises, among other elements, the complete resolution of Koyama's (2001) mean Lindel\"{o}f hypothesis over $\mathbb{Q}(i)$, an improved Brun--Titchmarsh-type theorem over short intervals, a bootstrapped multiplicative exponent pair in the limiting regime, and a zero density theorem for the symplectic family of quadratic characters. Notably, despite the theoretical strength of our manifestations towards the mean Lindel\"{o}f hypothesis, the fundamental toolbox relies exclusively on the optimal mean square asymptotics for the Fourier coefficients of Maa{\ss} cusp forms via the~pre-Kuznetsov formula.
\end{abstract}

\maketitle
\tableofcontents

\section{Introduction}

\subsection{Background}
Since Riemann's utilisation of automorphy for the theta function in one of his proofs of the functional equation for the zeta function, analytic number theorists have been captivated by the rich interplay between arithmetic and automorphic forms. Given the enigma surrounding the distribution of primes, precluding a definitive understanding of their intrinsic nature, it is natural to consider analogous entities from hyperbolic geometry -- prime geodesics. The \textit{prime geodesic theorem} seeks to asymptotically elucidate the distribution of oriented primitive closed geodesics on a fixed hyperbolic orbifold. For the modular orbifold, landmark results can be found in~\cites{BalogBiroHarcosMaga2019}{Cai2002}{CherubiniGuerreiro2018}{Iwaniec1984}{Kaneko2024}{LuoSarnak1995}{SoundararajanYoung2013}.

The present paper sheds some light on the prime geodesic theorem attached to arithmetic hyperbolic $3$-orbifolds, where the upper-half plane denoted by $\mathbb{H}^{3}$ is embedded geometrically in the $3$-dimensional Euclidean space $\mathbb{C} \times \mathbb{R}$, while algebraically in the Hamiltonian algebra with vanishing fourth coordinate. It is naturally endowed with a transitive action of $\mathrm{SL}_{2}(\mathbb{C})$ via linear fractional transformations, and the stabiliser of any point is a conjugate of $\mathrm{SU}_{2}(\mathbb{C})$. Hence, we identify
\begin{equation}
\mathbb{H}^{3} \cong \mathrm{SL}_{2}(\mathbb{C})/\mathrm{SU}_{2}(\mathbb{C}) \cong \mathrm{Z}(\mathbb{C}) \backslash \mathrm{GL}_{2}(\mathbb{C})/\mathrm{U}_{2}(\mathbb{C}).
\end{equation}
Given the basis $\{1, i, j, k \}$ of Hamiltonian quaternions, a typical point $v \in \mathbb{H}^{3}$ is represented by $v = z+rj$ with $r > 0$ and $z = x+yi \in \mathbb{C}$; thus, $\Im(v) \coloneqq r$, $\Im(z) \coloneqq y$, and $\Re(z) \coloneqq x$.~As a Riemannian geometry, the space $\mathbb{H}^{3}$ carries the hyperbolic metric $r^{-1} \sqrt{dx^{2}+dy^{2}+dr^{2}}$ and the corresponding volume element $r^{-3} \, dx \, dy \, dr$. Then the quotient $\mathrm{PSL}_{2}(\mathbb{Z}[i]) \backslash \mathbb{H}^{3}$ is known as the \textit{Picard orbifold}~\cites{Picard1884}. Unless otherwise specified, the symbol $\Gamma$ is reserved to~signify $\mathrm{PSL}_{2}(\mathbb{Z}[i])$, but the overall methodology is designed to guide broader applicability to Bianchi orbifolds~\cites{Bianchi1892} attached to imaginary quadratic fields $\mathbb{Q}(\sqrt{-d})$ for the $9$ Heegner numbers\footnote{In principle, the method ought to extend to any squarefree integer $d > 0$ without fundamental deadlocks, contingent upon the notion of primitive discriminants, as detailed in Speiser's doctoral thesis~\cites{Speiser1909}.} $d \in \{1, 2, 3, 7, 11, 19, 43, 67, 163 \}$, along with their respective congruence covers. For relevant nomenclature and theoretical foundations, the reader is directed to~\cites{ElstrodtGrunewaldMennicke1998}{MaclachlanReid2003}{Thurston2022}.

To explicate the conceptual paradigm in the $3$-dimensional setting, it is now convenient~to introduce the \textit{Chebyshev-like counting function}
\begin{equation}
\Psi_{\Gamma}(x) \coloneqq \sum_{\mathrm{N}(P) \leq x} \Lambda_{\Gamma}(P),
\end{equation}
where the sum runs through all hyperbolic and loxodromic conjugacy classes of $\Gamma$ of norm~not exceeding $x$. The summand is an analogue of the von Mangoldt function given by $\Lambda_{\Gamma}(P) \coloneqq \log \mathrm{N}(P_{0})$ if $P_{0}$ is the primitive closed geodesic underlying $P$, and $\Lambda_{\Gamma}(P) \coloneqq 0$ otherwise. Then the first effective asymptotic formula for $\Psi_{\Gamma}(x)$ is attributed to Sarnak~\cites[Theorem~5.1]{Sarnak1983}, which shows that for any cofinite Kleinian group $\Gamma \subset \mathrm{PSL}_{2}(\mathbb{C})$,
\begin{equation}\label{eq:Sarnak}
\mathcal{E}_{\Gamma}(x) \coloneqq \Psi_{\Gamma}(x)-\sum_{1 < s_{j} \leq 2} \frac{x^{s_{j}}}{s_{j}} \ll_{\Gamma, \epsilon} x^{\frac{5}{3}+\epsilon},
\end{equation}
where the subtracted term arises from the small eigenvalues $\lambda_{j} = s_{j}(2-s_{j}) = 1+t_{j}^{2} < 1$ of the Laplace--Beltrami operator $\Delta$ on $\Gamma \backslash \mathbb{H}^{3}$. Nakasuji~\cites[Theorem~5.2]{Nakasuji2000}[Theorem~4.1]{Nakasuji2001} leveraged the theory of Selberg zeta functions to deduce the \textit{spectral explicit formula}, thereby reproducing~\eqref{eq:Sarnak} additionally for any cocompact Kleinian group $\Gamma$. Unlike the Riemann~zeta function, the Selberg zeta function lacks a natural Dirichlet series representation, precluding the straightforward utilisation of standard analytic tools for Dirichlet polynomials.

For $\Gamma = \mathrm{PSL}_{2}(\mathbb{Z}[i])$, the existence of a fairly explicit spherical Kuznetsov formula~\cites{Motohashi1996}{Motohashi1997}{Motohashi1997-2} and the Selberg trace formula~\cites{Elstrodt1985}{Szmidt1983}{Tanigawa1977} enables a polynomial power-saving improvement over the $\frac{5}{3}$-barrier~\eqref{eq:Sarnak}. Of particular significance is the seminal result of Balkanova et al.~\cites[Theorem~1.1]{BalkanovaChatzakosCherubiniFrolenkovLaaksonen2019}, who proved the exponent $\frac{13}{8}+\epsilon$ by adapting the ideas of Luo and Sarnak~\cites[Section~3]{LuoSarnak1995} to the setting of $\mathbb{Q}(i)$. In faithful emulation of well-trodden paths in the $2$-dimensional case, Balog et al.~\cites[Corollary~1.4]{BalogBiroCherubiniLaaksonen2022} derived the improved exponent $\frac{3}{2}+\frac{4\vartheta}{7}+\epsilon$, where $\vartheta$ denotes a subconvex exponent for quadratic~Dirichlet $L$-functions over $\mathbb{Q}(i)$ in the conductor aspect. The best known unconditional result follows upon inserting~\cites[Theorem~1.1]{BalkanovaFrolenkov2022}[Theorem~3]{Qi2024} into the most comprehensive version by the author~\cites[Theorem~1.4]{Kaneko2020}[Equation~(1-9)]{Kaneko2022-2}, yielding
\begin{equation}\label{eq:Balkanova-Frolenkov}
\mathcal{E}_{\Gamma}(x) \ll_{\epsilon} x^{\delta_{0}(\vartheta)+\epsilon}, \qquad \delta_{0}(\vartheta) \coloneqq \frac{34+12\vartheta+8(2+\vartheta)\min(\frac{1}{4}, 2\vartheta)}{23+10\min(\frac{1}{4}, 2\vartheta)}.
\end{equation}
In contrast to any formulation therein, it proves both advantageous and aesthetically elegant to refrain from isolating the term $\frac{3}{2}$, as~\eqref{eq:Balkanova-Frolenkov} surpasses this artificial barrier provided $\vartheta < \frac{1}{24}$.

\begin{table}[!htp]
\centering
\setlength\tabcolsep{10pt}
\caption{A chronology of the known exponents}
\label{table}
\begin{tabular}{lccl}
	\toprule
	Year & Source & Fractional & Decimal \\ \midrule \midrule
	$1983$ & \cites[Theorem~5.1]{Sarnak1983} & $\frac{5}{3}$ & $1.66666 \cdots$ \\ \midrule
	$2019$ & \cites[Theorem~1.1]{BalkanovaChatzakosCherubiniFrolenkovLaaksonen2019} & $\frac{13}{8}$ & $1.625$ \\ \midrule
	$2022$ & \cites[Corollary~1.4]{BalogBiroCherubiniLaaksonen2022} & $\frac{67}{42}$ & $1.59524 \cdots$ \\ \midrule
	$2022$ & \cites[Theorem~1.1]{BalkanovaFrolenkov2022} & $\frac{376}{237}$ & $1.58649 \cdots$ \\ \midrule
	$2024$ & \cites[Theorem~3]{Qi2024} & $\frac{242}{153}$ & $1.58169 \cdots$ \\ \midrule
	$2001^{\dagger}$ & \cites[Theorem~1.1]{Koyama2001} & $\frac{11}{7}$ & $1.57143 \cdots$ \\ \midrule
	$2020^{\dagger}$ & \cites[Theorem~1.2]{BalkanovaFrolenkov2020}\footnote{The definition of a subconvex exponent in~\cites[Equation~(3.24)]{BalkanovaFrolenkov2020} differs from ours.} & $\frac{3}{2}$ & $1.5$ \\ \midrule
	$2022^{\dagger}$ & \cites[Corollary~1.5]{BalogBiroCherubiniLaaksonen2022} & $\frac{34}{23}$ & $1.47826 \cdots$ \\ \midrule
	$2022^{\dagger\ddagger}$ & \cites[Theorem~5.3]{Kaneko2022-2} & $\frac{10}{7}$ & $1.42857 \cdots$ \\ \bottomrule
	\multicolumn{4}{l}{\footnotesize{$^{\dagger}$Conditional on the generalised Lindel\"{o}f hypothesis $\vartheta = 0$}} \\
	\multicolumn{4}{l}{\footnotesize{$^{\ddagger}$Conditional on~\cites[Conjecture~5.1]{Kaneko2022-2}}}
\end{tabular}
\end{table}

One may consult \cref{table} for a chronology of the known exponents up to $\epsilon$ alongside~their corresponding numerical values, presented in both fractional and decimal forms. The sharpest known subconvex exponent $\vartheta = \frac{1}{6}$ of Nelson~\cites[Theorem~1.1]{Nelson2020} is substituted as necessary. On the conditional front, Koyama~\cites[Equation~(3.2)]{Koyama2001} proceeds under the assumption~of the \textit{mean Lindel\"{o}f hypothesis} for the Rankin--Selberg $L$-function over $\mathbb{Q}$, and subsequent work typically centres around a power-saving improvement towards the mean Lindel\"{o}f hypothesis. Nonetheless, we invoke~\cites[Theorem~1.3]{BalkanovaFrolenkov2022} and replace his assumption with $\vartheta = 0$, thereby maintaining consistency in the last three rows of \cref{table}. For additional~relevant progress, we refer the reader to~\cites{Avdispahic2018}{Benes2022}{ChatzakosCherubiniLaaksonen2022}{DeverMilicevic2023}{Kaneko2020}{Laaksonen2019}{Nakasuji2002}{Nakasuji2004}{Nakasuji2004-2}.

\subsection{Statement of results}
The primary objective of the present paper is to update \cref{table} by incorporating improved results -- both unconditional and conditional -- and to supply solid evidence in support of the validity of the seemingly optimal exponent $1+\epsilon$. Our first result serves as a $3$-dimensional analogue of~\cites[Theorem~1.1]{SoundararajanYoung2013} and establishes an unconditional pointwise power-saving improvement over the recent result of Qi~\cites[Theorem~3]{Qi2024}.
\begin{theorem}\label{thm:unconditional}
Let $\Gamma = \mathrm{PSL}_{2}(\mathbb{Z}[i])$. For a generator $D$ of the fundamental discriminant of a quadratic extension of $\mathbb{Q}(i)$, let $\chi_{D} = (\frac{D}{\cdot})$ denote the primitive quadratic character modulo $\mathfrak{d} \unlhd \mathbb{Z}[i]$, and let $\vartheta \in [0, \frac{1}{4})$ satisfy the subconvex bound in the conductor aspect
\begin{equation}\label{eq:conductor-aspect}
L \left(\frac{1}{2}+it, \chi_{D} \right) \ll_{\epsilon} \mathrm{N}(D)^{\vartheta+\epsilon} (1+|t|)^{\vartheta^{\prime}+\epsilon}
\end{equation}
for some arbitrarily large $\vartheta^{\prime} \geq 0$. Then we have for any $\epsilon > 0$ that
\begin{equation}\label{eq:unconditional}
\mathcal{E}_{\Gamma}(x) \ll_{\epsilon} x^{\frac{23+6\vartheta}{16}+\epsilon}+x^{\frac{7+3\vartheta}{5}+\epsilon}.
\end{equation}
In particular, the current record $\vartheta = \frac{1}{6}$ unconditionally yields
\begin{equation}\label{eq:substitute-1/6}
\mathcal{E}_{\Gamma}(x) \ll_{\epsilon} x^{\frac{3}{2}+\epsilon},
\end{equation}
while the generalised Lindel\"{o}f hypothesis $\vartheta = 0$ conditionally yields
\begin{equation}
\mathcal{E}_{\Gamma}(x) \ll_{\epsilon} x^{\frac{23}{16}+\epsilon},
\end{equation}
where $\frac{3}{2} = 1.5$ and $\frac{23}{16} = 1.4375$.
\end{theorem}

Moreover, the second moment theory of the prime geodesic theorem in the $3$-dimensional setting was initiated by Chatzakos, Cherubini, and Laaksonen~\cites[Theorem~1.1]{ChatzakosCherubiniLaaksonen2022}, as well as by the author~\cites[Theorem~1.1]{Kaneko2022-2}. We emphasise that any second moment approach to $\mathcal{E}_{\Gamma}(x)$ involves a more refined manoeuvre than that required for the first moment,~specifically the square mean Lindel\"{o}f hypothesis formulated in \cref{conj:square-mean}, which lies deeper than the mean Lindel\"{o}f hypothesis formulated in \cref{conj:mean-Lindelof}. In particular, the latter follows~from the former by virtue of the Cauchy--Schwarz inequality and Weyl's law. Essential~ingredients for establishing \cref{thm:unconditional} are adaptable to the second moment, thereby leading to a result slightly weaker than \cref{thm:unconditional}, modulo the discrepancy between \cref{conj:mean-Lindelof,conj:square-mean}. This phenomenon has yet to occur in the case of the full modular orbifold, as the best known hybrid result of the author~\cites[Theorem~1.1]{Kaneko2020} corresponds via~\cites[Lemma~2.1]{Kaneko2020} to the pointwise result of Luo and Sarnak~\cites[Theorem~1.4]{LuoSarnak1995}, rather than to~\cites[Theorem~1.1]{SoundararajanYoung2013}.
\begin{theorem}\label{thm:second-moment}
Keep the notation and assumptions as in \cref{thm:unconditional}. Let $\eta \in [0, \min(\frac{1}{4}, 2\vartheta)]$ denote an additional exponent towards the square mean Lindel\"{o}f hypothesis in \cref{conj:square-mean}. Then we have for any $Y \in [X^{\frac{1}{2}-\vartheta}, X^{1-\vartheta+\frac{\eta}{2}}]$ and $\epsilon > 0$ that
\begin{equation}\label{eq:second-moment}
\frac{1}{Y} \int_{X}^{X+Y} |\mathcal{E}_{\Gamma}(x)|^{2} \, dx \ll_{\epsilon} X^{\frac{13+2\vartheta+4(3+\vartheta) \eta}{4(1+\eta)}} Y^{-\frac{1}{2(1+\eta)}}+X^{\frac{2(5+\vartheta+(3+2\vartheta) \eta)}{3+2\eta}+\epsilon} Y^{-\frac{2}{3+2\eta}}.
\end{equation}
\end{theorem}

Inserting~\eqref{eq:second-moment} with $\eta = \frac{1}{4}$ into the mean-to-max argument~\cites[Remark~1.4]{BalkanovaChatzakosCherubiniFrolenkovLaaksonen2019} produces the complementary, albeit weaker, pointwise bound
\begin{equation}\label{eq:complementary-unconditional}
\mathcal{E}_{\Gamma}(x) \ll_{\epsilon} x^{\frac{3}{2}+\frac{\vartheta}{3}+\epsilon},
\end{equation}
yielding an unconditional exponent of $\frac{14}{9}+\epsilon$ upon substituting the current record $\vartheta = \frac{1}{6}$.

A meticulous reconsideration of the machinery developed in the prequel~\cites{Kaneko2024} facilitates an enhancement of \cref{thm:unconditional}. The core of our argument, though initially contingent~upon progress towards the hybrid subconvexity problem for quadratic Dirichlet $L$-functions over $\mathbb{Q}(i)$, culminates in the resolution of the optimisation problem in a manner that renders the result independent of the subconvex exponent in the archimedean aspect. As a $3$-dimensional incarnation of~\cites[Theorem~1.1]{Kaneko2024}, we now present our second result as follows.
\begin{theorem}\label{thm:conditional}
Keep the notation and assumptions as in \cref{thm:unconditional}. Then we have for~any $\epsilon > 0$ that
\begin{equation}\label{eq:conditional}
\mathcal{E}_{\Gamma}(x) \ll_{\epsilon} x^{\frac{245}{172}+\frac{15\vartheta}{43}+\epsilon}.
\end{equation}
In particular, the current record $\vartheta = \frac{1}{6}$ unconditionally yields
\begin{equation}
\mathcal{E}_{\Gamma}(x) \ll_{\epsilon} x^{\frac{255}{172}+\epsilon},
\end{equation}
while the generalised Lindel\"{o}f hypothesis conditionally yields
\begin{equation}
\mathcal{E}_{\Gamma}(x) \ll_{\epsilon} x^{\frac{245}{172}+\epsilon},
\end{equation}
where $\frac{255}{172} = 1.48256 \cdots$ and $\frac{245}{172} = 1.42442 \cdots$.
\end{theorem}

\begin{figure}
\centering
\begin{AnnotatedImage}[width=0.8]{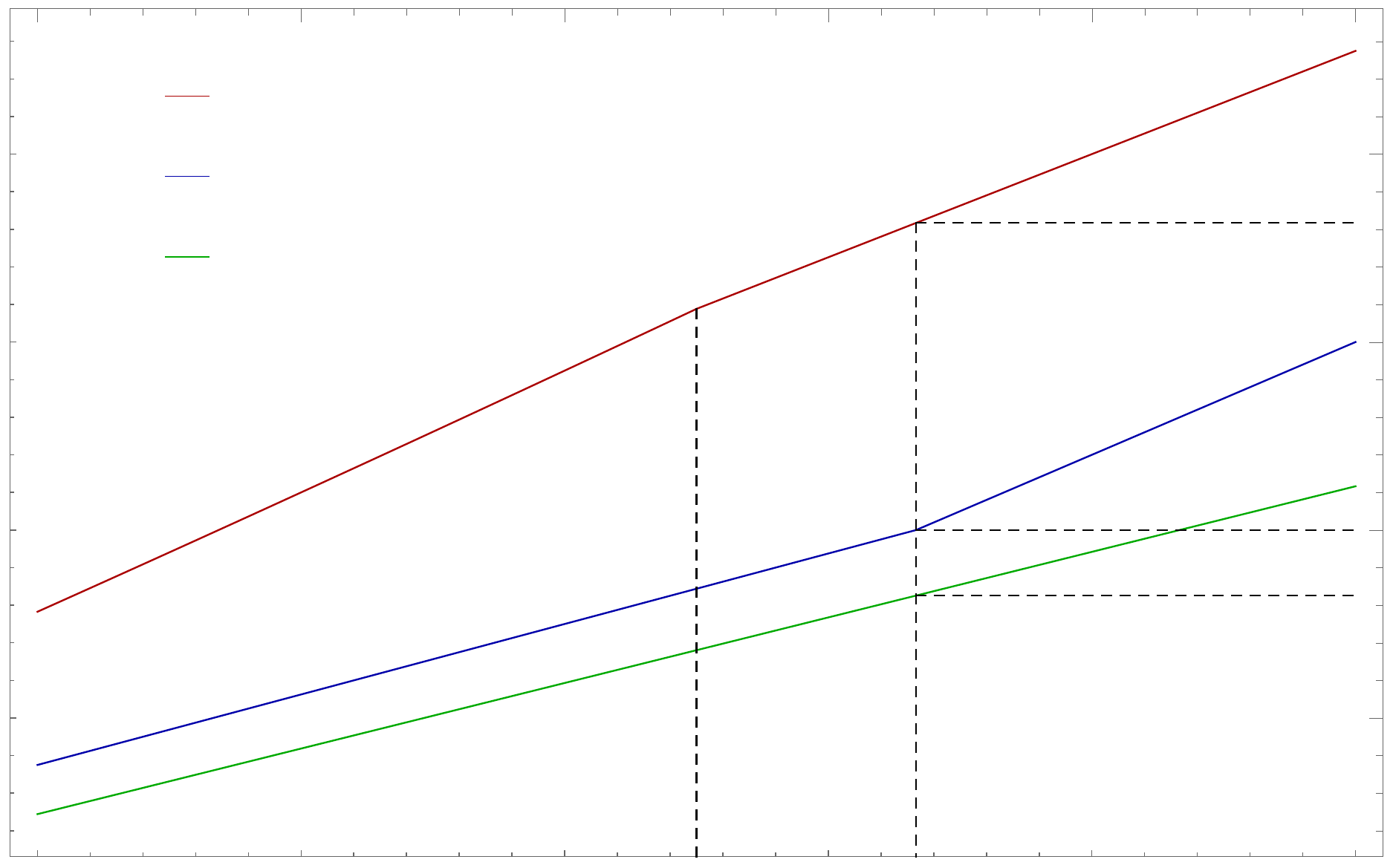}
	\annotate [draw=none,font=\normalsize] at (-0.015,0.2949){\color{black}{$\frac{34}{23}$}};
	\annotate [draw=none,font=\normalsize] at (-0.015,0.13684){\color{black}{$\frac{23}{16}$}};
	\annotate [draw=none,font=\normalsize] at (-0.015,0.0613){\color{black}{$\frac{245}{172}$}};
	\annotate [draw=none,font=\normalsize] at (0.0275,-0.03){\color{black}{0}};
	\annotate [draw=none,font=\normalsize] at (0.50039,-0.03){\color{black}{$\frac{1}{8}$}};
	\annotate [draw=none,font=\normalsize] at (0.65812,-0.03){\color{black}{$\frac{1}{6}$}};
	\annotate [draw=none,font=\normalsize] at (0.97435,-0.03){\color{black}{$\frac{1}{4}$}};
	\annotate [draw=none,font=\normalsize] at (1.015,0.31403){\color{black}{$\frac{255}{172}$}};
	\annotate [draw=none,font=\normalsize] at (1.015,0.38957){\color{black}{$\frac{3}{2}$}};
	\annotate [draw=none,font=\normalsize] at (1.015,0.74345){\color{black}{$\frac{242}{153}$}};
	\annotate [draw=none,font=\normalsize] at (0.2,0.88961){\color{black}{$\delta_{0}(\vartheta)$}};
	\annotate [draw=none,font=\normalsize] at (0.2,0.79691){\color{black}{$\delta_{1}(\vartheta)$}};
	\annotate [draw=none,font=\normalsize] at (0.2,0.70419){\color{black}{$\delta_{2}(\vartheta)$}};
\end{AnnotatedImage}
\caption{A comparison of~\eqref{eq:Balkanova-Frolenkov},~\eqref{eq:unconditional}, and~\eqref{eq:conditional} as $\vartheta \in [0, \frac{1}{4})$ varies}
\label{fig}
\end{figure}

\cref{thm:conditional} marks the first instance in which a long-standing barrier of $\frac{3}{2}+\epsilon$, traditionally thought of as a limitation of current methodologies owing to the abundance of eigenvalues by Weyl's law~\cites[Theorem~8.9.1]{ElstrodtGrunewaldMennicke1998}, is surpassed unconditionally. Moreover, the conditional exponent of $\frac{245}{172}+\epsilon$ is stronger than the previously best known conditional exponent of $\frac{10}{7}+\epsilon$ as in the last row of \cref{table}. For the purpose of visualising the extent to which \cref{thm:conditional} attains a polynomial power-saving improvement with respect to $\vartheta \in [0, \frac{1}{4})$, we define
\begin{equation}\label{eq:delta-1-theta}
\delta_{1}(\vartheta) \coloneqq \max \left(\frac{23+6\vartheta}{16}, \frac{7+3\vartheta}{5} \right) = 
	\begin{dcases}
	\tfrac{23+6\vartheta}{16} & \text{if $\vartheta \in [0, \tfrac{1}{6})$},\\
	\tfrac{7+3\vartheta}{5} & \text{if $\vartheta \in [\tfrac{1}{6}, \tfrac{1}{4})$},\\
	\end{dcases}
\qquad \delta_{2}(\vartheta) = \frac{245}{172}+\frac{15\vartheta}{43}.
\end{equation}
\cref{fig} offers a comparative manifestation of the quality of our results in juxtaposition with the prior record~\eqref{eq:Balkanova-Frolenkov}, thereby clarifying the shrinkage of our exponents as $\vartheta$ approaches $0$.

\begin{remark}
A breakthrough by Petrow and Young~\cites{PetrowYoung2020}{PetrowYoung2023} establishes hybrid Weyl-strength subconvex bounds for all Dirichlet $L$-functions with no restrictions on the conductor. An attempt to extend their machinery to number fields is made in the work of Nelson~\cites[Theorem~1.1]{Nelson2020}, although the archimedean subconvex exponent in~\eqref{eq:conductor-aspect} remains unspecified. A further extension in the hybrid aspect is explored by Balkanova, Frolenkov, and Wu~\cites[Theorem~1.1]{BalkanovaFrolenkovWu2024}, although their method is tailored to Gr\"{o}{\ss}encharaktern over totally real number fields of cubefree level. While likely plausible given the state-of-the-art technology, no hybrid result at or beyond the Weyl threshold for totally imaginary number fields has yet appeared in the literature; cf.~\cites[Theorem]{Sohne1997}[Theorem~0.3.4]{Wu2013}[Theorem~1.1]{Wu2019}.
\end{remark}

\subsection{Overview of the manoeuvres}
This section provides a comprehensive interpretation of \cref{thm:unconditional,thm:second-moment,thm:conditional} to highlight the principal innovations. Since the fundamental toolbox for establishing the first two claims is subsumed within the proof of \cref{thm:conditional}, our exposition naturally centres on the latter. Among other inputs, the core of our methodologies consists~of
\begin{enumerate}
\item\label{enum:1} the complete resolution of Koyama's mean Lindel\"{o}f hypothesis (\cref{thm:mean-Lindelof});
\item\label{enum:2} an improved Brun--Titchmarsh-type theorem over short intervals (\cref{thm:Brun-Titchmarsh});
\item\label{enum:3} a bootstrapped multiplicative exponent pair in the limiting regime (\cref{prop:exponent-pair});
\item\label{enum:4} a symplectic zero density theorem for quadratic characters over $\mathbb{Q}(i)$ (\cref{prop:density});
\item\label{enum:5} an interpolation with Koyama's result on spectral exponential sums (\cref{prop:key}).
\end{enumerate}

The optimal large sieve inequality is not known over $\Gamma \backslash \mathbb{H}^{3}$ in contrast to the $2$-dimensional background, unless the spherical family of Maa{\ss} cusp forms is extended to the nonspherical family of all cuspidal automorphic representations over $\Gamma \backslash \mathrm{PSL}_{2}(\mathbb{C})$; see~\cites[Theorem~1.1]{Qi2024-3}[Equation~(1.6)]{Qi2024}. This observation precludes the application of the existing techniques \`{a} la Luo and Sarnak~\cites[Section~3]{LuoSarnak1995}, thereby giving rise to suboptimal attempts towards the mean Lindel\"{o}f hypothesis when viewed from the perspective of the prime geodesic theorem. Our argument for~\eqref{enum:1} is based exclusively on the optimal spherical mean square asymptotics for the Fourier coefficients of Maa{\ss} cusp forms on $\Gamma \backslash \mathbb{H}^{3}$, which, in conjunction with Landau's trick~\cites[Section~3]{Landau1915}, facilitates a straightforward proof of the mean Lindel\"{o}f hypothesis. This approach obviates the need for any spectral large sieve inequality -- itself a consequence of the Kuznetsov formula -- and instead relies on the pre-Kuznetsov formula. The versatility of our methods holds potential for broader applicability to any Kleinian group $\Gamma \subset \mathrm{PSL}_{2}(\mathbb{C})$ via the theory of rapidly decreasing Poincar\'{e} series; cf. the first display on~\cites[Page~569]{Reznikov1993}.

A Brun--Titchmarsh-type theorem over short intervals~\eqref{enum:2} in the $3$-dimensional setting was first investigated by Balog et al.~\cites[Theorem~1.3]{BalogBiroCherubiniLaaksonen2022}. The proof makes crucial use of the $2$-dimensional Poisson summation formula, or the Vorono\u{\i} summation formula for the number of representations of $n$ as the sum of two squares, and the Weil--Gundlach bound for Gaussian Kloosterman sums. Since the latter fails to detect cancellations among sums of Kloosterman sums, we employ $2$-dimensional Poisson summation backwards across complementary regimes and asymptotically manipulate the Bessel function of the first kind,~thereby circumventing a potential tautology. Interpolating with~\cites[Equation~(23)]{BalogBiroCherubiniLaaksonen2022} then leads to a substantial refinement, with the aid of the additional lattice structure in Gaussian integers.

In harmony with the case of the modular orbifold, the proof features subconvex bounds for quadratic character sums over Gaussian integers~\eqref{enum:3}, introducing the notion of multiplicative (and dual) exponent pairs over $\mathbb{Q}(i)$. We emphasise that the seminal work of Burgess~\cites{Burgess1962-2}{Burgess1963}{Burgess1986} depends heavily on the Weil bound and hence the periodicity of the summand $n \mapsto \chi(n)$. However, its direct algebraic generalisation, which sums over integral ideals with norms lying in a fixed interval, results in the loss of periodicity; see~\cites{Hinz1983}{Hinz1983-2}{Hinz1986}. On the other hand, our argument utilises Mellin inversion to express quadratic character sums in terms of quadratic Dirichlet $L$-functions, enabling the substitution of hybrid subconvexity. This accounts for the appearance of the subconvex exponent $\vartheta \in [0, \frac{1}{4})$ in \cref{thm:conditional}.

Another crucial element of our argument is a zero density theorem for the symplectic~family of quadratic characters~\eqref{enum:4}. The best known unconditional exponent of Balog et al.~\cites[Lemma~2.5]{BalogBiroCherubiniLaaksonen2022} is derived from~\cites[Theorem~2]{Huxley1971}, where the summation is taken over moduli up to a given point, rather than over fundamental discriminants. Nonetheless, the restriction to fundamental discriminants by positivity in~\cites[Equation~(32)]{BalogBiroCherubiniLaaksonen2022} sacrifices the intrinsic nature of quadratic characters, namely the transition from the unitary family of all Dirichlet characters to the symplectic family of quadratic characters considerably worsens the quality of the resulting zero density theorem. To eschew this setback, we employ Onodera's~quadratic large sieve~\cites[Theorem~1]{Onodera2009} and adapt Heath-Brown's deduction~\cites[Section~11]{HeathBrown1995} of a zero density theorem over $\mathbb{Q}$, thereby yielding a result of comparable quality over $\mathbb{Q}(i)$.

Furthermore, an asymptotic form of the stationary phase analysis enables an enhancement of the techniques developed by Koyama~\cites[Section~4]{Koyama2001}, taking into account an additional oscillatory factor in an integral transform involving Bessel functions. When combined with an arithmetic expression of Kloosterman sums (\cref{lem:counting-function}), the technology of Iwaniec~\cites[Section~10]{Iwaniec1984}, Cai~\cites[Section~6]{Cai2002}, and the author~\cites[Section~6]{Kaneko2024} reduces the problem to the estimation of quadratic character sums of a special type over $\mathbb{Q}(i)$, as addressed in~\eqref{enum:3}. To facilitate the utilisation of partial summation, we interpolate our result with the penultimate display on~\cites[Page~792]{Koyama2001}, followed by the resolution of an intricate optimisation problem.

\subsection{Organisation of the paper}
In \cref{sect:miscellaneous}, we shall configure a miscellaneous toolbox -- encompassing both analytic and algebraic perspectives -- which merits a brief review by any reader who seeks to familiarise themselves with the basics. The arithmetic framework for the proofs of \cref{thm:unconditional,thm:second-moment,thm:conditional} is laid out in \cref{sect:Iwaniec--Bykovskii,sect:quadratic,sect:zero,sect:Zagier,sect:bilinear}, while the automorphic framework is laid out in \cref{sect:mean,sect:spectral}. The concluding step of the argument is presented in \cref{sect:proof}. Each section is designed to facilitate an effective and constructive methodological comparison with the $2$-dimensional setting, as the present paper serves as a sequel to~\cites{Kaneko2024}.

\section{Preliminaries}\label{sect:miscellaneous}
This section compiles conventional notation and catalogues several foundational facts and lemmata utilised in due course. A knowledgeable reader may opt to briefly peruse or skip~this section upon first reading.

\subsection{Asymptotic notation}\label{subsect:notation}
Throughout this paper, the Vinogradov asymptotic notation, denoted by $\ll$ and $\gg$, is employed synonymously with the Bachmann--Landau notation. The Hardy--Littlewood notation $f = \Omega_{\pm}(g)$ denotes both $\limsup f/g > 0$ and $\liminf f/g < 0$. We write $f \asymp g$ to indicate the existence of absolute constants $c_{1}, c_{2} > 0$ such that $c_{1} g < f < c_{2} g$, and we write $f \sim g$ to denote $f = (1+o(1))g$. Dependence on a parameter is indicated by a subscript, namely the notation $f \ll_{\xi} g$ or~$f = O_{\xi}(g)$ signifies the existence of an effectively computable constant $c = c(\xi) > 0$, depending at most on $\xi$, such that $|f(z)| \leq c|g(z)|$ for all $z$ within a range that is clear from context. In the absence of a subscript $\xi$, the constant $c$ is said to be absolute. To maintain rigour and completeness, we adopt the convention~that the numbers $c_{1}, c_{2}, c_{3}, \ldots$ form a sequence of certain positive, absolute, and effectively computable constants. Moreover, the symbol $\epsilon > 0$ represents an arbitrarily small positive quantity that may vary in each instance; it is therefore permissible to write $x^{2\epsilon} \ll x^{\epsilon}$ without reservation.

\subsection{Analytic notation}\label{subsect:stationary}
The notation $e(z) \coloneqq \exp(2\pi iz)$ denotes the complex exponential; see the first paragraph of \cref{subsect:Kloosterman-sums} for its algebraic counterpart. As usual, the notation~$\int_{(\sigma)}$ denotes a complex contour integral over the vertical line with real part $\sigma \in \mathbb{R}$. It is convenient to invoke Stirling's formula in a sufficiently crude form~\cites[Equation~(8.327.1)]{GradshteynRyzhik2014}
\begin{equation}\label{eq:Stirling}
\Gamma(s) = \sqrt{2\pi} s^{s-\frac{1}{2}} e^{-s} \Big(1+O \Big(\frac{1}{|s|} \Big) \Big)
\end{equation}
for $|\arg s| < \pi$, which plays a foundational role in subsequent asymptotic evaluations.

\subsection{Arithmetic notation}
The letter $p$ is reserved exclusively to signify a prime element, unless otherwise specified. The notation $\varphi$ is employed to denote Euler's totient function over $\mathbb{Z}[i]$, while $\mu$ denotes the M\"{o}bius function over $\mathbb{Z}[i]$ defined in terms of the prime factorisations of ideals. Let $r(n)$ denote the number of representations of $n$ as a sum of two squares, as in the formulation of the Gau{\ss} circle problem. For two (rational or Gaussian) integers $a$ and $b$, their greatest common divisor is denoted by $(a, b) = \mathrm{gcd}(a, b)$. The notation $\lfloor x \rfloor$ denotes the floor function. To circumvent any potential confusion, we denote the complex~conjugation by $\overline{x}$ and the multiplicative inverse (with respect to an appropriate modulus) by $x^{\ast}$. The symbol $\mathbf{1}$ denotes the indicator function of a given statement. For example, $\mathbf{1}_{\text{S}}$ evaluates to $1$ if the statement S is true and to $0$ otherwise. The letter $j \in \mathbb{N}_{0}$ is employed as an auxiliary index to systematically arrange the quantities of interest according to their order~of preference.

\subsection{Algebraic notation}
One may consult the book of Lemmermeyer~\cites{Lemmermeyer2021} for general theoretical foundations. We employ Gothic letters $\mathfrak{a}, \mathfrak{b}, \ldots$ to denote nonzero fractional ideals of $\mathbb{Q}(i)$, while we reserve $\mathfrak{n}$ and $\mathfrak{d}$ especially for nonzero integral ideals of $\mathbb{Q}(i)$. Furthermore, let $\mathrm{N}(n) = |n|^{2}$ and $\mathrm{N}(\mathfrak{n}) = |\mathbb{Z}[i]/\mathfrak{n}|$ denote the norm forms of $n$ and $\mathfrak{n}$, respectively. It~is~our convention to pass freely between ideals and their generators. Recall that the ring of integers of $\mathbb{Q}$ is $\mathbb{Z}[i]$ and that the class number of $\mathbb{Q}(i)$ is one, meaning that every ideal is principal. If $\mathfrak{n}$ is a nonzero ideal of $\mathbb{Z}[i]$, then a Dirichlet character modulo $\mathfrak{n}$ is a group homomorphism $\chi_{\mathfrak{n}}: \mathrm{Cl}^{\mathfrak{n}} \to S^{1}$, where $\mathrm{Cl}^{\mathfrak{n}}$ stands for the narrow ray class group modulo $\mathfrak{n}$ and serves as an alternative to the reduced residue class $(\mathbb{Z}/n \mathbb{Z})^{\times}$ having the base field $\mathbb{Q}$ and $\mathfrak{n} = n \mathbb{Z}$. Given $0 \neq \mathfrak{d}, \mathfrak{n} \unlhd \mathbb{Z}[i]$, the notation $\mathfrak{d} \mid \mathfrak{n}$ implies the existence of $\mathfrak{a} \unlhd \mathbb{Z}[i]$ such that $\mathfrak{n} = \mathfrak{a} \mathfrak{d}$. Similarly, given $0 \neq d, n \in \mathbb{Z}[i]$, the notation $d \mid n$ implies $(d) \mid (n)$.

\subsection{Characters and \texorpdfstring{$L$}{}-functions}\label{subsect:abuse}
The Dedekind zeta function over $\mathbb{Q}(i)$ is defined by
\begin{equation}\label{eq:Dedekind}
\zeta_{\mathbb{Q}(i)}(s) \coloneqq \sum_{0 \ne n \in \mathbb{Z}[i]} \frac{1}{\mathrm{N}(n)^{s}}, \qquad \Re(s) > 1,
\end{equation}
which converges absolutely for $\Re(s) > 1$ and admits a meromorphic continuation to $\mathbb{C}$ except for one simple pole at $s = 1$, with residue given in terms of the analytic class number formula. Following~\cites[Section~2.1]{BalogBiroCherubiniLaaksonen2022}, we abuse notation by replacing ideals with their generators, thereby allowing our summations to range over the elements of $\mathbb{Z}[i]$. However, this approach may introduce a minor inexactitude, as each ideal corresponds to four generators. To resolve this discrepancy, one could either attach a factor of $\frac{1}{4}$ to summations over the elements of $\mathbb{Z}[i]$ or specify a choice of a generator for each ideal to restrict consideration to appropriate subsets of $\mathbb{Z}[i]$, such as the first quadrant. We refrain from adopting either~approach, confident that the reader will nonetheless be able to follow the remainder of the paper without ambiguity.

We are interested in the family of quadratic characters associated to the generalised Jacobi symbol~\cites[Section~14.2]{IrelandRosen1990}\footnote{The notation differs from the first display on~\cites[Page~1899]{BalogBiroCherubiniLaaksonen2022} in that the positions of the numerator and denominator are interchanged. The choice of convention for any algebraic object is ultimately contingent upon the preferences of the author(s). Nonetheless, through minor rearrangements -- including the application of quadratic reciprocity -- differences of such sort neither compromise the rigour of the ensuing computations nor detract from the quality of our results.}
\begin{equation}
\chi_{d}(n) \coloneqq \Big(\frac{d}{n} \Big),
\end{equation}
with $0 \ne d, n \in \mathbb{Z}[i]$. If $D \in \mathbb{Z}[i]$ is a generator of the fundamental discriminant of a quadratic extension of $\mathbb{Q}(i)$, then $\chi_{D}$ is a primitive Dirichlet character modulo the ideal $\mathfrak{d} = (D) \unlhd \mathbb{Z}[i]$. In conjunction with the rational case, summing over nonzero Gaussian integers $n \in \mathbb{Z}[i]$ gives rise to a quadratic Dirichlet $L$-function over $\mathbb{Q}(i)$ associated to $\chi_{D}$, namely
\begin{equation}
L(s, \chi_{D}) \coloneqq \sum_{0 \ne n \in \mathbb{Z}[i]} \frac{\chi_{D}(n)}{\mathrm{N}(n)^{s}}, \qquad \Re(s) > 1.
\end{equation}
which converges absolutely for $\Re(s) > 1$ and extends to an entire function with a functional equation relating the values at $s$ and $1-s$. Then the generalised Riemann hypothesis states that all the nontrivial zeros of $L(s, \chi_{D})$ would lie on the critical line $\Re(s) = \frac{1}{2}$. As a weaker version thereof, the generalised Lindel\"{o}f hypothesis predicts that $\vartheta = 0$ would hold in~\eqref{eq:conductor-aspect}.

\subsection{Gaussian Zagier \texorpdfstring{$L$}{}-series}\label{subsect:Zagier}
As a precursor, Zagier~\cites[Equation~(6)]{Zagier1977} investigated an $L$-function of an indefinite binary quadratic form over $\mathbb{Z}$, encoding arithmetic information of the underlying quadratic extension of $\mathbb{Q}$. If we write $\delta = n^{2}-4$ for $0 \ne n \in \mathbb{Z}[i]$, then $\delta$ is a discriminant of an indefinite binary quadratic form over $\mathbb{Z}[i]$, to which we attach~the~Gaussian Zagier $L$-series~\cites[Equation~(3.5)]{Szmidt1983}\footnote{The interested reader is directed to~\cites[Section~2.9]{Szmidt1987} for an exhaustive consideration applicable to~any imaginary quadratic field of class number one, although this source is difficult to access online.}
\begin{equation}\label{eq:Zagier}
L(s, \delta) \coloneqq \frac{\zeta_{\mathbb{Q}(i)}(2s)}{\zeta_{\mathbb{Q}(i)}(s)} \sum_{c \neq 0} \frac{\rho_{c}(\delta)}{\mathrm{N}(c)^{s}} = \sum_{c \neq 0} \frac{\lambda_{c}(\delta)}{\mathrm{N}(c)^{s}}, \qquad \Re(s) > 1,
\end{equation}
where the coefficients are defined via M\"{o}bius inversion by
\begin{equation}\label{eq:rho}
\rho_{c}(\delta) \coloneqq \#\{x \tpmod{2c}: x^{2} \equiv \delta \tpmod{4c} \}
\end{equation}
and
\begin{equation}\label{eq:rho-lambda}
\lambda_{c}(\delta) \coloneqq \sum_{c_{1}^{2} c_{2} c_{3} = c} \mu(c_{2}) \rho_{c_{3}}(\delta), \qquad \rho_{c}(\delta) = \sum_{c_{1} c_{2} = c} \mu(c_{2})^{2} \lambda_{c_{1}}(\delta).
\end{equation}
It is straightforward to verify that $\rho_{c}(\delta)$ and $\lambda_{c}(\delta)$ are both multiplicative functions in $c$ for $\delta$ fixed. If $\delta \sim D \ell^{2}$ are associates, namely they are equal up to multiplication by a unit in $\mathbb{Z}[i]$; if $p^{r}$ denotes the exact power of a prime element $p \in \mathbb{Z}[i]$ dividing $\ell$; and if $a \coloneqq \min(\lfloor \frac{k}{2} \rfloor, r)$, then (cf.~\cites[Equation~(2)]{SoundararajanYoung2013})
\begin{equation}\label{eq:prime-powers}
\lambda_{p^{k}}(\delta) = \mathrm{N}(p)^{a} \chi_{\delta p^{-2a}}(p^{k-2a}),
\end{equation}
from which it follows uniformly in $\delta \in \mathbb{Z}[i]$ that
\begin{equation}\label{eq:uniformly}
\rho_{c}(\delta), \lambda_{c}(\delta) \ll_{\epsilon} \mathrm{N}(c)^{\epsilon} \max_{c_{\star}^{2} \mid c} \mathrm{N}(c_{\star}).
\end{equation}
Since the right-hand side of~\eqref{eq:uniformly} is bounded by $\ll_{\epsilon} \mathrm{N}(c)^{\epsilon}$ on average, the Zagier $L$-series~\eqref{eq:Zagier} is absolutely convergent for $\Re(s) > 1$ and extends meromorphically to an entire function~over $\mathbb{C}$ with at most a simple pole at $s = 1$. Notably, up to multiplication by a certain well-behaved Dirichlet polynomial, the Zagier $L$-series $L(s, \delta)$ aligns with a quadratic Dirichlet $L$-function $L(s, \chi_{D})$ over $\mathbb{Q}(i)$. Given $0 \ne D, \ell \in \mathbb{Z}[i]$, we define
\begin{equation}\label{eq:Dirichlet-polynomial}
T_{\ell}^{(D)}(s) \coloneqq \sum_{\ell_{1} \ell_{2} = \ell} \frac{\mu(\ell_{1}) \chi_{D}(\ell_{1}) \sigma_{1-2s}(\ell_{2})}{\mathrm{N}(\ell_{1})^{s}}, \qquad \sigma_{\xi}(n) \coloneqq \sum_{d \mid n} \mathrm{N}(d)^{\xi}.
\end{equation}
Szmidt~\cites[Proposition 6]{Szmidt1983} established the following identity.
\begin{lemma}\label{lem:Szmidt-1}
Keep the notation as above. Let $\delta$ be a discriminant of a binary quadratic form over $\mathbb{Z}[i]$, and let $\delta \sim D \ell ^{2}$, where $D$ is a generator of the fundamental discriminant of the quadratic extension $\mathbb{Q}(i)(\sqrt{\delta})$. Then
\begin{equation}
L(s, \delta) = T^{(D)}_{\ell}(s) L(s, \chi_{D}).
\end{equation}
\end{lemma}

A variation of~\cites[Lemma~2.1]{SoundararajanYoung2013} over $\mathbb{Q}(i)$ implies that all the nontrivial zeros of $T_{\ell}^{(D)}(s)$ lie on the critical line $\Re(s) = \frac{1}{2}$, and hence the generalised Riemann hypothesis for $L(s, \delta)$ is equivalent to that for $L(s, \chi_{D})$.
\begin{lemma}\label{lem:Szmidt-2}
The generalised Lindel\"{o}f hypothesis for $L(s, \delta)$ is equivalent to $\vartheta = 0$.
\end{lemma}

\begin{proof}
The argument adapts that of~\cites[Lemma~4.2]{BalkanovaFrolenkov2018} with minor adjustments.
\end{proof}

\subsection{Gaussian Kloosterman sums}\label{subsect:Kloosterman-sums}
Let $\mathrm{Tr}_{\mathbb{C}/\mathbb{R}}(z) \coloneqq z+\overline{z}$ denote the standard trace map. To maintain consistency with their treatment~over $\mathbb{Q}$, one may introduce the modern~notation for exponentials, namely for $z \in \mathbb{C}$,
\begin{equation}\label{eq:abuse}
\check{e}(z) \coloneqq e(\Re(z)) = \exp(\pi i \, \mathrm{Tr}_{\mathbb{C}/\mathbb{R}}(z)),
\end{equation}
which, while possibly different from the conventional wisdom in the literature, does not affect the essence of the subsequent discussion; the choice of adopting either the inner product or the trace map in its definition is ultimately contingent upon the preferences of the author(s). Accordingly, it is convenient to abbreviate $\check{e}_{c}(z) \coloneqq \check{e}(\frac{z}{c})$ for $z \in \mathbb{C}$ and $0 \ne c \in \mathbb{Z}[i]$.

Given $m, n, c \in \mathbb{Z}[i]$ with $c \ne 0$, we now define the Gaussian Kloosterman sum by~\cites[Page~268]{Motohashi1997}[Equation~(3.2)]{Sarnak1983}
\begin{equation}\label{eq:Gaussian-Kloosterman}
S(m, n, c) \coloneqq \sum_{a \in (\mathbb{Z}[i]/(c))^{\times}} \check{e}_{c}(ma+na^{\ast}),
\end{equation}
where $a^{\ast}$ denotes the multiplicative inverse of $a$ modulo the ideal $(c)$, namely $aa^{\ast} \equiv 1 \tpmod c$. The Gaussian Kloosterman sum obeys the Weil--Gundlach bound~\cites[Section~4]{Gundlach1954}[Equation~(3.5)]{Motohashi1997} as follows.
\begin{lemma}
Keep the notation as above. Then
\begin{equation}\label{eq:Weil}
|S(m, n, c)| \leq |(m, n, c)| \tau(c) \mathrm{N}(c)^{\frac{1}{2}},
\end{equation}
where $\tau(n) \coloneqq \sigma_{0}(n)$ denotes the number of divisors of $0 \ne n \in \mathbb{Z}[i]$. Furthermore, the~greatest common divisor is essentially bounded on average over $0 \ne c \in \mathbb{Z}[i]$, since
\begin{equation}\label{eq:bounded-on-average}
\sum_{\mathrm{N}(c) \leq x} |(m, n, c)| \ll_{\epsilon} |mn|^{\epsilon} x^{1+\epsilon}.
\end{equation}
\end{lemma}

Following Iwaniec~\cites[Page~140]{Iwaniec1984}, it is also convenient to introduce the notation
\begin{equation}\label{eq:rho-c-a}
\rho(c, a) \coloneqq \#\{d \tpmod{c}: d^{2}-ad+1 \equiv 0 \tpmod{c} \}.
\end{equation}
If $c \sim k \ell$ with $k$ a squarefree Gaussian integer and $4\ell$ a squareful Gaussian integer coprime to $k$, then the multiplicativity relation for $\rho(c, a)$ with respect to $c$ asserts
\begin{equation}\label{eq:multiplicativity}
\rho(c, a) = \rho(k, a) \rho(\ell, a),
\end{equation}
where
\begin{equation}\label{eq:incongruent}
\rho(k, a) = \prod_{p \mid k} \left(1+\left(\frac{a^{2}-4}{p} \right) \right) = \sum_{r \mid k} \left(\frac{a^{2}-4}{r} \right).
\end{equation}
Hence, the counting function $\rho(c, a)$ is consistent with $\rho_{c}(a^{2}-4)$ as defined in~\eqref{eq:rho}.
\begin{lemma}\label{lem:identify}
Keep the notation as above. Then $\rho(c, a) = \rho_{c}(a^{2}-4)$.
\end{lemma}

\begin{proof}
The assertion follows directly from the existence of a bijective correspondence between the solutions $x \tpmod{2q}$ of
$x^{2} \equiv n^{2}-4 \tpmod{4q}$ and the solutions $y \tpmod{q}$ of $y^{2}+yn+1 \equiv 0 \tpmod{q}$, by writing $x = 2y+n$. In particular, any such $y$ must be coprime to $q$.
\end{proof}

\begin{lemma}\label{lem:rho-phi}
Keep the notation as above. Then
\begin{equation}
\sum_{a \tpmod{c}} \rho(c, a) = \varphi(c).
\end{equation}
\end{lemma}

\begin{proof}
The claim follows from the fact that a solution of the congruence in~\eqref{eq:rho-c-a} exists only when $a \in (\mathbb{Z}[i]/(c))^{\times}$.
\end{proof}

It is convenient to record the following expression \`{a} la Iwaniec~\cites[Equation~(14)]{Iwaniec1984} for the Gaussian Kloosterman sum as a sum of additive characters weighted by the multiplicative counting function $\rho(c, a)$, in the case where $m \sim n$ are associates.
\begin{lemma}\label{lem:counting-function}
Keep the notation as above. Then
\begin{equation}
S(n, n, c) = \sum_{a \tpmod{c}} \rho(c, a) \check{e}_{c}(an).
\end{equation}
\end{lemma}

\begin{proof}
The claim follows from substituting~\eqref{eq:rho-c-a} and making a change of variables.
\end{proof}

\cref{lem:counting-function} is instrumental in addressing the arithmetic side of the Kuznetsov formula over $\mathbb{Q}(i)$. A parallel phenomenon emerges in the $2$-dimensional setting, where the arithmetic and automorphic methods are concurrently refined to their utmost extent by the~author~\cites{Kaneko2024}. Nonetheless, the Kloosterman summation formula -- interpreted as an inverted version of the Kuznetsov formula -- remains unused in any prior pursuit towards the prime geodesic~theorem over $\mathbb{Q}(i)$ owing to certain additional complexities~\cites[Theorem~4]{BruggemanMotohashi2001-2}[Theorem~13.1]{BruggemanMotohashi2003}[Section~10]{Motohashi2001}. In fact, there exists a dilemma between the spherical Kuznetsov formula and the nonspherical Kloosterman summation formula, which will be revisited elsewhere; see also a version due to Miatello and Wallach~\cites[Theorem~2.2]{MiatelloWallach1990}, albeit with the absence~of sufficiently explicit information concerning the Bessel--Kuznetsov transforms on the spectral side.

Although dispensable for the core of the present paper, it would be advantageous for future endeavours to formulate an additively twisted Linnik--Selberg conjecture over $\mathbb{Q}(i)$, building upon the heuristic of Iwaniec~\cites[Page~139]{Iwaniec1984}[Page~189]{Iwaniec1984-2}[Conjecture~1]{Iwaniec1987-2} over $\mathbb{Q}$, as well as its justification by Balkanova, Frolenkov, and Risager~\cites[Proposition~1.4]{BalkanovaFrolenkovRisager2022}.
\begin{conjecture}\label{conj:twisted-Linnik-Selberg}
For any $C, D \geq 1$, $0 \ne n \in \mathbb{Z}[i]$, and $\epsilon > 0$ that
\begin{equation}
\sum_{\mathrm{N}(c) \leq C} \frac{S(n, n, c)}{\mathrm{N}(c)} e \left(\frac{D}{|c|} \right) \ll_{\epsilon} (|n| CD)^{\epsilon}.
\end{equation}
\end{conjecture}

For $D = 0$, \cref{conj:twisted-Linnik-Selberg} constitutes an extension of the Linnik--Selberg-type conjecture over $\mathbb{Q}(i)$, which ought to be sharp enough to imply the generalised Ramanujan conjectures over $\mathbb{Q}(i)$ for both the finite and infinite places; see~\cites[Section~2]{SarnakTsimerman2009} for discussions over $\mathbb{Q}$. Moreover, an argument akin to~\cites[Section~5]{BalkanovaFrolenkovRisager2022} demonstrates that \cref{conj:twisted-Linnik-Selberg} implies the square-root cancellation in the $3$-dimensional spectral exponential sum, namely
\begin{equation}\label{eq:spectral-exponential-sum}
\sum_{t_{j} \leq T} X^{it_{j}} \ll_{\epsilon} T^{\frac{3}{2}+\epsilon} X^{\epsilon},
\end{equation}
which fortifies our belief towards the optimal exponent $1+\epsilon$ in the prime geodesic theorem. Note that the trivial bound for the left-hand side of~\eqref{eq:spectral-exponential-sum} is $O(T^{3})$ by estimating everything trivially with Weyl's law~\cites[Theorem~8.9.1]{ElstrodtGrunewaldMennicke1998}.

\section{Mean Lindel\"{o}f hypothesis}\label{sect:mean}
The purpose of this section is to establish the mean Lindel\"{o}f hypothesis unconditionally.

\subsection{Prior breakthroughs}
Let $\mathbb{H}^{2} \coloneqq \{x+iy \in \mathbb{C}: y > 0 \}$ denote the $2$-dimensional upper half-plane, upon which the full modular group $\Gamma = \mathrm{PSL}_{2}(\mathbb{Z})$ acts discontinuously~via~M\"{o}bius transformations. Let $\mathcal{B}_{it_{j}}(\Gamma \backslash \mathbb{H}^{2})$ denote an orthonormal basis of Maa{\ss} cusp forms of weight $0$ and spectral parameter $t_{j}$. For $u_{j} \in \mathcal{B}_{it_{j}}(\Gamma \backslash \mathbb{H}^{2})$, the sequence $(u_{j})_{j \in \mathbb{N}}$ is ordered such~that the corresponding sequence of spectral parameters $(t_{j})_{j \in \mathbb{N}}$ increases monotonically. If $L(s, u_{j} \otimes u_{j})$ stands for the normalised Rankin--Selberg $L$-function associated to $u_{j}$, then Iwaniec~\cites[Equation~(13)]{Iwaniec1984}[Page~188]{Iwaniec1984-2} predicts the mean Lindel\"{o}f hypothesis over $\mathbb{Q}$, asserting~the existence of an absolute and effectively computable constant $\Cl[constant]{Iwaniec} > 0$ such that for any $\tau \in \mathbb{R}$ and $\epsilon > 0$,
\begin{equation}\label{eq:Iwaniec}
\sum_{t_{j} \leq T} \frac{1}{\cosh \pi t_{j}} \left|L \left(\frac{1}{2}+i\tau, u_{j} \otimes u_{j} \right) \right| \ll_{\epsilon} (1+|\tau|)^{\Cr{Iwaniec}} T^{2+\epsilon}.
\end{equation}
One may further consult~\cites[Section~7]{Iwaniec1984-2}, wherein the intricate relationship between~the prime geodesic theorem for $\Gamma \backslash \mathbb{H}^{2}$ and~\eqref{eq:Iwaniec} is elucidated with solid heuristic manifestations. To the best of our knowledge, the nomenclature of the mean Lindel\"{o}f hypothesis is attributed to Zelditch~\cites[Equation~(0.17)]{Zelditch1991}[Equation~(4.4)]{Zelditch1992}, who replaced the bound $O_{\epsilon}(T^{2+\epsilon})$ in~\eqref{eq:Iwaniec} with $o(T^{2})$; indeed, his avenue of approach yields the slightly stronger result $O(\frac{T^{2}}{\log T})$.

Nonetheless, some confusion may be traced back to Luo and Sarnak~\cites[Theorem~3.2]{LuoSarnak1995}, albeit with their appropriate usage of the terminology in~\cites[Equation~(5)]{LuoSarnak1995}. Notably, they examined the~\textit{second} moment of the Rankin--Selberg $L$-function on the critical line, provided that the sum over the spectrum traverses Maa{\ss} cusp forms that are joint eigenfunctions of all the Hecke operators, and resolved the mean Lindel\"{o}f hypothesis~\eqref{eq:Iwaniec} on the~\textit{first} moment via the application of the Cauchy--Schwarz inequality. Since the Laplace--Beltrami operator commutes with all the Hecke operators, every nonconstant Laplace eigenfunction on $\Gamma \backslash \mathbb{H}^{2}$ is expressed as a linear combination of Hecke--Maa{\ss} cusp forms. Hence, the Hecke assumption of Luo and Sarnak~\cites[Page~208]{LuoSarnak1995} ought to be superfluous, given that the cuspidal spectrum of level $1$ is conjectured to be simple by Cartier~\cites[Conjecture~(B)]{Cartier1971}.

\subsection{Koyama's formulation}
To preclude any potential misconceptions in the subsequent discussion, we clarify that the mean Lindel\"{o}f hypothesis pertains to the first moment on the left-hand side of~\eqref{eq:Iwaniec}. An indispensable ingredient for the proofs of \cref{thm:unconditional,thm:conditional} is the resolution of the mean Lindel\"{o}f hypothesis for the Rankin--Selberg $L$-function, which had remained elusive since its somewhat oblique formulation by Koyama~\cites[Equation~(3.2)]{Koyama2001}, thereby causing terminological confusion in due course; see also~\cites[Pages~58--59]{Laaksonen2019}.
\begin{conjecture}[Mean Lindel\"{o}f hypothesis]\label{conj:mean-Lindelof}
Let $\Gamma = \mathrm{PSL}_{2}(\mathbb{Z}[i])$. Then there exists an absolute and effectively computable constant $\Cl[constant]{Koyama} > 0$ such that for any $\tau \in \mathbb{R}$ and $\epsilon > 0$,
\begin{equation}\label{eq:mean-Lindelof}
\sum_{t_{j} \leq T} \frac{t_{j}}{\sinh \pi t_{j}} \left|L \left(\frac{1}{2}+i\tau, u_{j} \otimes u_{j} \right) \right| \ll_{\epsilon} (1+|\tau|)^{\Cr{Koyama}} T^{3+\epsilon}.
\end{equation}
\end{conjecture}

The main result of this section is as follows.
\begin{theorem}\label{thm:mean-Lindelof}
\cref{conj:mean-Lindelof} holds for any $\Cr{Koyama} > 3$.
\end{theorem}

The proof of \cref{thm:mean-Lindelof} does not rely on any of the following: (1) the multiplicativity of the Fourier coefficients, which simplifies the problem to estimating the second moment of the associated symmetric square $L$-function; (2) lower and upper bounds for harmonic weights \`{a} la Iwaniec~\cites[Theorem~2]{Iwaniec1990} and Hoffstein and Lockhart~\cites[Corollary~0.3]{HoffsteinLockhart1994}, respectively; or (3) the spectral large sieve inequality \`{a} la Deshouillers and Iwaniec~\cites[Theorem~2]{DeshouillersIwaniec1982}~-- itself a consequence of the Kuznetsov formula. Furthermore, the flexibility of our techniques ensures that the machinery extends naturally to the $2$-dimensional case, thereby providing a simpler proof of~\eqref{eq:Iwaniec} while eliminating the need for the Hecke assumption imposed by Luo and Sarnak~\cites[Page~208]{LuoSarnak1995}. Novelties of the proof include a subtle enhancement of the ideas of Iwaniec~\cites[Section~8]{Iwaniec1984} based on Landau's trick~\cites[Section~3]{Landau1915} and a simultaneous improvement of the results of Raghavan and Sengupta~\cites[Theorem]{RaghavanSengupta1994}[Theorem~1]{RaghavanSengupta1994-2}\footnote{For the reader's convenience, it may be beneficial to consult both nearly identical papers, as each contains several typesetting errors throughout.} and Motohashi~\cites[Lemma~11]{Motohashi1997}, thereby rendering~\cites[Corollary~10.1]{BruggemanMotohashi2003} effective in the spherical case. The most unbalanced scenario exhibited in the spherical case accounts~for suboptimal attempts towards \cref{conj:mean-Lindelof} in the literature, much of which, unlike the first moment that we shall address, falls within the framework of the spectral large sieve inequality.

\begin{remark}
The author became aware, subsequent to the completion of the present paper, that Matthes~\cites{Matthes1994}{Matthes1995}{Matthes1995-2}{Matthes1996} had partially refined Iwaniec's machinery to address the prime geodesic theorem twisted by the $3$-dimensional theta multiplier system. In this context, the Fourier coefficients of half-integral weight forms are not multiplicative~except at squares and correspond in turn to central values of quadratic twists of modular $L$-functions via Waldspurger's formula~\cites[Th\'{e}or\`{e}me~1]{Waldspurger1981}. Consequently, the mean Lindel\"{o}f hypothesis boils down essentially to the Lindel\"{o}f-on-average bound for a certain double Dirichlet series, and no explicit incarnation of Iwaniec's techniques appears to be known.
\end{remark}

The preceding discussion does not necessarily indicate that the second moment perspective of Luo and Sarnak~\cites[Section~3]{LuoSarnak1995} invariably represents a detour in the study of the prime geodesic theorem! In fact, the proof of Balog et al.~\cites[Theorem~1]{BalogBiroHarcosMaga2019} leverages the full strength of~\cites[Theorem~3.2]{LuoSarnak1995}, which serves as a key to refining the result of Cherubini and Guerreiro~\cites[Theorem~1.4]{CherubiniGuerreiro2018}, whose proof relies fundamentally on~\eqref{eq:Iwaniec}. Furthermore, in a comparable fashion, Chatzakos, Cherubini, and Laaksonen~\cites[Equation~(5.4)]{ChatzakosCherubiniLaaksonen2022} as well as the author~\cites[Equation~(3-54)]{Kaneko2022-2} adapted the second moment framework to demonstrate power-saving improvements on average for the prime geodesic theorem in the $3$-dimensional setting. In general, the pointwise analysis of the prime geodesic theorem pertains to the first moment of the Rankin--Selberg $L$-function, whereas the averaged analysis pertains to higher moments of the Rankin--Selberg $L$-function, the latter of which exhibits greater complexity. Notably, this distinction becomes salient especially over number fields, where the absence of an optimal spectral large sieve inequality may hider the applicability of existing techniques; cf.~\cites{Qi2024-3}{Qi2024}{Qi2024-2}{Watt2013}{Watt2014}. This observation deserves further penetration.

We conclude this section by formulating and naming the \textit{square mean Lindel\"{o}f hypothesis}, which is equivalent to~\cites[Assumption~3.2]{Koyama2001} and related via the Watson--Ichino formula to a Lindel\"{o}f-on-average bound in the quantum variance problem for unitary Eisenstein series.
\begin{conjecture}[Square mean Lindel\"{o}f hypothesis]\label{conj:square-mean}
Let $\Gamma = \mathrm{PSL}_{2}(\mathbb{Z}[i])$. Define $\eta \in [0, 1]$ to be such that there exists an absolute and effectively computable constant $\Cl[constant]{square-mean} > 0$ such that for any $\tau \in \mathbb{R}$ and $\epsilon > 0$,
\begin{equation}\label{eq:square-mean}
\sum_{t_{j} \leq T} \left|\frac{t_{j}}{\sinh \pi t_{j}} L \left(\frac{1}{2}+i\tau, u_{j} \otimes u_{j} \right) \right|^{2} \ll_{\epsilon} (1+|\tau|)^{\Cr{square-mean}} T^{3+\eta+\epsilon}.
\end{equation}
Then $\eta = 0$ is admissible.
\end{conjecture}

\begin{theorem}
The quantity $\eta = \min(\frac{1}{2}, 4\vartheta)$ is admissible for any $\Cr{square-mean} > 3$.
\end{theorem}

\begin{proof}
By shifting our attention to symmetric square $L$-functions~\cites[Equation~(3.2)]{BalkanovaChatzakosCherubiniFrolenkovLaaksonen2019}, the claim is immediate from the combination of~\cites[Theorem~1.3]{BalkanovaFrolenkov2022}[Theorem~2]{Qi2024}.
\end{proof}

Furthermore, it behoves us to record the following reduction.
\begin{theorem}
\cref{conj:square-mean} implies \cref{conj:mean-Lindelof} for any $\Cr{Koyama} \geq \frac{\Cr{square-mean}}{2}$.
\end{theorem}

\begin{proof}
The claim is immediate from the Cauchy--Schwarz inequality and Weyl's law~\cites[Theorem~8.9.1]{ElstrodtGrunewaldMennicke1998}.
\end{proof}

As Qi~\cites[Page~3]{Qi2024} observes, establishing \cref{conj:square-mean} via the approach of Khan and Young~\cites[Theorem~1.3]{KhanYoung2023} or that of Balkanova and Frolenkov~\cites[Theorem~1.3]{BalkanovaFrolenkov2024} would constitute an undoubtedly feasible, albeit formidable, task.

\subsection{Automorphic data}
As a preliminary step, it is convenient for the reader to assemble requisite background material. Let $\mathbb{H}^{3} \cong \mathrm{SL}_{2}(\mathbb{C})/\mathrm{SU}_{2}(\mathbb{C})$ denote the standard $3$-dimensional upper half-plane, upon which the Picard group $\Gamma = \mathrm{PSL}_{2}(\mathbb{Z}[i])$ acts discontinuously via~linear fractional transformations. Given the basis $\{1, i, j, k \}$ of Hamiltonian quaternions, a typical point $v \in \mathbb{H}^{3}$ is represented by $v = z+rj$ with $r > 0$ and $z = x+yi \in \mathbb{C}$; thus, $\Im(v) \coloneqq r$, $\Im(z) \coloneqq y$, and $\Re(z) \coloneqq x$. The space $\mathbb{H}^{3}$ carries the hyperbolic metric $r^{-1} \sqrt{dx^{2}+dy^{2}+dr^{2}}$ and the corresponding volume element $d\mu(v) \coloneqq r^{-3} \, dx \, dy \, dr$. The Laplace--Beltrami operator on $\Gamma \backslash \mathbb{H}^{3}$ or rather its negative is denoted by $\Delta$, namely~\cites[Equation~(1.1.6)]{ElstrodtGrunewaldMennicke1998}
\begin{equation}
\Delta \coloneqq -y^{2} \left(\frac{\partial^{2}}{\partial x^{2}}+\frac{\partial^{2}}{\partial y^{2}}+\frac{\partial^{2}}{\partial r^{2}} \right)+r \frac{\partial}{\partial r}.
\end{equation}
For the motions of the points in $\mathbb{H}^{3}$, we define
\begin{equation}
\mathbb{T}(\mathbb{H}^{3}) \coloneqq \{v \mapsto (av+b)(cv+d)^{-1}: a, b, c, d \in \mathbb{C}, \, ac-bd = 1 \}.
\end{equation}
In the coordinate notation, this generic map sends~\cites[Equations~(1.1.9) and~(1.1.10)]{ElstrodtGrunewaldMennicke1998}
\begin{equation}
v = (z, r) \mapsto \left(\frac{(az+b) \overline{(cz+d)}+a \overline{c} r^{2}}{|cz+d|^{2}+|cr|^{2}}, \frac{r}{|cz+d|^{2}+|cr|^{2}} \right) \in \mathbb{H}^{3},
\end{equation}
from which the identification $\mathbb{T}(\mathbb{H}^{3}) \cong \mathrm{PSL}_{2}(\mathbb{C})$ follows. The metric, the volume element,~and the Laplace--Beltrami operator $\Delta$ are all invariant with respect to the maps in $\mathbb{T}(\mathbb{H}^{3})$, namely these maps are rigid motions of $\mathbb{H}^{3}$. A common choice for the fundamental domain of $\Gamma \backslash \mathbb{H}^{3}$~is the region~\cites[Page~47]{Picard1884}
\begin{equation}\label{eq:fundamental-domain}
\mathrm{vol}(\Gamma \backslash \mathbb{H}^{3}) \cong \left\{v \in \mathbb{H}^{3}: |x| \leq \frac{1}{2}, \, |y| \leq \frac{1}{2}, \, x^{2}+y^{2}+r^{2} \geq 1 \right\},
\end{equation}
whose volume is finite~\cites[Equation~(17)]{Humbert1919}[Theorem~7.4.1]{Thurston2022}
\begin{equation}\label{eq:volume}
\mathrm{vol}(\Gamma \backslash \mathbb{H}^{3}) = \frac{\zeta_{\mathbb{Q}(i)}(2)}{4\pi^{2}},
\end{equation}
where $\zeta_{K}(s)$ denotes the Dedekind zeta function of the field $K = \mathbb{Q}(i)$, as defined by~\eqref{eq:Dedekind}.

Let $L^{2}(\Gamma \backslash \mathbb{H}^{3})$ denote the set of all $\Gamma$-invariant functions over $\mathbb{H}^{3}$ that are square integrable with respect to $d\mu(v)$ over $\Gamma \backslash \mathbb{H}^{3}$. This constitutes a Hilbert space endowed with the Petersson inner product
\begin{equation}
\langle f, g \rangle \coloneqq \int_{\Gamma \backslash \mathbb{H}^{3}} f(v) \overline{g(v)} \, d\mu(v),
\end{equation}
which is well-defined for any $f, g \in L^{2}(\Gamma \backslash \mathbb{H}^{3})$. The Roelcke--Selberg spectral resolution of $\Delta$ over $\mathbb{H}^{3}$ parallels the case of $\mathbb{H}^{2}$, thereby admitting a self-adjoint extension on $L^{2}(\Gamma \backslash \mathbb{H}^{3})$ with its spectrum comprising both discrete and continuous components. Moreover, let $\mathcal{B}_{it_{j}}(\Gamma \backslash \mathbb{H}^{3})$ denote an orthonormal basis of Maa{\ss} cusp forms having Laplace eigenvalue~$\lambda_{j} = 1+t_{j}^{2} \geq \pi^{2}$; see~\cites[Proposition~2]{Szmidt1983}. For $u_{j} \in \mathcal{B}_{it_{j}}(\Gamma \backslash \mathbb{H}^{3})$, the sequence $(u_{j})_{j \in \mathbb{N}}$ is ordered such that the corresponding sequence of spectral parameters $(t_{j})_{j \in \mathbb{N}}$ increases monotonically. A Maa{\ss}~cusp form $u_{j}$ enjoys the Fourier expansion~\cites[Theorem~3.3.1]{ElstrodtGrunewaldMennicke1998}[Equation~(2.20)]{Sarnak1983}
\begin{equation}\label{eq:Fourier-expansion}
u_{j}(v) = r \sum_{0 \ne n \in \mathbb{Z}[i]} \rho_{j}(n) K_{it_{j}}(2\pi |n|r) \check{e}(nz),
\end{equation}
where $K_{\nu}$ denotes the modified Bessel function of the second kind. For the ensuing discussion, it is equally advantageous to introduce the Eisenstein series~\cites[Definition~8.1.1]{ElstrodtGrunewaldMennicke1998}[Equation~(8)]{Kubota1968}
\begin{equation}
E(v, s) \coloneqq \sum_{\gamma \in \Gamma_{\infty} \backslash \Gamma} r(\gamma v)^{s}, \qquad \Re(s) > 2,
\end{equation}
where $r(v)$ denotes the third coordinate of $v$, and $\Gamma_{\infty}$ denotes the stabiliser in $\Gamma$ of the cusp at $\infty$. Note that $\Gamma_{\infty} = \Gamma_{t} \cup \iota \Gamma_{t}$, where $\Gamma_{t} \coloneqq \{z \mapsto z+b: b \in \mathbb{Z}[i] \}$ denotes the translation~group, and $\iota(v) \coloneqq -z+rj$ denotes the involution, namely $z \mapsto (iz) (-i)^{-1}$ in~the~Hamiltonian sense. Hence, we identify
\begin{equation}
\Gamma_{\infty} \cong \left\{\pm \begin{pmatrix} 1 & b \\ 0 & 1 \end{pmatrix}, \, \pm \begin{pmatrix} -i & b^{\prime} \\ 0 & i \end{pmatrix}: b, b^{\prime} \in \mathbb{Z}[i] \right\} \Big/ \{\pm 1 \},
\end{equation}
from which it follows that~\cites[Page~394]{Szmidt1983}
\begin{equation}
E(v, s) = \sum_{\substack{c, d \in \mathbb{Z}[i] \\ (c, d) = 1}} \frac{r^{s}}{(|cz+d|^{2}+|cr|^{2})^{s}}.
\end{equation}
By expanding the Poincar\'{e} series into a double Fourier series in $x$ and $y$ and applying Poisson summation, the Eisenstein series $E(v, s)$ enjoys the Fourier expansion~\cites[Equation~(13)]{Asai1970}[Theorem~8.2.11]{ElstrodtGrunewaldMennicke1998}
\begin{equation}\label{eq:Fourier-expansion-Eisenstein}
E(v, s) = r^{s}+\frac{\pi}{s-1} \frac{\zeta_{\mathbb{Q}(i)}(s-1)}{\zeta_{\mathbb{Q}(i)}(s)} r^{2-s}+\frac{2\pi^{s} r}{\Gamma(s) \zeta_{\mathbb{Q}(i)}(s)} \sum_{0 \ne n \in \mathbb{Z}[i]} |n|^{s-1} \sigma_{1-s}(n) K_{s-1}(2\pi |n|r) \check{e}(nz).
\end{equation}
Hence, $E(v, s)$ is meromorphic over $\mathbb{C}$ as a function of $s$ and holomorphic for $\Re(s) \geq 1$~except for the simple pole at $s = 2$ with residue given in terms of the reciprocal of~\eqref{eq:volume}; cf.~\cites[Equation~(8.3.2)]{ElstrodtGrunewaldMennicke1998}.

We consider the Rankin--Selberg convolution $u_{j} \otimes u_{j} = 1 \boxplus \mathrm{sym}^{2} u_{j}$, where $\mathrm{sym}^{2} u_{j}$ denotes the symmetric square lift of $u_{j}$. The corresponding Rankin--Selberg $L$-function is defined by
\begin{equation}\label{eq:Rankin-Selberg}
L(s, u_{j} \otimes u_{j}) \coloneqq \sum_{0 \ne n \in \mathbb{Z}[i]} \frac{|\rho_{j}(n)|^{2}}{\mathrm{N}(n)^{s}}, \qquad \Re(s) > 1,
\end{equation}
which converges absolutely for $\Re(s) > 1$ and admits a meromorphic continuation to $\mathbb{C}$. The functional equation of the Rankin--Selberg $L$-function $L(s, u_{j} \otimes u_{j})$ is then inherited from~the Eisenstein series via the standard unfolding argument, namely
\begin{equation}\label{eq:functional-equation}
\Lambda(s, u_{j} \otimes u_{j}) \coloneqq \frac{\pi^{2s} \Gamma(s)^{2} \Gamma(s+it_{j}) \Gamma(s-it_{j})}{\Gamma(2s)} L(s, u_{j} \otimes u_{j}) = \Lambda(1-s, u_{j} \otimes u_{j}).
\end{equation}
In practice, it is straightforward to verify that the inner product~\cites[Page~803]{Koyama1995}
\begin{equation}\label{eq:inner-product}
\langle |u_{j}|^{2}, E(\cdot, 2s) \rangle = \int_{\Gamma \backslash \mathbb{H}^{3}} |u_{j}(v)|^{2} E(v, 2s) d\mu(v) = \frac{\Gamma(s)^{2} \Gamma(s+it_{j}) \Gamma(s-it_{j})}{8\pi^{2s} \Gamma(2s)} L(s, u_{j} \otimes u_{j})
\end{equation}
is invariant under $s \leftrightarrow 1-s$. For future notational convenience, one may abbreviate
\begin{equation}
L^{\diamond}(s, u_{j} \otimes u_{j}) \coloneqq \zeta_{\mathbb{Q}(i)}(2s) L(s, u_{j} \otimes u_{j}), \qquad \gamma(s, u_{j} \otimes u_{j}) \coloneqq \frac{\pi^{2s} \Gamma(s)^{2} \Gamma(s+it_{j}) \Gamma(s-it_{j})}{\Gamma(2s) \zeta_{\mathbb{Q}(i)}(2s)},
\end{equation}
so that the functional equation~\eqref{eq:functional-equation} translates to
\begin{equation}\label{eq:functional-equation-2}
\gamma(s, u_{j} \otimes u_{j}) L^{\diamond}(s, u_{j} \otimes u_{j}) = \gamma(1-s, u_{j} \otimes u_{j}) L^{\diamond}(1-s, u_{j} \otimes u_{j}).
\end{equation}

\subsection{Optimal mean square asymptotics for Fourier coefficients}
In his seminal work, Kuznetsov~\cites[Theorem~6]{Kuznetsov1978} leveraged the pre-Kuznetsov formula\footnote{The pre-Kuznetsov formula serves as an auxiliary version of the Kuznetsov formula~\cites[Theorem]{Motohashi1997} and arises from evaluating the Fourier coefficient of an automorphic kernel via two distinct approaches. One approach utilises the Bruhat decomposition, while the other utilises the spectral decomposition of $L^{2}(\Gamma \backslash \mathbb{H}^{3})$} to evaluate the spectral mean square of Fourier coefficients of Maa{\ss} cusp forms on the modular orbifold, namely we have for any $0 \ne n \in \mathbb{Z}$ and $\epsilon > 0$ that (cf.~\cites[Proposition~(4.1)]{Bruggeman1978}[Theorem~1]{Emirleroglu1997})
\begin{equation}\label{eq:Kuznetsov-mean-square}
\sum_{t_{j} \leq T} \frac{|\rho_{j}(n)|^{2}}{\cosh \pi t_{j}} = \frac{T^{2}}{\pi^{2}}+O_{\epsilon}(T \log T+T|n|^{\epsilon}+|n|^{\frac{1}{2}+\epsilon}).
\end{equation}
As a $3$-dimensional counterpart of~\eqref{eq:Kuznetsov-mean-square}, Raghavan and Sengupta~\cites[Theorem]{RaghavanSengupta1994}[Theorem~1]{RaghavanSengupta1994-2} investigated the spherical mean square of Fourier coefficients of Maa{\ss} cusp forms on a $3$-dimensional hyperbolic orbifold, which asserts in the case of $\Gamma \backslash \mathbb{H}^{3}$ that there exists an absolute and effectively computable constant $\Cl[constant]{Raghavan-Sengupta} > 0$ such that for any $0 \ne n \in \mathbb{Z}[i]$~and~$\epsilon > 0$,
\begin{equation}\label{eq:Raghavan-Sengupta}
\sum_{t_{j} \leq T} \frac{t_{j}}{\sinh \pi t_{j}} |\rho_{j}(n)|^{2} = \Cr{Raghavan-Sengupta} T^{3}+O_{\epsilon}(T^{\frac{5}{2}} \mathrm{N}(n)^{\frac{1}{2}+\epsilon}).
\end{equation}
In the meantime, Motohashi~\cites[Lemma~11]{Motohashi1997} established a comparatively coarser upper bound in place of an asymptotic formula, albeit with the error term of superior quality in~the spectral aspect, namely (after multiplying appropriate normalising factors)
\begin{equation}\label{eq:Motohashi-bound}
\sum_{t_{j} \leq T} \frac{t_{j}}{\sinh \pi t_{j}} |\rho_{j}(n)|^{2} \ll T^{3}+T^{2} \mathrm{N}(n).
\end{equation}
In principle, a more careful treatment of the diagonal term on the right-hand side of~\cites[Equation~(5.2)]{Motohashi1997} enables the isolation of the desired main term on the right-hand side of~\eqref{eq:Raghavan-Sengupta}. In this regard, after stating~\eqref{eq:Motohashi-bound}, Motohashi~\cites[Page~284]{Motohashi1997} writes
\begin{quote}
It should be remarked that our argument can be refined at least in two immediate respects outside the aforementioned possibility of its large sieve extension: The factor $|m|^{2}$ could be replaced by $|m|^{1+\epsilon}$ for any $\epsilon > 0$. Also an asymptotic equality could be proved instead of the inequality. We note that Raghavan and Sengupta [9] considered a sum that is essential the same as ours, and proved an asymptotic expression for it. However, it appears to us that their argument is hard to extend to the large sieve situation.
\end{quote}

In particular, Motohashi refrained from improving upon~\eqref{eq:Motohashi-bound}, as its sufficiency in deriving the Kuznetsov formula for the Picard orbifold $\Gamma \backslash \mathbb{H}^{3}$ aligned with his primary objective. The following result validates the above quotation and simultaneously sharpens~\eqref{eq:Raghavan-Sengupta} and~\eqref{eq:Motohashi-bound}.
\begin{proposition}\label{prop:mean-square}
Let $\Gamma = \mathrm{PSL}_{2}(\mathbb{Z}[i])$. Then there exists an absolute and effectively computable constant $\Cr{Raghavan-Sengupta} > 0$ such that for any $0 \ne n \in \mathbb{Z}[i]$ and $\epsilon > 0$,
\begin{equation}\label{eq:mean-square}
\sum_{t_{j} \leq T} \frac{t_{j}}{\sinh \pi t_{j}} |\rho_{j}(n)|^{2} = \Cr{Raghavan-Sengupta} T^{3}+O_{\epsilon}(T^{2} \mathrm{N}(n)^{\frac{1}{2}+\epsilon}+T^{1+\epsilon} \mathrm{N}(n)^{1+\epsilon}).
\end{equation}
\end{proposition}

\begin{remark}
The error term in~\eqref{eq:mean-square} is provably optimal. As a heuristic digression, we~may highlight that the left-hand side of~\eqref{eq:mean-Lindelof} involves a summation over $0 \ne n \in \mathbb{Z}[i]$ of norm not exceeding $(1+t_{j})^{2}$ by the approximate functional equation. Consequently, the contribution of the main term in~\eqref{eq:mean-square} becomes negligibly small, while the error term contributes $O_{\epsilon}(T^{3+\epsilon})$, in harmony with our desideratum. Furthermore, extending the result of Emirlero\u{g}lu~\cites[Theorem~1]{Emirleroglu1997} to our setting via the arithmetic machinery would pose additional challenges.
\end{remark}

\begin{remark}
The author became aware, subsequent to the completion of the present paper, that Bruggeman and Motohashi~\cites[Corollary~10.1]{BruggemanMotohashi2003} demonstrated a bound of comparable quality, extending the spherical family of Maa{\ss} cusp forms on $\Gamma \backslash \mathbb{H}^{3}$ to the family of cuspidal automorphic representations on $\Gamma \backslash \mathrm{PSL}_{2}(\mathbb{C})$. Our argument is more elementary than theirs.
\end{remark}

\begin{proof}
The primary difference between the derivations of~\eqref{eq:Raghavan-Sengupta} and~\eqref{eq:Motohashi-bound} lies in the analysis of certain integral transforms against Bessel functions. We intend to present a refined version of the fourth display on~\cites[Page~283]{Motohashi1997}. To this end, it is convenient to commence with a reformulation of the pre-Kuznetsov formula over $\mathbb{Q}(i)$. Given $(s_{1}, s_{2}) \in \mathbb{C}^{2}$ and $r \in \mathbb{R}$, let
\begin{equation}
\Lambda(s_{1}, s_{2}, r) \coloneqq \Gamma(s_{1}-1+ir) \Gamma(s_{1}-1-ir) \Gamma(s_{2}-1+ir) \Gamma(s_{2}-1-ir).
\end{equation}
Then the pre-Kuznetsov formula of Motohashi~\cites[Equation~(5.2)]{Motohashi1997} demonstrates that~for any $0 \ne n \in \mathbb{Z}[i]$,
\begin{multline}\label{eq:pre-Kuznetsov}
\sum_{j = 1}^{\infty} |\rho_{j}(n)|^{2} \Psi(t, t_{j})+2\pi^{2} \int_{-\infty}^{\infty} \frac{|\sigma_{ir}(n)|^{2}}{|\Gamma(1+ir) \zeta_{\mathbb{Q}(i)}(1+ir)|^{2}} \Psi(t, r) dr\\
 = 1+\pi^{2} \mathrm{N}(n) \sum_{0 \ne c \in \mathbb{Z}[i]} \frac{S(n, n, c)}{\mathrm{N}(c)^{2}} \Phi(t; n, c),
\end{multline}
where
\begin{align}
\Psi(t, r) &\coloneqq \frac{\pi \Lambda(2+it, 2-it, r)}{|\Gamma(\frac{3}{2}+it)|^{2}} = \frac{\pi^{2}(t^{2}-r^{2}) \cosh \pi t}{(\frac{1}{4}+t^{2}) \sinh(\pi(t+r)) \sinh(\pi(t-r))},\\
\Phi(t; n, c) &\coloneqq -i \int_{0}^{1} \int_{0}^{1} u^{it-1}(1-u)^{-it-1} v^{-\frac{1}{2}}\\
& \quad \hspace{.32em} \times \int_{(\alpha)} \frac{\Gamma(\eta)}{\Gamma(2-\eta)} \left(\frac{\mathrm{N}(c)u(1-u)v}{\pi^{2} \mathrm{N}(n)(1-(2\sin \vartheta_{0})^{2} u(1-u)v)} \right)^{\eta} \, d\eta \, du \, dv,
\end{align}
with $\alpha \in (0, \frac{1}{2})$ and $\vartheta_{0} \coloneqq \arg \frac{n}{c}$. We also invoke the pair of integral representations
\begin{equation}\label{eq:J-integral}
J_{\nu}(x) = \int_{(\alpha)} \frac{\Gamma(\eta+\frac{\nu}{2})}{\Gamma(1-\eta+\frac{\nu}{2})} \left(\frac{x}{2} \right)^{-2\eta} \frac{d\eta}{2\pi i}, \qquad \int_{0}^{\infty} J_{\nu}(x) \left(\frac{x}{2} \right)^{-\mu} \, dx = \frac{\Gamma(\frac{1+\nu-\mu}{2})}{\Gamma(\frac{1+\nu+\mu}{2})}
\end{equation}
for $x > 0$, $\Re(\nu) > -2\alpha > 0$, and $\frac{1}{2} < \Re(\mu) < 1+\Re(\nu)$.

With the notation established, we multiply both sides of~\eqref{eq:pre-Kuznetsov} by a nonnegative function
\begin{equation}
w_{T}(t) \coloneqq 
	\begin{dcases}
	\frac{|\Gamma(\frac{3}{2}+it)|^{2} \cosh \pi t}{\pi} = \frac{1}{4}+t^{2} & \text{if $t \in [0, T]$},\\
	0 & \text{otherwise},
	\end{dcases}
\end{equation}
for a sufficiently large parameter $T > 0$, and integrate with respect to $t \in \mathbb{R}$. The adaptation of the techniques of Kuznetsov~\cites[Section~5]{Kuznetsov1980} then enables the elimination of the weight function $\Psi(t, r)$; hence, the first term on the left-hand side of~\eqref{eq:pre-Kuznetsov} gives rise to the~spectral mean square in question, at the expense of introducing an error term of size $O_{\epsilon}(T^{1+\epsilon} \mathrm{N}(n)^{1+\epsilon})$. More precisely, it follows from~\cites[Equations~(14) and~(15)]{RaghavanSengupta1994} that
\begin{align}
\frac{1}{\pi |\Gamma(1+ir)|^{2}} \int_{0}^{T} \Psi(t, r) w_{T}(t) dt &= \int_{0}^{T} \frac{t^{2}-r^{2}}{2r} \left(\frac{1}{\sinh(\pi(t-r))}-\frac{1}{\sinh(\pi(t+r))} \right) \, dt\\
& = \begin{dcases}
\frac{1}{2}+O(T^{-\frac{\pi}{2}}+e^{-\pi |r|}) & \text{if $|r| \leq T-\log T$},\\
O((|r|-T)e^{-\pi(|r|-T)}) & \text{if $|r| \geq T+\log T$}.
	\end{dcases}
\end{align}
thereby leading to an approximation to the left-hand side of~\eqref{eq:mean-square} via the reasoning identical to that of Raghavan and Sengupta~\cites[Page~83]{RaghavanSengupta1994}[Pages~281--282]{RaghavanSengupta1994-2}. Moreover, the required main term in~\eqref{eq:mean-square} follows from the diagonal term on the right-hand side of~\eqref{eq:pre-Kuznetsov}, since
\begin{equation}
\hat{w}_{T}(0) = \frac{T^{3}}{3}+\frac{T}{4}.
\end{equation}
As for the continuous spectrum, arguing analogously to~\cites[Sections~3.10 and 3.11]{Titchmarsh1986} shows $\zeta_{\mathbb{Q}(i)}(1+ir) \gg (\log(2+|r|))^{-1}$ for $r \in \mathbb{R}$, so that the entire Eisenstein contribution amounts to $O(\tau(n)^{2} T(\log T)^{2})$, which is negligibly small and thus absorbed into the other terms.

Altogether, we are led to the problem of nontrivially estimating the Kloosterman term on the right-hand side of~\eqref{eq:pre-Kuznetsov}, and there exists an absolute and effectively computable~constant $\Cr{Raghavan-Sengupta} > 0$ such that for any $0 \ne n \in \mathbb{Z}[i]$,
\begin{equation}\label{eq:P-T}
\sum_{t_{j} \leq T} \frac{t_{j}}{\sinh \pi t_{j}} |\rho_{j}(n)|^{2} = \Cr{Raghavan-Sengupta} T^{3}+O \left(T^{2} \mathrm{N}(n)^{\frac{1}{2}} \sum_{0 \ne c \in \mathbb{Z}[i]} \frac{|S(n, n, c)|}{\mathrm{N}(c)^{\frac{3}{2}}} |P(n, c)| \right),
\end{equation}
where, by~\eqref{eq:J-integral}, we set
\begin{equation}
P(n, c) \coloneqq \int_{0}^{1} \int_{0}^{1} \frac{u^{-1} (1-u)^{-1} v^{\frac{1}{2}}}{\arctanh(1-2u)} J_{1} \left(\frac{2\pi \mathrm{N}(n)^{\frac{1}{2}}}{\mathrm{N}(c)^{\frac{1}{2}}} \left(\frac{1-(2\sin \vartheta_{0})^{2} u(1-u)v}{u(1-u)v} \right)^{\frac{1}{2}} \right) \, du \, dv,
\end{equation}
with $J_{\nu}$ denoting the Bessel function of the first kind. Invoking the asymptotic formul{\ae}~\cites[Equations~(8.440) and~(8.451.1)]{GradshteynRyzhik2014}
\begin{equation}\label{eq:J-1}
J_{1}(y) \sim 
	\begin{dcases}
	\frac{y}{2} & \text{as $y \to 0$},\\
	\sqrt{\frac{2}{\pi y}} \cos \left(y-\frac{3\pi}{4} \right) = \frac{\sin y-\cos y}{\sqrt{\pi y}} & \text{as $y \to \infty$},
	\end{dcases}
\end{equation}
we truncate the summation over $c$ in~\eqref{eq:P-T} according as $\mathrm{N}(c) \ll \mathrm{N}(n)$ or not. In the former range, estimating everything trivially with the second case of~\eqref{eq:J-1} in conjunction with the Weil--Gundlach bound~\eqref{eq:Weil} yields
\begin{multline}\label{eq:former-range}
T^{2} \mathrm{N}(n)^{\frac{1}{2}} \sum_{0 \ne \mathrm{N}(c) \ll \mathrm{N}(n)} \frac{|S(n, n, c)|}{\mathrm{N}(c)^{\frac{3}{2}}} |P(n, c)|\\
\ll T^{2} \mathrm{N}(n)^{\frac{1}{2}} \sum_{0 \ne \mathrm{N}(c) \ll \mathrm{N}(n)} \frac{(n, c) \tau(c) \mathrm{N}(c)^{\frac{1}{2}}}{\mathrm{N}(c)^{\frac{3}{2}}} \frac{\mathrm{N}(c)^{\frac{1}{4}}}{\mathrm{N}(n)^{\frac{1}{4}}} \ll_{\epsilon} T^{2} \mathrm{N}(n)^{\frac{1}{2}+\epsilon},
\end{multline}
since the greatest common divisor $(n, c)$ is bounded on average by~\eqref{eq:bounded-on-average}. In a similar~manner, the contribution from the tail range $\mathrm{N}(c) \gg \mathrm{N}(n)$ amounts to
\begin{multline}
T^{2} \mathrm{N}(n)^{\frac{1}{2}} \sum_{0 \ne \mathrm{N}(c) \gg \mathrm{N}(n)} \frac{|S(n, n, c)|}{\mathrm{N}(c)^{\frac{3}{2}}} |P(n, c)|\\
\ll T^{2} \mathrm{N}(n)^{\frac{1}{2}} \sum_{0 \ne \mathrm{N}(c) \gg \mathrm{N}(n)} \frac{(n, c) \tau(c) \mathrm{N}(c)^{\frac{1}{2}}}{\mathrm{N}(c)^{\frac{3}{2}}} \frac{\mathrm{N}(n)^{\frac{1}{2}}}{\mathrm{N}(c)^{\frac{1}{2}}} \ll_{\epsilon} T^{2} \mathrm{N}(n)^{\frac{1}{2}+\epsilon}.\label{eq:latter-range}
\end{multline}
Combining~\eqref{eq:P-T},~\eqref{eq:former-range}, and~\eqref{eq:latter-range} thus completes the proof of \cref{prop:mean-square}.
\end{proof}

While not essential to our subsequent argument, a comparison of the main term with~\cites[Theorem]{RaghavanSengupta1994}[Theorem~1]{RaghavanSengupta1994-2} determines the leading coefficient in~\eqref{eq:mean-square}.
\begin{proposition}
Keep the notation as in \cref{prop:mean-square}. Then
\begin{equation}
\Cr{Raghavan-Sengupta} = \frac{\mathrm{vol}(\Gamma \backslash \mathbb{H}^{3})}{6} = \frac{G}{144} = 0.00636 \cdots,
\end{equation}
where $G$ is Catalan's constant.
\end{proposition}

\begin{proof}
The claim follows from substituting~\eqref{eq:volume} and
\begin{equation}
2 \int_{0}^{\infty} \frac{u}{\sinh \pi u} \, du = \frac{1}{2}, \qquad |d_{K}| = 4, \qquad \zeta_{\mathbb{Q}(i)}(2) = \zeta(2) L(2, \chi_{-4}) = \frac{\pi^{2} G}{6}
\end{equation}
into~\cites[Theorem~1]{RaghavanSengupta1994-2}, where $\chi_{-4}$ denotes the nontrivial Dirichlet character modulo~$4$.
\end{proof}

We now embark on enhancing the convexity machinery of Iwaniec~\cites[Pages~152--153]{Iwaniec1984} to demonstrate an auxiliary version of \cref{thm:mean-Lindelof}.
\begin{proposition}\label{prop:Iwaniec-consult}
Let $\Gamma = \mathrm{PSL}_{2}(\mathbb{Z}[i])$, and let $u_{j} \in \mathcal{B}_{it_{j}}(\Gamma \backslash \mathbb{H}^{3})$ traverse an orthonormal~basis of Maa{\ss} cusp forms. Then we have for any arbitrarily small $\delta > 0$ that
\begin{equation}\label{eq:Iwaniec-consult}
\sum_{t_{j} \leq T} \frac{t_{j}}{\sinh \pi t_{j}} L^{\diamond}(1+\delta, u_{j} \otimes u_{j}) \ll \frac{T^{3}}{\delta}.
\end{equation}
\end{proposition}

\begin{proof}
Since the Rankin--Selberg $L$-function under consideration lies off the critical strip, no oscillation can be detected. Hence, to estimate the summand in~\eqref{eq:Iwaniec-consult} individually, it follows from the integral representation~\eqref{eq:inner-product}, the Fourier expansion~\eqref{eq:Fourier-expansion}, and the standard bound (cf.~\eqref{eq:fundamental-domain} and~\eqref{eq:Fourier-expansion-Eisenstein})
\begin{equation}
|E(v, 2s)| \ll \frac{1}{\delta}+r^{2+2\delta}, \qquad \Re(s) = 1+\delta, \qquad r \coloneqq \Im(v) \geq \frac{1}{\sqrt{2}}
\end{equation}
that
\begin{align}
\frac{t_{j}}{\sinh \pi t_{j}} L(1+\delta, u_{j} \otimes u_{j}) &\ll t_{j}^{-2\delta} \left(\frac{1}{\delta}+\int_{\frac{1}{\sqrt{2}}}^{\infty} r^{2\delta-1} \int_{0}^{1} \int_{0}^{1} |u_{j}(v)|^{2} \, dx dy dr \right)\\
&\ll t_{j}^{-2\delta} \left(\frac{1}{\delta}+\sum_{0 \ne n \in \mathbb{Z}[i]} |\rho_{j}(n)|^{2} \int_{\frac{1}{\sqrt{2}}}^{\infty} K_{it_{j}}(2\pi|n|r)^{2} r^{1+2\delta} dr \right).
\end{align}
By~\cites[Equation~(6.576.4)]{GradshteynRyzhik2014} and Stirling's formula~\eqref{eq:Stirling}, we estimate the integral on the right-hand side for $\mathrm{N}(n) \leq t_{j}$ by
\begin{equation}
\int_{0}^{\infty} K_{it_{j}}(2\pi|n|r)^{2} r^{1+2\delta} dr = \frac{\Gamma(1+\delta)^{2} |\Gamma(1+\delta+it_{j})|^{2}}{8(\pi|n|)^{2+2\delta} \Gamma(2+2\delta)} \ll \frac{\mathrm{N}(n)^{-1-\delta} t_{j}^{1+2\delta}}{\sinh \pi t_{j}}.
\end{equation}
On the other hand, if $\mathrm{N}(n) > t_{j}$, then we employ the cruder bounds $|K_{it_{j}}(y)| \ll \exp(\frac{\pi}{2} t_{j}-y)$ and $|\rho_{j}(n)| \leq |\rho_{j}(1)| \mathrm{N}(n)$, thereby ensuring the truncation of the summation up to $\mathrm{N}(n) \leq t_{j}$. Consequently, executing the summation over $t_{j} \leq T$ via Weyl's law~\cites[Theorem~8.9.1]{ElstrodtGrunewaldMennicke1998} and \cref{prop:mean-square} yields
\begin{equation}
\sum_{t_{j} \leq T} \frac{t_{j}}{\sinh \pi t_{j}} L(1+\delta, u_{j} \otimes u_{j}) \ll \frac{T^{3-2\delta}}{\delta}+\sum_{t_{j} \leq T} \frac{t_{j}}{\sinh \pi t_{j}} \sum_{\mathrm{N}(n) \leq t_{j}} \frac{|\rho_{j}(n)|^{2}}{\mathrm{N}(n)^{1+\delta}} \ll \frac{T^{3}}{\delta}.
\end{equation}
As the boundedness of the Dedekind zeta function off the critical strip justifies the transition from $L(1+\delta, u_{j} \otimes u_{j})$ to $L^{\diamond}(1+\delta, u_{j} \otimes u_{j})$, the proof of \cref{prop:Iwaniec-consult} is thus complete.
\end{proof}

\subsection{Endgame: Landau's trick}
Recall that our objective is to demonstrate the existence of an absolute and effectively computable constant $\Cr{Koyama} > 0$ such that for any $\epsilon > 0$,
\begin{equation}
\sum_{t_{j} \leq T} \frac{t_{j}}{\sinh \pi t_{j}} |L(w, u_{j} \otimes u_{j})| \ll_{\epsilon} (1+|\Im(w)|)^{\Cr{Koyama}} T^{3+\epsilon}, \qquad \Re(w) = \frac{1}{2}.
\end{equation}
Our proof strategy fundamentally traces its origins to the ideas of Landau~\cites[Section~3]{Landau1915}. Moreover, one may intentionally adapt Matthes's approach~\cites[Section~5]{Matthes1995-2}, alongside the necessary corrections at some places, to facilitate a more effective methodological comparison.

It is convenient to embed the normalised Rankin--Selberg $L$-function into a single Dirichlet series, namely
\begin{equation}\label{eq:embed}
L^{\diamond}(s, u_{j} \otimes u_{j}) = \sum_{0 \ne n \in \mathbb{Z}[i]} \frac{1}{\mathrm{N}(n)^{s}} \sum_{d^{2} \mid n} \left|\rho_{j} \left(\frac{n}{d^{2}} \right) \right|^{2}, \qquad \Re(s) > 1,
\end{equation}
where we utilised the Dirichlet series expansions~\eqref{eq:Dedekind} and~\eqref{eq:Rankin-Selberg}. Given $x > 0$, $0 < \Re(w) < 1$, $\nu \in \mathbb{N}_{\geq 2}$, and $\xi \in \mathbb{R}$, the Cauchy formula for repeated integration yields
\begin{align}
\mathcal{R}_{it_{j}}(x, w, \xi, \nu) &\coloneqq \int_{0}^{x} \int_{0}^{x_{1}} \cdots \int_{0}^{x_{\nu-1}} \int_{(\xi)} L^{\diamond}(s+w, u_{j} \otimes u_{j}) \frac{x_{\nu}^{s}}{s} \frac{ds}{2\pi i} \, dx_{\nu} \cdots dx_{2} \, dx_{1}\\
& = \int_{(\xi)} L^{\diamond}(s+w, u_{j} \otimes u_{j}) \frac{x^{s+\nu}}{s(s+1) \cdots (s+\nu)} \frac{ds}{2\pi i}.
\end{align}
While addressing repeated integration in such generality is dispensable, this argument serves as a heuristic beacon for its potential applicability to relevant scenarios. It follows from~\eqref{eq:embed} and Perron's formula that for $\xi > 1-\Re(w)$,
\begin{equation}\label{eq:Perron}
\mathcal{R}_{it_{j}}(x, w, \xi, \nu) = \int_{0}^{x} \int_{0}^{x_{1}} \cdots \int_{0}^{x_{\nu-1}} \sum_{0 \ne \mathrm{N}(d) \leq \sqrt{x_{\nu}}} \frac{1}{\mathrm{N}(d)^{2w}} \ \sideset{}{^{\prime}} \sum_{0 \ne \mathrm{N}(n) \leq \frac{x_{\nu}}{\mathrm{N}(d)^{2}}} \frac{|\rho_{j}(n)|^{2}}{\mathrm{N}(n)^{w}} \, dx_{\nu} \cdots dx_{2} \, dx_{1},
\end{equation}
where the prime indicates that the last term of the sum must be multiplied by $\frac{1}{2}$ when $x_{\nu} \in \mathbb{Z}$. It is straightforward to verify from the functional equation~\eqref{eq:functional-equation-2}, Stirling's formula~\eqref{eq:Stirling}, and the Phragm\'{e}n--Lindel\"{o}f principle that $\mathcal{R}_{it_{j}}(x, w, \mu, \nu)$ is well-defined for $\mu > 1-\frac{\nu}{2}-\Re(w)$. Note that Matthes's range in~\cites[Page~8]{Matthes1995-2} is erroneous. The residue theorem now leads to our key identity
\begin{multline}\label{eq:key-identity}
\mathcal{R}_{it_{j}}(x, w-1, \xi+1, \nu-1)+(w-1) \mathcal{R}_{it_{j}}(x, w, \xi, \nu)\\
 = -\mathcal{R}_{it_{j}}(x, w-1, \mu+1, \nu-1)-(w-1) \mathcal{R}_{it_{j}}(x, w, \mu, \nu)+(w-1) L^{\diamond}(w, u_{j} \otimes u_{j}) \frac{x^{\nu}}{\nu!}
\end{multline}
for some $\mu \in (1-\frac{\nu}{2}-\Re(w), 0)$. Since
\begin{multline}
\mathcal{R}_{it_{j}}(x, w-1, \xi+1, \nu-1)+(w-1) \mathcal{R}_{it_{j}}(x, w, \xi, \nu)\\
 = \int_{(\xi)} (s+w-1) L^{\diamond}(s+w, u_{j} \otimes u_{j}) \frac{x^{s+\nu}}{s(s+1) \cdots (s+\nu)} \frac{ds}{2\pi i},
\end{multline}
there exists no pole at $s = 1-w$ corresponding to the simple pole at $s = 1$ of $L^{\diamond}(s, u_{j} \otimes u_{j})$. The analysis of the left-hand side of~\eqref{eq:key-identity} is conducted as follows.
\begin{lemma}\label{lem:Matthes}
Keep the notation and assumptions as above. Then we have for any $\epsilon > 0$~that
\begin{multline}
\sum_{t_{j} \leq T} \frac{t_{j}}{\sinh \pi t_{j}}(\mathcal{R}_{it_{j}}(x, w-1, \xi+1, \nu-1)+(w-1) \mathcal{R}_{it_{j}}(x, w, \xi, \nu))\\
\ll_{\epsilon} T^{3} x^{\nu+\epsilon}+T^{2} x^{\frac{3}{2}+\nu-\Re(w)+\epsilon}+T^{1+\epsilon} x^{2+\nu-\Re(w)+\epsilon}.
\end{multline}
\end{lemma}

\begin{proof}
Upon executing the summation over $t_{j} \leq T$ of the integrand in~\eqref{eq:Perron}, \cref{prop:mean-square} implies the existence of absolute and effectively computable constants $\Cr{Raghavan-Sengupta}, \Cl[constant]{key-identity} > 0$ such that
\begin{align}
&\sum_{t_{j} \leq T} \frac{t_{j}}{\sinh \pi t_{j}} \sum_{0 \ne \mathrm{N}(d) \leq \sqrt{y}} \frac{1}{\mathrm{N}(d)^{2w}} \ \sideset{}{^{\prime}} \sum_{0 \ne \mathrm{N}(n) \leq \frac{y}{\mathrm{N}(d)^{2}}} \frac{|\rho_{j}(n)|^{2}}{\mathrm{N}(n)^{w}}\\
& = \Cr{Raghavan-Sengupta} T^{3} \sum_{0 \ne \mathrm{N}(d) \leq \sqrt{y}} \frac{1}{\mathrm{N}(d)^{2w}} \sum_{0 \ne \mathrm{N}(n) \leq \frac{y}{\mathrm{N}(d)^{2}}} \frac{1}{\mathrm{N}(n)^{w}}+O_{\epsilon}(T^{2} y^{\frac{3}{2}-\Re(w)+\epsilon}+T^{1+\epsilon} y^{2-\Re(w)+\epsilon})\\
& = \Cr{Raghavan-Sengupta} T^{3} \sum_{d \leq \sqrt{y}} \frac{r(d)}{d^{2w}} \sum_{n \leq \frac{y}{d^{2}}} \frac{r(n)}{n^{w}}+O_{\epsilon}(T^{2} y^{\frac{3}{2}-\Re(w)+\epsilon}+T^{1+\epsilon} y^{2-\Re(w)+\epsilon})\\
& = \frac{\Cr{key-identity} T^{3} y^{1-w}}{1-w}+O_{\epsilon}(T^{3} y^{\epsilon}+T^{2} y^{\frac{3}{2}-\Re(w)+\epsilon}+T^{1+\epsilon} y^{2-\Re(w)+\epsilon}),
\end{align}
where we utilised the standard approximation of the Dirichlet series at the vicinity of $w = 1$ with the leading coefficient unspecified, and $r(n)$ stands for the number of representations of $n$ as a sum of two squares. Note that Matthes's computation in~\cites[Page~9]{Matthes1995-2} is~erroneous. By implementing some rearrangements using the Riemann--Stieltjes integral, the contribution of $\mathcal{R}_{it_{j}}(x, w-1, \xi+1, \nu-1)$ amounts to
\begin{align}
&\int_{0}^{x} \int_{0}^{x_{1}} \cdots \int_{0}^{x_{\nu-2}} \left(\int_{0}^{x_{\nu-1}} \frac{\Cr{key-identity} T^{3}}{1-w} x_{\nu} d_{x_{\nu}} x_{\nu}^{1-w} \right) \, dx_{\nu-1} \cdots dx_{2} \, dx_{1}\\
&\quad \hspace{.15em} + O_{\epsilon}(T^{3} x^{\nu+\epsilon}+T^{2} x^{\frac{3}{2}+\nu-\Re(w)+\epsilon}+T^{1+\epsilon} x^{2+\nu-\Re(w)+\epsilon})\\
& = \Cr{key-identity} T^{3} \int_{0}^{x} \int_{0}^{x_{1}} \cdots \int_{0}^{x_{\nu-1}} x_{\nu}^{1-w} \, dx_{\nu} \cdots dx_{2} \, dx_{1}\\
&\quad \hspace{.15em} + O_{\epsilon}(T^{3} x^{\nu+\epsilon}+T^{2} x^{\frac{3}{2}+\nu-\Re(w)+\epsilon}+T^{1+\epsilon} x^{2+\nu-\Re(w)+\epsilon})\\
& = \frac{\Cr{key-identity} T^{3} x^{1+\nu-w}}{(2-w) \cdots (1+\nu-w)}+O_{\epsilon}(T^{3} x^{\nu+\epsilon}+T^{2} x^{\frac{3}{2}+\nu-\Re(w)+\epsilon}+T^{1+\epsilon} x^{2+\nu-\Re(w)+\epsilon}),\label{eq:first-term}
\end{align}
while the contribution of $(w-1) \mathcal{R}_{it_{j}}(x, w, \xi, \nu)$ amounts to
\begin{align}
&(w-1) \int_{0}^{x} \int_{0}^{x_{1}} \cdots \int_{0}^{x_{\nu-1}} \sum_{t_{j} \leq T} \frac{t_{j}}{\sinh \pi t_{j}} \sum_{0 \ne \mathrm{N}(d) \leq \sqrt{x_{\nu}}} \frac{1}{\mathrm{N}(d)^{2w}} \ \sideset{}{^{\prime}} \sum_{0 \ne \mathrm{N}(n) \leq \frac{x_{\nu}}{\mathrm{N}(d)^{2}}} \frac{|\rho_{j}(n)|^{2}}{\mathrm{N}(n)^{w}} \, dx_{\nu} \cdots dx_{2} \, dx_{1}\\
& = -\Cr{key-identity} T^{3} \int_{0}^{x} \int_{0}^{x_{1}} \cdots \int_{0}^{x_{\nu-1}} x_{\nu}^{1-w} \, dx_{\nu} \cdots dx_{2} \, dx_{1}\\
& \quad \hspace{.15em} + O_{\epsilon}(T^{3} x^{\nu+\epsilon}+T^{2} x^{\frac{3}{2}+\nu-\Re(w)+\epsilon}+T^{1+\epsilon} x^{2+\nu-\Re(w)+\epsilon})\\
& = -\frac{\Cr{key-identity} T^{3} x^{1+\nu-w}}{(2-w) \cdots (1+\nu-w)}+O_{\epsilon}(T^{3} x^{\nu+\epsilon}+T^{2} x^{\frac{3}{2}+\nu-\Re(w)+\epsilon}+T^{1+\epsilon} x^{2+\nu-\Re(w)+\epsilon}).\label{eq:second-term}
\end{align}
Summing up~\eqref{eq:first-term} and~\eqref{eq:second-term} completes the proof of \cref{lem:Matthes}.
\end{proof}

On the other hand, to address the first two terms on the right-hand side of~\eqref{eq:key-identity}, we~utilise the functional equation~\eqref{eq:functional-equation-2}, so that $\mathcal{R}_{it_{j}}(x, w, \mu, \nu)$ is equal to
\begin{multline}
\int_{(\mu)} \frac{\gamma(1-s-w, u_{j} \otimes u_{j})}{\gamma(s+w, u_{j} \otimes u_{j})} L^{\diamond}(1-s-w, u_{j} \otimes u_{j}) \frac{x^{s+\nu}}{s(s+1) \cdots (s+\nu)} \frac{ds}{2\pi i}\\
 = \int_{(\mu)} \frac{\gamma(1-s-w, u_{j} \otimes u_{j})}{\gamma(s+w, u_{j} \otimes u_{j})} \sum_{0 \ne n \in \mathbb{Z}[i]} \mathrm{N}(n)^{s+w-1} \sum_{d^{2} \mid n} \left|\rho_{j} \left(\frac{n}{d^{2}} \right) \right|^{2} \frac{x^{s+\nu}}{s(s+1) \cdots (s+\nu)} \frac{ds}{2\pi i}
\end{multline}
for some $\mu \in (\frac{1}{2}-\frac{\nu}{2}-\Re(w), -\Re(w))$. For a parameter $N \geq 1$ to be chosen later, Stirling's formula~\eqref{eq:Stirling} demonstrates that the right-hand side is approximated by
\begin{multline}
\int_{(\mu)} \frac{\gamma(1-s-w, u_{j} \otimes u_{j})}{\gamma(s+w, u_{j} \otimes u_{j})} \sum_{0 \ne \mathrm{N}(n) \leq N} \mathrm{N}(n)^{s+w-1} \sum_{d^{2} \mid n} \left|\rho_{j} \left(\frac{n}{d^{2}} \right) \right|^{2} \frac{x^{s+\nu}}{s(s+1) \cdots (s+\nu)} \frac{ds}{2\pi i}\\
 + O_{\delta}((1+t_{j})^{2-4\mu-4\Re(w)} (1+|\Im(w)|)^{2-4\mu-4\Re(w)} x^{\mu+\nu} N^{\mu+\Re(w)} L^{\diamond}(1+\delta, u_{j} \otimes u_{j})),
\end{multline}
where the replacement of the exponent $s+w-1$ with $-1-\delta$ is guaranteed at the expense of introducing a negligibly small error term, thereby resulting in the factor of $L^{\diamond}(1+\delta, u_{j} \otimes u_{j})$. By arguing analogously, $\mathcal{R}_{it_{j}}(x, w-1, \mu+1, \nu-1)$ is equal to (cf.~\cites[Equation~(8)]{Matthes1995-2})
\begin{multline}
\int_{(\mu+1)} \frac{\gamma(2-s-w, u_{j} \otimes u_{j})}{\gamma(s+w-1, u_{j} \otimes u_{j})} \sum_{0 \ne \mathrm{N}(n) \leq N} \mathrm{N}(n)^{s+w-2} \sum_{d^{2} \mid n} \left|\rho_{j} \left(\frac{n}{d^{2}} \right) \right|^{2} \frac{x^{s+\nu-1}}{s(s+1) \cdots (s+\nu-1)} \frac{ds}{2\pi i}\\
 + O_{\delta}((1+t_{j})^{2-4\mu-4\Re(w)} (1+|\Im(w)|)^{2-4\mu-4\Re(w)} x^{\mu+\nu} N^{\mu+\Re(w)} L^{\diamond}(1+\delta, u_{j} \otimes u_{j})).
\end{multline}

By shifting the contour $(\mu)$ to the right in the spirit of Matthes~\cites[Pages~10--11]{Matthes1995-2}~and gathering everything together with
\begin{equation}
x = T(1+|\Im(w)|)^{2}, \qquad \Re(w) = \frac{1}{2}, \qquad \mu = 1, \qquad \nu = 4, \qquad N = T^{2}(1+|\Im(w)|)^{4},
\end{equation}
it follows from the key identity~\eqref{eq:key-identity} and \cref{prop:Iwaniec-consult} with $\delta = (\log T)^{-1}$ that
\begin{align}
&(w-1) \sum_{t_{j} \leq T} \frac{t_{j}}{\sinh \pi t_{j}} |L^{\diamond}(w, u_{j} \otimes u_{j})|\\
& \ll_{\epsilon} (1+|\Im(w)|)^{2+\epsilon} T^{3+\epsilon}+(1+|\Im(w)|)^{4+\epsilon} \sum_{t_{j} \leq T} \frac{t_{j}}{\sinh \pi t_{j}} |L^{\diamond}(1+\delta, u_{j} \otimes u_{j})|\\
& \ll_{\epsilon} (1+|\Im(w)|)^{4+\epsilon} T^{3+\epsilon}.
\end{align}
By dividing both sides by $w-1$, the proof of \cref{thm:mean-Lindelof} is thus complete. \qed

\section{Brun--Titchmarsh-type theorem}\label{sect:Iwaniec--Bykovskii}
The Iwaniec--Bykovski\u{\i} formula in the $3$-dimensional setting serves as a counterpart of the classical Brun--Titchmarsh theorem in short intervals and forms the cornerstone for obtaining a polynomial power-saving improvement to surpass the threshold in the nonsmooth explicit formula due to Nakasuji~\cites[Theorem~5.2]{Nakasuji2000}[Theorem~4.1]{Nakasuji2001}. The interested reader is directed to~\cites{Bykovskii1994}{ChatzakosHarcosKaneko2024}{CherubiniWuZabradi2022}{Golubeva1997}{Iwaniec1984}{Kuznetsov1978}{SoundararajanYoung2013} and references therein~for the $2$-dimensional archetype, its variations, and its refinement under the generalised Riemann hypothesis for quadratic Dirichlet $L$-functions.

For the Picard group $\Gamma = \mathrm{PSL}_{2}(\mathbb{Z}[i])$, the record\footnote{The conditional asymptotic~\cites[Equation~(10)]{BalogBiroCherubiniLaaksonen2022} remains weaker than~\eqref{eq:BBCL22-9} as soon as $x^{\frac{1}{2}+\epsilon} \leq y \leq x$.} of Balog et al.~\cites[Equation~(9)]{BalogBiroCherubiniLaaksonen2022} asserts that for $x^{\frac{1+4\vartheta}{3}+\epsilon} \leq y \leq x$,
\begin{equation}\label{eq:BBCL22-9}
\Psi_{\Gamma}(x+y)-\Psi_{\Gamma}(x) = xy+\frac{y^{2}}{2}+O_{\epsilon}(x^{\frac{2(3+2\vartheta)}{5}+\epsilon} y^{\frac{2}{5}}).
\end{equation}
This corresponds to the record of Soundararajan and Young~\cites[Equation~(16)]{SoundararajanYoung2013} (see also the first display\footnote{The term $x^{\frac{1}{2}+\frac{\vartheta}{2}+\epsilon}$ in the first display of the errata reads $x^{\frac{1}{2}+\vartheta+\epsilon}$, which aligns with our threshold in~\eqref{eq:SY13-16}.} in its errata) for the full modular group $\Gamma = \mathrm{PSL}_{2}(\mathbb{Z})$, which asserts that~for $x^{\frac{1}{2}+\vartheta+\epsilon} \leq y \leq x$,
\begin{equation}\label{eq:SY13-16}
\Psi_{\Gamma}(x+y)-\Psi_{\Gamma}(x) = y+O_{\epsilon}(x^{\frac{1}{4}+\frac{\vartheta}{2}+\epsilon} y^{\frac{1}{2}}),
\end{equation}
where one may abuse notation to enable its provisional adaptation to $\mathbb{Q}$. In the $2$-dimensional setting, there exist infinitely many pairs $(x, y)$ such that the left-hand side of~\eqref{eq:SY13-16} vanishes if $y \leq x^{\frac{1}{2}}$; hence, the threshold of Soundararajan and Young is provably optimal. On the other hand, up to the knowledge towards the generalised Lindel\"{o}f hypothesis, the threshold~in~\eqref{eq:BBCL22-9} represents the suboptimal exponent $\frac{1}{3}+\epsilon$. Furthermore, if one intends to retain the secondary main term contrary to the approach of Balog et al.~\cites[Page~1897]{BalogBiroCherubiniLaaksonen2022}, then the threshold shrinks to $x^{\frac{3+2\vartheta}{4}+\epsilon} \leq y \leq x$. If we refer to the former threshold as global, then the conjectural global threshold ought to be $x^{\epsilon} \leq y \leq x$ (cf.~\cites[Remark~3]{BalogBiroCherubiniLaaksonen2022}), whereas the appropriate local threshold remains unclear at the current state, commensurate with the yet unresolved question concerning the optimal exponent in the prime geodesic theorem over $\Gamma \backslash \mathbb{H}^{3}$.

To facilitate the ensuing considerations, suppose that the Iwaniec--Bykovski\u{\i} formula~\eqref{eq:BBCL22-9} holds with an error term of size $O(x^{\kappa} y^{\lambda})$ with $\kappa+\lambda < 2$. For the main term $xy$ to dominate the error term, one must restrict $x^{\frac{1-\kappa}{1-\lambda}} \leq y \leq x$, with a particular emphasis on the boundary regime $y \asymp x^{\frac{1-\kappa}{1-\lambda}}$. On the other hand, if one conditionally equates $x^{\kappa} y^{\lambda} = x^{\frac{3}{2}} y^{-\frac{1}{2}}$ under the assumption of the square-root cancellation in the spectral exponential sum and resolves the optimisation problem with respect to $y$, then this implies the optimal quantities
\begin{equation}\label{eq:seemingly}
\frac{1-\kappa}{1-\lambda} = \frac{3-2\kappa}{1+2\lambda} \qquad \Longleftrightarrow \qquad (\kappa, \lambda) = \left(\frac{5\lambda-2}{4\lambda-1}, \frac{2-\kappa}{5-4\kappa} \right)
\end{equation}
for \textit{any} admissible pair $(\kappa, \lambda)$. For example, substituting $\lambda = \frac{1}{2}$ into~\eqref{eq:seemingly} implies $\kappa = \frac{1}{2}$ and $y \asymp x$, and hence $\mathcal{E}_{\Gamma}(x) \ll_{\epsilon} x^{1+\epsilon}$ at least heuristically. This marks the first occasion~wherein the ostensibly optimal exponent $1+\epsilon$ is explicitly exemplified in accordance with the $\Omega$-result of Nakasuji~\cites[Theorem~1.2]{Nakasuji2000}[Theorem~1.1]{Nakasuji2001}.

This section is devoted to improving upon the best known result~\eqref{eq:BBCL22-9}. This result in turn helps establish \cref{thm:unconditional,thm:second-moment,thm:conditional} in one fell swoop.
\begin{theorem}\label{thm:Brun-Titchmarsh}
Let $\Gamma = \mathrm{PSL}_{2}(\mathbb{Z}[i])$, and let $x^{\frac{1+2\vartheta}{3}+\epsilon} \leq y \leq x$. Then we have that
\begin{equation}
\Psi_{\Gamma}(x+y)-\Psi_{\Gamma}(x) = xy+\frac{y^{2}}{2}+O_{\epsilon}(x^{1+\vartheta+\epsilon} y^{\frac{1}{2}}+x^{\frac{5}{4}+\frac{\vartheta}{2}+\epsilon} y^{\frac{1}{4}}).
\end{equation}
In particular, assuming the generalised Lindel\"{o}f hypothesis $\vartheta = 0$ in the conductor aspect,~we have for any $x^{\frac{1}{3}+\epsilon} \leq y \leq x$ and $\epsilon > 0$ that
\begin{equation}
\Psi_{\Gamma}(x+y)-\Psi_{\Gamma}(x) = xy+\frac{y^{2}}{2}+O_{\epsilon}(x^{\frac{5}{4}+\epsilon} y^{\frac{1}{4}}).
\end{equation}
\end{theorem}

In preparation for the optimisation problem that we encounter in the proof of \cref{thm:Brun-Titchmarsh}, it is convenient~to refine~\cites[Lemma~2.2]{BalogBiroCherubiniLaaksonen2022} through a delicate analysis of Bessel functions.
\begin{proposition}\label{prop:Bessel}
Let $0 \ne c \in \mathbb{Z}[i]$, and let $c_{\star}$ be maximal such that $c_{\star}^{2} \mid c$. Then we have for any $Z \geq 1$ and $\epsilon > 0$ that
\begin{align}
\sum_{\mathrm{N}(n) \leq Z} \rho_{c}(n^{2}-4) &= \pi Z \frac{\varphi(c)}{\mathrm{N}(c)}+O_{\epsilon}((\mathrm{N}(c)^{\frac{1}{2}+\epsilon}+\mathrm{N}(c)^{\frac{1}{4}+\epsilon} Z^{\frac{1}{4}+\epsilon}) \mathrm{N}(c_{\star})),\label{eq:rho-c-n}\\
\sum_{\mathrm{N}(n) \leq Z} \lambda_{c}(n^{2}-4) &= \pi Z \sum_{c_{1}^{2} c_{2} = c} \frac{\mu(c_{2})}{\mathrm{N}(c_{2})}+O_{\epsilon}((\mathrm{N}(c)^{\frac{1}{2}+\epsilon}+\mathrm{N}(c)^{\frac{1}{4}+\epsilon} Z^{\frac{1}{4}+\epsilon}) \mathrm{N}(c_{\star})).
\end{align}
\end{proposition}

\begin{proof}
The proof strategy adapts that of Balog et al.~\cites[Pages~1901--1903]{BalogBiroCherubiniLaaksonen2022} verbatim, up to the treatment of the dual summation after the application of the $2$-dimensional Poisson summation formula. For completeness and comparison, part of their argument is reproduced here with some rearrangements especially to avoid notational conflict.

Since the claim is immediate from~\eqref{eq:uniformly} when $\mathrm{N}(c) \geq Z^{2}$, we shall assume that $\mathrm{N}(c) < Z^{2}$ in what follows. Let $\ast$ denote the standard convolution on $\mathbb{R}^{2}$, and we fix a test function
\begin{equation}
v(x) \coloneqq \frac{1}{\pi \Delta^{2}}(\mathbf{1}_{[0, Z^{\frac{1}{2}}]} \ast \mathbf{1}_{[0, \Delta]})(|x|)
\end{equation}
for an auxiliary variable $\Delta \in (Z^{-\frac{1}{2}}, Z^{\frac{1}{2}})$ at our discretion. By the compatibility relation~\eqref{eq:rho-lambda} and \cref{lem:identify}, it suffices to focus on the first assertion for $\rho(c, n)$. We begin by smoothing the expression on the left-hand side of~\eqref{eq:rho-c-n} as
\begin{equation}\label{eq:smoothing}
\sum_{\mathrm{N}(n) \leq Z} \rho(c, n) = \sum_{n \in \mathbb{Z}[i]} \rho(c, n) v(n)+O_{\epsilon}(\mathrm{N}(c)^{\epsilon} \mathrm{N}(c_{\star}) \Delta Z^{\frac{1}{2}}),
\end{equation}
where the error term arises from estimating (cf.~\eqref{eq:uniformly})
\begin{align}
\left|\sum_{\mathrm{N}(n) \leq Z} \rho(c, n)-\sum_{n \in \mathbb{Z}[i]} \rho(c, n) v(n) \right| &\ll_{\epsilon} \mathrm{N}(c)^{\epsilon} \mathrm{N}(c_{\star}) \sum_{n \in \mathbb{Z}[i]} \mathbf{1}_{[Z^{\frac{1}{2}}-\Delta, Z^{\frac{1}{2}}+\Delta]}(|n|)\\
&\ll_{\epsilon} \mathrm{N}(c)^{\epsilon} \mathrm{N}(c_{\star}) \Delta Z^{\frac{1}{2}}.\label{eq:error-approximate}
\end{align}
Decomposing the summation over $n$ in~\eqref{eq:smoothing} into residue classes and applying $2$-dimensional Poisson summation\footnote{This may in principle be considered Vorono\u{\i} summation associated to $r(n)$, the number of representations of $n$ as a sum of two squares.} in conjunction with \cref{lem:identify,lem:counting-function}, we now deduce
\begin{align}
\sum_{n \in \mathbb{Z}[i]} \rho(c, n) v(n) &= \frac{1}{\mathrm{N}(c)} \sum_{b \tpmod{c}} \rho(c, b) \sum_{n \in \mathbb{Z}[i]} \check{e}_{c}(bn) \hat{v} \left(\frac{nc}{\mathrm{N}(c)} \right)\\
& = \frac{\varphi(c)}{\mathrm{N}(c)} \pi Z+\frac{1}{\mathrm{N}(c)} \sum_{0 \ne n \in \mathbb{Z}[i]} S(n, n, c) \hat{v} \left(\frac{nc}{\mathrm{N}(c)}\right),\label{eq:Voronoi}
\end{align}
since $\hat{v}(0) = \pi Z$ and $S(0, 0, c) = \varphi(c)$ for $n = 0$ and $0 \ne c \in \mathbb{Z}[i]$.

If $n \ne 0$, then one may approximate the integral transform $\hat{v}$ and identify the leading~term followed by negligibly small lower order terms, thereby deviating from the approach of Balog et al.~\cites[Page~1902]{BalogBiroCherubiniLaaksonen2022}. The antepenultimate display therein reads
\begin{equation}\label{eq:antepenultimate}
\hat{v}(x) = \frac{Z^{\frac{1}{2}}}{\pi \Delta|x|^{2}} J_{1}(2\pi Z^{\frac{1}{2}}|x|) J_{1}(2\pi \Delta|x|),
\end{equation}
where $J_{\nu}$ denotes the Bessel function of the first kind. Inserting~\eqref{eq:J-1} into~\eqref{eq:antepenultimate} and invoking the condition $\Delta \in (Z^{-\frac{1}{2}}, Z^{\frac{1}{2}})$ implies
\begin{equation}
\hat{v}(x) \sim 
	\begin{dcases}
	\pi Z & \text{if $|x| < Z^{-\frac{1}{2}}$},\\
	\frac{Z^{\frac{1}{4}}}{\pi|x|^{\frac{3}{2}}} \cos \left(2\pi Z^{\frac{1}{2}} |x|-\frac{3\pi}{4} \right) & \text{if $Z^{-\frac{1}{2}} \leq |x| < \Delta^{-1}$},\\
	\frac{Z^{\frac{1}{4}}}{\pi^{3} \Delta^{\frac{3}{2}}|x|^{3}} \cos \left(2\pi Z^{\frac{1}{2}}|x|-\frac{3\pi}{4} \right) \cos \left(2\pi \Delta|x|-\frac{3\pi}{4} \right) &\text{if $|x| \geq \Delta^{-1}$}.
	\end{dcases}
\end{equation}
In the initial range $|x| < Z^{-\frac{1}{2}}$, the Weil--Gundlach bound~\eqref{eq:Weil} yields
\begin{equation}\label{eq:3/2}
\sum_{\mathrm{N}(n) \ll \mathrm{N}(c) Z^{-1}} S(n, n, c) \hat{v} \left(\frac{nc}{\mathrm{N}(c)} \right) \ll_{\epsilon} \mathrm{N}(c)^{\frac{3}{2}+\epsilon},
\end{equation}
where $\tau(c) \ll_{\epsilon} \mathrm{N}(c)^{\epsilon}$ and the greatest common divisor $(n, c)$ is bounded on average by~\eqref{eq:bounded-on-average}. Since the estimation of the contribution from the tail range $|x| \geq \Delta^{-1}$ ultimately reduces to that of the bulk range $Z^{-\frac{1}{2}} \leq |x| < \Delta^{-1}$, it suffices to focus on the latter. By trigonometric computations, dyadic subdivisions, and partial summation, we are then led to the problem~of determining nontrivial bounds for an expression of the shape
\begin{equation}\label{eq:after-trigonometric}
\Delta^{\frac{3}{2}} Z^{\frac{1}{4}} \sum_{\mathrm{N}(n) \asymp \mathrm{N}(c) \Delta^{-2}} S(n, n, c) e \bigg(\frac{|n|Z^{\frac{1}{2}}}{|c|} \bigg),
\end{equation}
where we temporarily neglected the lower bound $\mathrm{N}(n) \gg \mathrm{N}(c) Z^{-1}$ for brevity. Using here the Weil--Gundlach bound~\eqref{eq:Weil} reproduces the second term on the right-hand side of~\cites[Equation~(23)]{BalogBiroCherubiniLaaksonen2022}. Alternatively, an application of $2$-dimensional Poisson summation backwards in~\eqref{eq:after-trigonometric} rules out potential tautology, since the range of summation is nontrivially elongated to $\mathrm{N}(n) \asymp \mathrm{N}(c) \Delta^{-2}$. In particular, if the summation originally had length $\mathrm{N}(n) \asymp \mathrm{N}(c) Z^{-1}$, then the additional exponential factor in~\eqref{eq:after-trigonometric} would become flattened and lose its oscillatory nature, thereby precluding the application of $2$-dimensional Poisson summation. Our success is in turn attributed to the conductor-dropping phenomenon in the range $\mathrm{N}(n) \asymp \mathrm{N}(c) \Delta^{-2}$.

Let $w$ be a smooth and compactly supported function supported on $[1, 2]$. By $2$-dimensional Poisson summation and estimating the dual summation roughly of length $\max(1, \Delta^{2})$ using \cref{lem:identify,lem:counting-function} and the asymptotic formul{\ae}~\cites[Equations~(8.440) and~(8.451.1)]{GradshteynRyzhik2014}
\begin{equation}\label{eq:J-0}
J_{0}(y) \sim 
	\begin{dcases}
	1 & \text{as $y \to 0$},\\
	\sqrt{\frac{2}{\pi y}} \cos \left(y-\frac{\pi}{4} \right) & \text{as $y \to \infty$},
	\end{dcases}
\end{equation}
the resulting expression boils down to
\begin{align}
&\Delta^{\frac{3}{2}} Z^{\frac{1}{4}} \sum_{n \in \mathbb{Z}[i]} S(n, n, c) e \bigg(\frac{|n|Z^{\frac{1}{2}}}{|c|} \bigg) w \bigg(\frac{|n| \Delta}{|c|} \bigg)\\
& = 2\pi \mathrm{N}(c) \Delta^{-\frac{1}{2}} Z^{\frac{1}{4}} \sum_{n \in \mathbb{Z}[i]} \rho(c, n) \int_{0}^{\infty} w(x) e \bigg(\frac{xZ^{\frac{1}{2}}}{\Delta} \bigg) J_{0} \bigg(\frac{2\pi |n|x}{\Delta} \bigg) x \, dx\\
&\ll \mathrm{N}(c) \Delta^{-\frac{1}{2}} Z^{\frac{1}{4}} \sum_{\mathrm{N}(n) \ll \max(1, \Delta^{2})} \rho(c, n) \left|\int_{0}^{\infty} w(x) e \bigg(\frac{xZ^{\frac{1}{2}}}{\Delta} \bigg) x \, dx \right|\\
&\ll \mathrm{N}(c)^{1+\epsilon} \max(\Delta^{\frac{1}{2}}, \Delta^{\frac{5}{2}}) Z^{-\frac{1}{4}}.\label{eq:dual}
\end{align}
Substituting~\eqref{eq:3/2} into~\eqref{eq:Voronoi}, interpolating~\eqref{eq:dual} with the second term on the right-hand side of~\cites[Equation~(23)]{BalogBiroCherubiniLaaksonen2022}, and applying the inequality $\min(A, B) \leq A^{\gamma} B^{1-\gamma}$ for $A, B > 0$ and $\gamma \in (0, 1)$, we obtain
\begin{align}
\left|\sum_{n \in \mathbb{Z}[i]} \rho(c, n) v(n)-\frac{\varphi(c)}{\mathrm{N}(c)} \pi Z \right| &\ll_{\epsilon} \mathrm{N}(c)^{\frac{1}{2}+\epsilon}+\min(\mathrm{N}(c)^{\frac{1}{2}+\epsilon} \Delta^{-\frac{1}{2}} Z^{\frac{1}{4}}, \mathrm{N}(c)^{\epsilon} \max(\Delta^{\frac{1}{2}}, \Delta^{\frac{5}{2}}) Z^{-\frac{1}{4}})\\
&\ll_{\epsilon} \mathrm{N}(c)^{\frac{1}{2}+\epsilon}+\mathrm{N}(c)^{\frac{\gamma}{2}+\epsilon} \Delta^{\frac{5}{2}-3\gamma} Z^{\frac{\gamma}{2}-\frac{1}{4}}+\mathrm{N}(c)^{\frac{\gamma}{2}+\epsilon} \Delta^{\frac{1}{2}-\gamma} Z^{\frac{\gamma}{2}-\frac{1}{4}}.\label{eq:interpolate}
\end{align}
Substituting~\eqref{eq:interpolate} into~\eqref{eq:smoothing} and optimising $\Delta = \mathrm{N}(c)^{\frac{\gamma}{3(2\gamma-1)}} Z^{-\frac{3-2\gamma}{6(2\gamma-1)}} \in (Z^{-\frac{1}{2}}, Z^{\frac{1}{2}})$ and~$\gamma = \frac{3}{4}$ yields
\begin{equation}
\sum_{\mathrm{N}(n) \leq Z} \rho(c, n) = \frac{\varphi(c)}{\mathrm{N}(c)} \pi Z+O_{\epsilon}((\mathrm{N}(c)^{\frac{1}{2}+\epsilon}+\mathrm{N}(c)^{\frac{1}{4}+\epsilon} Z^{\frac{1}{4}+\epsilon}) \mathrm{N}(c_{\star})),
\end{equation}
which concludes the proof of the first claim of \cref{prop:Bessel}. As stated earlier, the second claim follows from M\"{o}bius inversion, namely
\begin{equation}
\sum_{c_{1}^{2} c_{2} c_{3} = c} \frac{\mu(c_{2}) \varphi(c_{3})}{\mathrm{N}(c_{3})} = \sum_{c_{1}^{2} c_{2} \mid c} \frac{\mu(c_{2})}{\mathrm{N}(c_{2})} \sum_{c_{3} \mid c/c_{1}^{2} c_{2}} \mu(c_{3}) = \sum_{c_{1}^{2} c_{2} = c} \frac{\mu(c_{2})}{\mathrm{N}(c_{2})},
\end{equation}
as required.
\end{proof}

It is convenient to emphasise a distinction from the ideas of Soundararajan and Young~\cites[Page~111--112]{SoundararajanYoung2013}. By the Selberg--Kuznetsov identity over number fields due to Pacharoni~\cites[Theorem]{Pacharoni1997} (which is also a consequence of the deduction \`{a} la Matthes~\cites[Theorem~1.3]{Matthes1990}), the counterpart of the first display of~\cites[Page~112]{SoundararajanYoung2013}~reads
\begin{equation}
S(z, z, c_{3}) = \sum_{d \mid (c_{3}, z)} \mathrm{N}(d) S \left(\frac{z^{2}}{d^{2}}, 1, \frac{c_{3}}{d} \right).
\end{equation}
The analogue of~\cites[Lemma~2.3]{SoundararajanYoung2013} over $\mathbb{Q}(i)$ is then given by
\begin{align}
\rho_{c}(n^{2}-4) &= \frac{1}{\mathrm{N}(c)} \sum_{h \tpmod{c}} \check{e} \left(\frac{hn}{c} \right) S(h, h, c),\\
\lambda_{c}(n^{2}-4) &= \sum_{c_{1}^{2} c_{2} = c} \frac{1}{\mathrm{N}(c_{2})} \sum_{h \tpmod{c_{2}}} \check{e} \left(\frac{hn}{c_{2}} \right) S(h^{2}, 1, c_{2}).
\end{align}
A serious bottleneck arises when attempting to replicate the fourth display on~\cites[Page~112]{SoundararajanYoung2013}, namely the trivial bound for the error term in~\eqref{eq:rho-c-n} represents
\begin{equation}\label{eq:Z-1/2}
\sum_{\mathrm{N}(n) \leq Z} \rho_{c}(n^{2}-4) = \pi Z \frac{\varphi(c)}{\mathrm{N}(c)}+O_{\epsilon}(\mathrm{N}(c)^{\frac{1}{2}+\epsilon} \mathrm{N}(c_{\star}) Z^{\frac{1}{2}}),
\end{equation}
where the factor $Z^{\frac{1}{2}}$ reflects the radial growth of lattice points in the Gaussian integers.~Since the radius of the disk increases, the number of terms in $\sum_{\mathrm{N}(n) \leq Z} \check{e}(\frac{hn}{c})$ scales proportionally with the circumference of concentric circles. This phenomenon necessitates the application of $2$-dimensional Poisson summation together with a sophisticated analysis of Bessel functions, as described above, thereby enabling the complete unconditional removal of the factor of $Z^{\frac{1}{2}}$ in~\eqref{eq:Z-1/2} at the expense introducing an additional contribution of comparable magnitude.

\begin{proof}[Proof of \cref{thm:Brun-Titchmarsh}]
The proof strategy adapts that of Balog et al.~\cites[Section~3.4]{BalogBiroCherubiniLaaksonen2022}. Upon renaming the variables, the third display on~\cites[Page~1913]{BalogBiroCherubiniLaaksonen2022} translates to
\begin{align}
\left|\Psi_{\Gamma}(x+y)-\Psi_{\Gamma}(x)-xy-\frac{y^{2}}{2} \right| &\ll_{\epsilon} V^{-\frac{1}{2}} x^{1+2\vartheta+\epsilon} y+V^{\frac{1}{2}+\epsilon} x+V^{\frac{1}{4}+\epsilon} x^{\frac{5}{4}}\\
&\ll_{\epsilon} x^{1+\vartheta+\epsilon} y^{\frac{1}{2}}+x^{\frac{5}{4}+\frac{\vartheta}{2}+\epsilon} y^{\frac{1}{4}},\label{eq:renaming}
\end{align}
where $V = x^{2\vartheta} y$ with $x^{\frac{1+2\vartheta}{3}+\epsilon} \leq y \leq x$. The proof of \cref{thm:Brun-Titchmarsh} is~thus complete.
\end{proof}

\section{Exponent pairs in limiting regime}\label{sect:quadratic}
The seminal work of Burgess~\cites{Burgess1962-2}{Burgess1963}{Burgess1986} makes extensive use of the periodicity of the summand $n \mapsto \chi(n)$; hence, its straightforward algebraic generalisation, summing over integral ideals whose norm lies in a fixed interval, loses the periodicity. Nonetheless, the study of short character sums over arbitrary algebraic number fields as developed by Hinz~\cites{Hinz1983}{Hinz1983-2}{Hinz1986} represents an extension of the Burgess bound~\cites[Theorem~12.6]{IwaniecKowalski2004}. Over the rationals, the author~\cites[Lemma~2.1]{Kaneko2024} incorporates hybrid subconvex bounds for Dirichlet $L$-functions with the Burgess bound to demonstrate a bootstrapped variant in long intervals, involving an auxiliary parameter $r \in \mathbb{N}_{\geq 2}$ at our discretion. The proof of~\cref{thm:conditional}~relies exclusively on the limiting regime $r \to \infty$, thereby necessitating solely a trivial treatment of character sums in conjunction with the assumed subconvex bounds. While it falls within the realm of plausible adjustments, generalising our pursuits to arbitrary algebraic number fields, while retaining the auxiliary parameter $r \in \mathbb{N}_{\geq 2}$, would entail brute force computations.

We now introduce the notion of multiplicative exponent pairs over $\mathbb{Q}(i)$.
\begin{definition}\label{def:alpha-beta}
Let $D \in \mathbb{Z}[i]$ be a generator of the fundamental discriminant of a quadratic extension of $\mathbb{Q}(i)$, and let $(\alpha, \beta) \in \mathbb{R}^{2}$ belong to the set $\{0 \leq \beta \leq \frac{1}{2} \leq \alpha \leq 1 \} \cup \{(0, \frac{1}{2}), (1, 0) \}$. Then $(\alpha, \beta)$ is called a quadratic exponent pair over $\mathbb{Q}(i)$ if we have for any $1 \leq y \leq x$ and $\epsilon > 0$ that
\begin{equation}\label{eq:alpha-beta}
\sum_{x < \mathrm{N}(n) \leq x+y} \chi_{D}(n) \ll_{\epsilon} y^{\alpha} \mathrm{N}(D)^{\beta+\epsilon}.
\end{equation}
If~\eqref{eq:alpha-beta} holds irrespective of quadraticity, then $(\alpha, \beta)$ is called a multiplicative exponent pair over $\mathbb{Q}(i)$.
\end{definition}

The special case of $n = 2$ in~\cites[Theorem~1]{Hinz1986} confirms an analogue of Burgess's exponent pair $(\alpha, \beta) = (1-\frac{1}{r}, \frac{r+1}{4r^{2}})$ over the rationals~\cites[Theorem~12.6]{IwaniecKowalski2004} for any fixed integer~$r \in \mathbb{N}_{\geq 2}$, with the cubefree condition therein eliminated through a fundamental property of quadratic characters. Nonetheless, Hinz imposes some cumbersome constraints on the summation over $n$, rendering the result not immediately applicable to our scenario; see~\cites[Equation~(1)]{Landau1918} for a generalisation of the P\'{o}lya--Vinogradov inequality~\cites[Theorem~12.5]{IwaniecKowalski2004} to any algebraic number field over $\mathbb{Q}$, on par with the exponent pair $(\alpha, \beta) = (0, \frac{1}{2})$ in \cref{def:alpha-beta}. Another quadratic exponent pair over $\mathbb{Q}(i)$ constitutes
\begin{equation}
(\alpha, \beta) = \left(\frac{1}{2}, \vartheta \right),
\end{equation}
where $\vartheta \in [0, \frac{1}{4})$ denotes a subconvex exponent for quadratic Dirichlet $L$-functions over $\mathbb{Q}(i)$ in the conductor aspect, as defined in~\eqref{eq:conductor-aspect}; cf.~\cites[Remark~2.4]{Kaneko2024}. Moreover, the result of Balog et al.~\cites[Lemma~2.4]{BalogBiroCherubiniLaaksonen2022} leads to a conditional quadratic exponent pair in harmony with~\cites[Lemma~2.5]{Kaneko2024}. The following reduction leads to the notion of \textit{dual}~multiplicative exponent pairs\footnote{The first appearance of such duality over $\mathbb{Q}$ can be traced back to the work of Vinogradov~\cites{Vinogradov1965},~whose omission was rectified by Graev and Karatsuba~\cites[Theorem~1]{GaraevKaratsuba2005}. The additional lattice structure exhibited in the Gaussian integers may introduce complexities to establish the genuine counterpart over $\mathbb{Q}(i)$.} over $\mathbb{Q}(i)$. Since it is not requisite in a narrow sense at the core of our ideas, we elide the proof and relegate relevant pursuits to future work.
\begin{lemma}\label{lem:dual}
If $(\alpha, \beta)$ is a multiplicative exponent pair, then so is $(1-\alpha, \alpha+\beta-\frac{1}{2})$.
\end{lemma}

In what follows, we focus on the case of $x = y$ in \cref{def:alpha-beta} and refer to a multiplicative exponent pair simply as a quadratic exponent pair for simplicity, even though the majority~of the ensuing discussion remains valid irrespective of quadraticity. We are prepared to establish the $2$-fold version of the author's result~\cites[Lemma~2.3]{Kaneko2024}, with a particular emphasis on the fact that the proof is independent of the degree of the underlying algebraic number field.
\begin{proposition}\label{prop:exponent-pair}
Keep the notation as above. Then the quadratic exponent pair
\begin{equation}\label{eq:exponent-pair}
(\alpha, \beta) = \left(1-\frac{1}{2(1+\vartheta^{\prime})}, \frac{\vartheta}{1+\vartheta^{\prime}} \right)
\end{equation}
is admissible.
\end{proposition}

\begin{remark}
We highlight that \cref{prop:exponent-pair} becomes nontrivial when $x \geq \mathrm{N}(D)^{2\vartheta+\epsilon}$,~yet fails to improve upon the landmark Burgess threshold $x \geq \mathrm{N}(D)^{\frac{1}{4}+\epsilon}$ unless $\vartheta \in [0, \frac{1}{8})$. If~the generalised Lindel\"{o}f hypothesis is assumed, then the optimal quadratic exponent pair $(\frac{1}{2}, 0)$ would follow. It is of independent interest to examine other quadratic exponent pairs in light of arithmetic exponent pairs hitherto discovered~\cites{GrahamKolesnik1991}.
\end{remark}

\begin{proof}
The argument adapts the proofs of~\cites[Lemmata~2.1 and~2.3]{Kaneko2024}. Let $0 < h < x$ be a parameter at our discretion. Fix a smooth compactly supported weight function $w$~satisfying $w(t) = 1$ for $0 \leq \mathrm{N}(t) \leq x$, $w(t) = 0$ for $\mathrm{N}(t) \geq x+h$, and the derivative bounds $w^{(\ell)}(t) \ll_{\ell} h^{-\ell}$ for each $j \in \mathbb{N}_{0}$. Using the notation of Petrow and Young~\cites[Page~1889]{PetrowYoung2023}, we define
\begin{equation}\label{eq:S-chi-w}
S(\chi_{D}, w) \coloneqq \sum_{0 \ne n \in \mathbb{Z}[i]} \chi_{D}(n) w(n) = \int_{(\sigma)} \widetilde{w}(s) L(s, \chi_{D}) \frac{ds}{2\pi i}.
\end{equation}
Integration by parts enables the truncation of the integral at $\Im(s) \ll (\frac{x}{h})^{1+\epsilon}$. Taking~$\sigma = \frac{1}{2}$ and invoking~\eqref{eq:conductor-aspect} then yields
\begin{equation}
S(\chi_{D}, w) \ll_{\epsilon} x^{\frac{1}{2}} \mathrm{N}(D)^{\vartheta+\epsilon} \left(\frac{x}{h} \right)^{\vartheta^{\prime}+\epsilon}.
\end{equation}
Since
\begin{equation}
\sum_{\mathrm{N}(n) \leq x} \chi_{D}(n) = S(\chi_{D}, w)-\sum_{x < \mathrm{N}(n) \leq x+h} \chi_{D}(n) w(n),
\end{equation}
we estimate the second term on the right-hand side trivially with absolute values, deducing
\begin{equation}
\sum_{\mathrm{N}(n) \leq x} \chi_{D}(n) \ll_{\epsilon} x^{\frac{1}{2}} \mathrm{N}(D)^{\vartheta+\epsilon} \left(\frac{x}{h} \right)^{\vartheta^{\prime}+\epsilon}+h.
\end{equation}
Optimising $h = x^{1-\frac{1}{2(1+\vartheta^{\prime})}} \mathrm{N}(D)^{\frac{\vartheta}{1+\vartheta^{\prime}}}$ concludes the proof of \cref{prop:exponent-pair} for $x \geq \mathrm{N}(D)^{2\vartheta+\epsilon}$. In the complementary regime $x \leq \mathrm{N}(D)^{2\vartheta+\epsilon}$, the desired claim is a consequence of the trivial exponent pair $(1, 0)$ with adjustments to the respective exponents of $x$ and $\mathrm{N}(D)$, namely
\begin{equation}
\sum_{\mathrm{N}(n) \leq x} \chi_{D}(n) \ll x \ll_{\epsilon} x^{\alpha} \mathrm{N}(D)^{2\vartheta(1-\alpha)+\epsilon}.
\end{equation}
Inserting the values given by~\eqref{eq:exponent-pair} now equates $\beta = 2\vartheta(1-\alpha)$. The proof of \cref{prop:exponent-pair} is thus~complete.
\end{proof}

The following result is of independent interest.
\begin{corollary}\label{cor:duality}
Keep the notation as above. Then the quadratic exponent pair
\begin{equation}
(\alpha, \beta) = \left(\frac{1}{2(1+\vartheta^{\prime})}, \frac{2\vartheta+\vartheta^{\prime}}{2(1+\vartheta^{\prime})} \right)
\end{equation}
is admissible.
\end{corollary}

\begin{remark}
\cref{cor:duality} becomes nontrivial when $x \geq \mathrm{N}(D)^{\frac{2\vartheta+\vartheta^{\prime}}{1+2\vartheta^{\prime}}+\epsilon}$, yet fails to improve upon the threshold where \cref{prop:exponent-pair} becomes nontrivial. Under the generalised Lindel\"{o}f hypothesis $\vartheta = 0$, the quadratic exponent pairs as well as the associated thresholds coincide, respectively.
\end{remark}

\begin{proof}
The claim follows from the combination of \cref{lem:dual} and \cref{prop:exponent-pair}.
\end{proof}

\section{Symplectic zero density estimates}\label{sect:zero}
Expanding upon prior endeavours, Onodera~\cites[Theorem~1]{Onodera2009} extended Heath-Brown's quadratic large sieve inequality~\cites[Theorem~1]{HeathBrown1995} to the setting of $\mathbb{Q}(i)$, which asserts that if $M, N \geq 1$ and $(a_{n})$ is an arbitrary sequence of complex numbers, then
\begin{equation}\label{eq:Onodera}
\ \sideset{}{^{\sharp}} \sum_{\mathrm{N}(m) \leq M} \left|\ \sideset{}{^{\sharp}} \sum_{\mathrm{N}(n) \leq N} a_{n} \left(\frac{n}{m} \right) \right|^{2} \ll_{\epsilon} (MN)^{\epsilon}(M+N) \ \sideset{}{^{\sharp}} \sum_{\mathrm{N}(n) \leq N} |a_{n}|^{2},
\end{equation}
where $\sharp$ on the summation indicates restriction to odd squarefree integers. Shortly thereafter, Goldmakher and Louvel~\cites[Theorem~1.1]{GoldmakherLouvel2013} extended~\eqref{eq:Onodera} to primitive Gr\"{o}{\ss}encharaktern of trivial infinite type lying in the quadratic Hecke family over a number field~\cites{FisherFriedberg2004}{FriedbergHoffsteinLieman2003}.

In conjunction with Heath-Brown's deduction of a zero density theorem for the symplectic family of quadratic Dirichlet $L$-functions~\cites[Theorem~3]{HeathBrown1995} from the quadratic large sieve inequality over $\mathbb{Q}$, it is indispensable to prove the corresponding result in the setting of $\mathbb{Q}(i)$ due to the inferior quality of the best known unconditional result in the existing literature. For $\sigma \in [0, 1]$ and $T \geq 1$, let
\begin{equation}
N(\sigma, T, \chi) \coloneqq \#\{\rho = \beta+i\gamma \in \mathbb{C}: L(\rho, \chi) = 0, \, \sigma \leq \beta \leq 1, \, |\gamma| \leq T \},
\end{equation}
where the zeros are counted with multiplicity. Balog et al.~\cites[Lemma~2.5]{BalogBiroCherubiniLaaksonen2022} (cf.~\cites[Theorem~2]{Huxley1971}) established that for any fixed $\sigma \in [\frac{1}{2}, 1]$ and any $\epsilon > 0$,
\begin{equation}\label{eq:Balog-et-al}
\ \sideset{}{^{\flat}} \sum_{\mathrm{N}(D) \leq Q} N(\sigma, T, \chi_{D}) \ll_{\epsilon} Q^{\frac{10(1-\sigma)}{3-\sigma}+\epsilon} T^{\frac{7-5\sigma}{3-\sigma}+\epsilon},
\end{equation}
where $\flat$ indicates that the summation is taken over fundamental discriminants $D$ of norm up to $Q$. For our purposes, the exponent of $T$ proves inconsequential in practice, as the proof of \cref{thm:conditional} requires a zero density theorem only for $T \ll Q^{\epsilon}$. The following result provides an unconditional improvement over~\eqref{eq:Balog-et-al} and serves as the crux in establishing \cref{thm:conditional}.
\begin{proposition}\label{prop:density}
Keep the notation as above. Then we have for any fixed $\sigma \in [\frac{1}{2}, 1]$ and~any $\epsilon > 0$ that
\begin{equation}\label{eq:zero-density-estimate}
\ \sideset{}{^{\flat}} \sum_{\mathrm{N}(D) \leq Q} N(\sigma, T, \chi_{D}) \ll_{\epsilon} Q^{\frac{3(1-\sigma)}{2-\sigma}+\epsilon} T^{\frac{4-3\sigma}{2-\sigma}+\epsilon}.
\end{equation}
\end{proposition}

\begin{figure}[ht]
	\centering
	\begin{subfigure}{0.49\textwidth}
		\centering
		\begin{AnnotatedImage}[width=0.9]{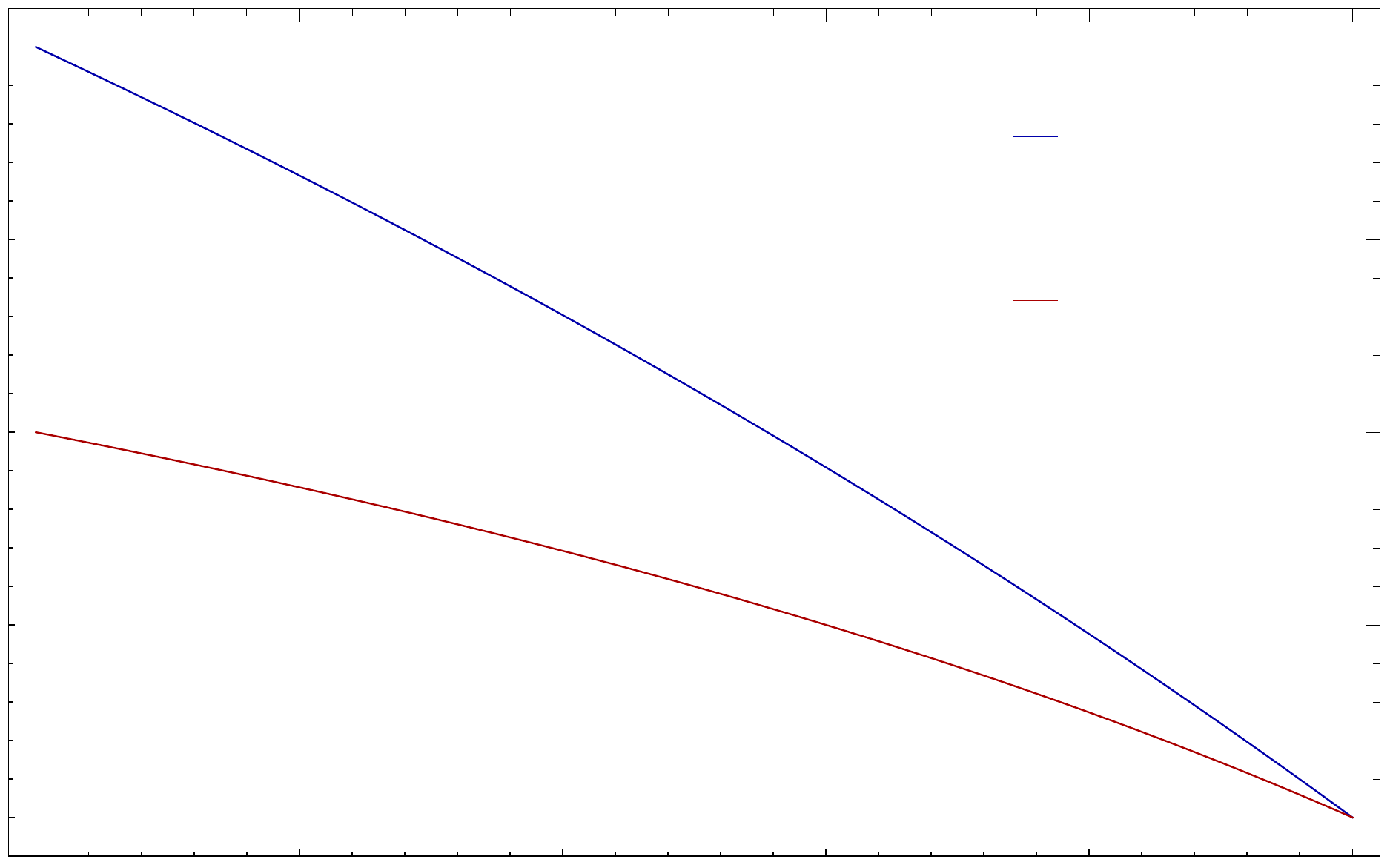}
		\annotate [draw=none,font=\tiny] at (0.85887,0.8479){\color{black}{$\frac{10(1-\sigma)}{3-\sigma}$}};
		\annotate [draw=none,font=\tiny] at (0.85,0.6567){\color{black}{$\frac{3(1-\sigma)}{2-\sigma}$}};
		\annotate [draw=none,font=\tiny] at (-0.015,0.9433){\color{black}{$2$}};
		\annotate [draw=none,font=\tiny] at (-0.015,0.5027){\color{black}{$1$}};
		\annotate [draw=none,font=\tiny] at (-0.015,0.0621){\color{black}{$0$}};
		\annotate [draw=none,font=\tiny] at (0.0266,-0.04){\color{black}{$\frac{1}{2}$}};
		\annotate [draw=none,font=\tiny] at (0.9745,-0.04){\color{black}{$1$}};
		\end{AnnotatedImage}
	\caption{Conductor aspect}
	\label{fig:conductor}
	\end{subfigure}
	\begin{subfigure}{0.49\textwidth}
		\centering
		\begin{AnnotatedImage}[width=0.9]{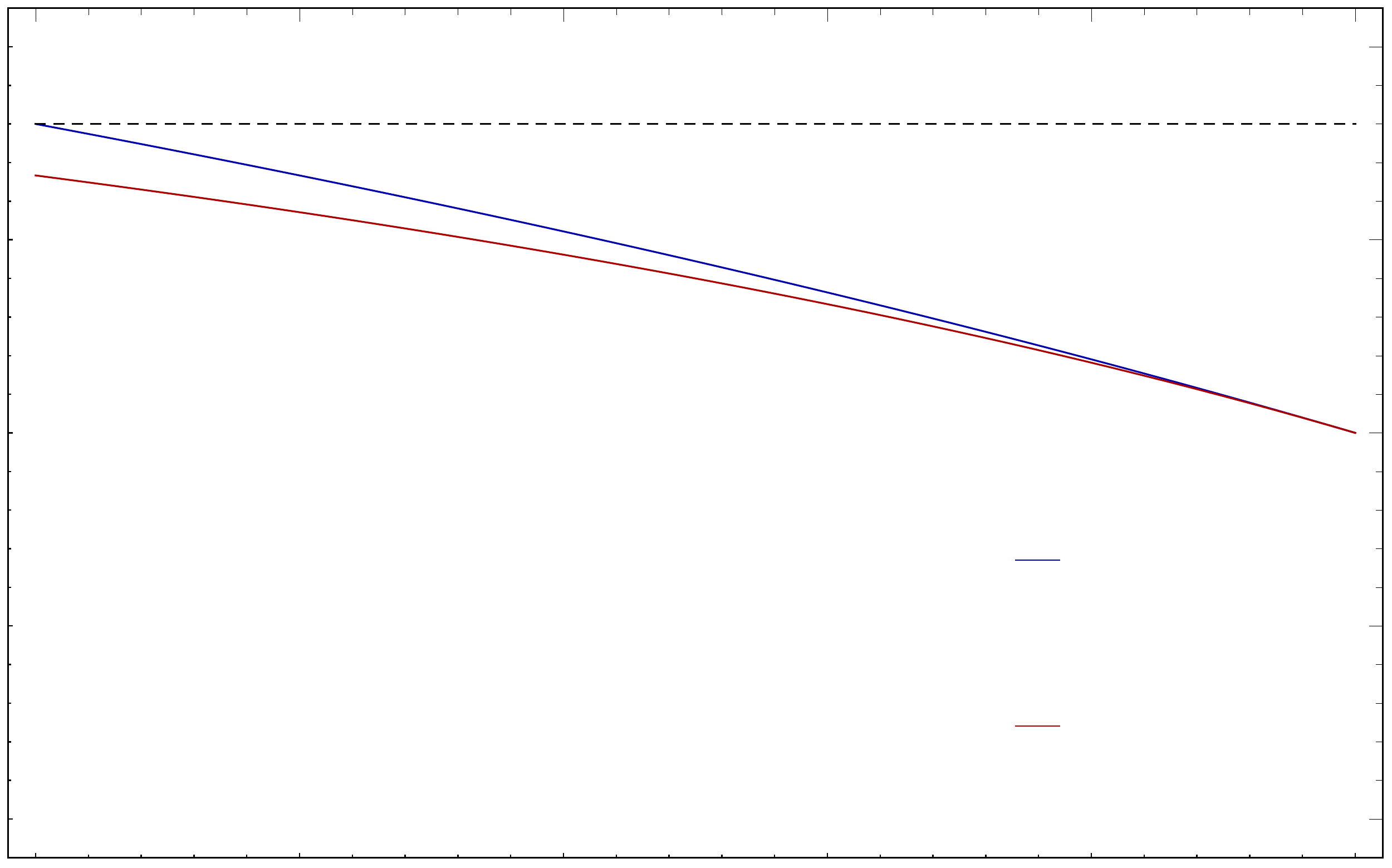}
		\annotate [draw=none,font=\tiny] at (0.836,0.3534){\color{black}{$\frac{7-5\sigma}{3-\sigma}$}};
		\annotate [draw=none,font=\tiny] at (0.836,0.1622){\color{black}{$\frac{4-3\sigma}{2-\sigma}$}};
		\annotate [draw=none,font=\tiny] at (1.015,0.8589){\color{black}{$\frac{9}{5}$}};
		\annotate [draw=none,font=\tiny] at (-0.015,0.7994){\color{black}{$\frac{5}{3}$}};
		\annotate [draw=none,font=\tiny] at (-0.015,0.0621){\color{black}{$0$}};
		\annotate [draw=none,font=\tiny] at (1.015,0.5027){\color{black}{$1$}};
		\annotate [draw=none,font=\tiny] at (0.0266,-0.04){\color{black}{$\frac{1}{2}$}};
		\annotate [draw=none,font=\tiny] at (0.9745,-0.04){\color{black}{$1$}};
		\end{AnnotatedImage}
	\caption{Archimedean aspect}
	\label{fig:archimedean}
	\end{subfigure}
\caption{A comparison between~\eqref{eq:Balog-et-al} and~\eqref{eq:zero-density-estimate} as $\sigma \in [\frac{1}{2}, 1]$ varies}
\label{fig:density}
\end{figure}

For reference, \cref{fig:density} juxtaposes the quality of~\eqref{eq:Balog-et-al} and~\eqref{eq:zero-density-estimate} in each aspect as $\sigma \in [\frac{1}{2}, 1]$ varies, visualising on the same scale the extent to which \cref{prop:density} attains a polynomial power-saving improvement. The inferior quality of the former in both aspects arises from the direct application of~\cites[Theorem~2]{Huxley1971}, where the summation runs through all conductors not exceeding $Q$ rather than fundamental discriminants. However, the ensuing restriction to fundamental discriminants by positivity as in~\cites[Equation~(32)]{BalogBiroCherubiniLaaksonen2022} sacrifices the intrinsic nature of quadratic characters, thereby motivating our enhancement. Notably, the transition from the unitary family to the symplectic family results in a significant deterioration of the quality of the exponent especially in the conductor aspect, as illustrated in \cref{fig:density}.

We conclude this section by noting that the yet unproven density hypothesis asserts
\begin{equation}
\ \sideset{}{^{\flat}} \sum_{\mathrm{N}(D) \leq Q} N(\sigma, T, \chi_{D}) \ll_{\epsilon} (QT^{2})^{2(1-\sigma)+\epsilon}.
\end{equation}

\begin{proof}[Proof of \cref{prop:density}]
The argument aligns with that of Heath-Brown~\cites[Section~11]{HeathBrown1995}. In particular, key components including~\cites[Theorem~2 and Corollary~3]{HeathBrown1995} can be derived in a sufficiently straightforward manner from~\cites[Theorem~1]{HeathBrown1995}. Consequently, the~Gaussian quadratic large sieve inequality~\eqref{eq:Onodera}, \cref{lem:heath-brown-1,lem:heath-brown-2}, as well as the standard procedure of Montgomery~\cites[Chapter~12]{Montgomery1971} collectively play an analogous role in our reasoning.

As an initial reduction, it suffices to demonstrate that the number $N$ of primitive quadratic characters $\chi_{D}$ with $\mathrm{N}(D) \leq Q$ for which $L(s, \chi_{D})$ has a zero within the rectangle
\begin{equation}\label{eq:rectangle}
[\sigma, \sigma+(\log QT)^{-1}) \times [\tau, \tau+(\log QT)^{-1})
\end{equation}
satisfies
\begin{equation}
N \ll (Q^{3} T)^{\frac{1-\sigma}{2-\sigma}}(QT)^{\epsilon}
\end{equation}
for $|\tau| \leq T$. Without loss of generality, one may assume $\sigma > \frac{1}{2}+(\log QT)^{-1}$, as otherwise the required bound becomes trivial. Following~\cites[Page~272]{HeathBrown1995} \textit{mutatis mutandis}, we introduce two auxiliary parameters $Y \gg X \gg 1$ and define the mollifier
\begin{equation}
M_{X}(s, \chi_{D}) \coloneqq \sum_{\mathrm{N}(n) \leq X} \mu(n) \chi_{D}(n) \mathrm{N}(n)^{-s}.
\end{equation}
In anticipation of the ensuing simplifications, it is convenient to handle the following two~cases separately.
\begin{enumerate}
\item\label{item:first} There exist $\gg N$ characters $\chi_{D}$, with corresponding zeros $\rho = \beta+i\gamma$, for which
\begin{equation}
\left|\int_{|u| \ll \log QT} L \left(\frac{1}{2}+i\gamma+iu, \chi_{D} \right) M_{X} \left(\frac{1}{2}+i\gamma+iu, \chi_{D} \right) Y^{\frac{1}{2}-\beta+iu} \Gamma \left(\frac{1}{2}-\beta+iu \right) \, du \right| \gg 1.
\end{equation}
\item\label{item:second} There exist an integer $U$ in the range $X \leq U \leq Y^{2}$ and $\gg N(\log Y)^{-1}$ characters~$\chi_{D}$, with corresponding zeros $\rho = \beta+i\gamma$, for which
\begin{equation}
\left|\sum_{U < \mathrm{N}(n) \leq 2U} a_{n} \chi_{D}(n) \mathrm{N}(n)^{-\rho} e^{-\frac{n}{Y}} \right| \gg (\log Y)^{-1},
\end{equation}
where the coefficients $a_{n} = a_{n}(X)$ are given by the expansion
\begin{equation}
L(s, \chi_{D}) M_{X}(s, \chi_{D}) = \sum_{0 \ne n \in \mathbb{Z}[i]} a_{n} \chi_{D}(n) \mathrm{N}(n)^{-s}.
\end{equation}
\end{enumerate}

In the first case~\eqref{item:first}, given that $\rho$ is required to lie in the rectangle~\eqref{eq:rectangle}, we observe
\begin{equation}
\int_{|u-\tau| \ll \log QT} \left|L \left(\frac{1}{2}+iu, \chi_{D} \right) M_{X} \left(\frac{1}{2}+iu, \chi_{D} \right) \right| \, du \gg Y^{\sigma-\frac{1}{2}}(\log QT)^{-1}.
\end{equation}
Applying H\"{o}lder's inequality and summing over fundamental discriminants with $\mathrm{N}(D) \leq Q$ yields
\begin{align}
&N(Y^{\sigma-\frac{1}{2}}(\log QT)^{-1})^{\frac{4}{3}}(\log QT)^{-\frac{1}{3}}\\
& \ll \int_{|u-\tau| \ll \log QT} \ \sideset{}{^{\flat}} \sum_{\mathrm{N}(D) \leq Q} \left|L \left(\frac{1}{2}+iu, \chi_{D} \right) M_{X} \left(\frac{1}{2}+iu, \chi_{D} \right) \right|^{\frac{4}{3}} \, du\\
& \ll \left(\int_{|u-\tau| \ll \log QT} \ \sideset{}{^{\flat}} \sum_{\mathrm{N}(D) \leq Q} \left|L \left(\frac{1}{2}+iu, \chi_{D} \right) \right|^{4} \, du \right)^{\frac{1}{3}}\\
& \hspace{.34em} \times \left(\int_{|u-\tau| \ll \log QT} \ \sideset{}{^{\flat}} \sum_{\mathrm{N}(D) \leq Q} \left|M_{X} \left(\frac{1}{2}+iu, \chi_{D} \right) \right|^{2} \, du \right)^{\frac{2}{3}}.
\end{align}
It follows from \cref{lem:heath-brown-1} below that the first integral on the right-hand side is bounded by $\ll_{\epsilon}(QT^{2})^{1+\epsilon}$. On the other hand, the second integral may be handled by performing dyadic subdivisions of $M_{X}(s, \chi_{D})$ into the ranges $V < \mathrm{N}(n) \leq 2V$ and applying the Cauchy--Schwarz inequality. Hence, \cref{lem:heath-brown-2} below implies
\begin{align}
&\int_{|u-\tau| \ll \log QT} \ \sideset{}{^{\flat}} \sum_{\mathrm{N}(D) \leq Q} \left|M_{X} \left(\frac{1}{2}+iu, \chi_{D} \right) \right|^{2} \, du\\
& \ll \log X \sum_{V} \int_{|u-\tau| \ll \log QT} \ \sideset{}{^{\flat}} \sum_{\mathrm{N}(D) \leq Q} \left|\sum_{V < \mathrm{N}(n) \leq 2V} \mu(n) \chi_{D}(n) \mathrm{N}(n)^{-\frac{1}{2}-iu} \right|^{2} \, du\\
& \ll_{\epsilon} \log X \sum_{V} (Q+V)(QV)^{\epsilon} \ll_{\epsilon} (Q+X)(QX)^{\epsilon}.
\end{align}
Altogether, the first case~\eqref{item:first} contributes
\begin{equation}
N \ll (QT^{2})^{\frac{1}{3}}(Q+X)^{\frac{2}{3}} Y^{\frac{2-4\sigma}{3}}(QTY)^{\epsilon}.
\end{equation}

For the second case~\eqref{item:second}, one may assume that $Y \leq (QT)^{A}$ for some constant $A > 0$. If $\rho$ lies within the rectangle~\eqref{eq:rectangle}, then by partial summation
\begin{multline}
\sum_{U < \mathrm{N}(n) \leq 2U} a_{n} \chi_{D}(n) \mathrm{N}(n)^{-\rho} e^{-\frac{n}{Y}} \ll \left|\sum_{U < \mathrm{N}(n) \leq 2U} a_{n} \chi_{D}(n) \mathrm{N}(n)^{-s} e^{-\frac{n}{Y}} \right|\\
 + \int_{U}^{2U} \left|\sum_{U < \mathrm{N}(n) \leq V} a_{n} \chi_{D}(n) \mathrm{N}(n)^{-s} e^{-\frac{n}{Y}} \right| \frac{dV}{V},
\end{multline}
where $s = \sigma+i\tau$. It then follows from the Cauchy--Schwarz inequality that
\begin{equation}
\left|\sum_{U < \mathrm{N}(n) \leq 2U} a_{n} \chi_{D}(n) \mathrm{N}(n)^{-s} e^{-\frac{n}{Y}} \right|^{2}+\int_{U}^{2U} \left|\sum_{U < \mathrm{N}(n) \leq V} a_{n} \chi_{D}(n) \mathrm{N}(n)^{-s} e^{-\frac{n}{Y}} \right|^{2} \frac{dV}{V} \gg (\log QT)^{-2}
\end{equation}
for $\gg N(\log QT)^{-1}$ characters $\chi_{D}$. Summing over fundamental discriminants with $\mathrm{N}(D) \leq Q$ yields
\begin{multline}
\ \sideset{}{^{\flat}} \sum_{\mathrm{N}(D) \leq Q}\left|\sum_{U < \mathrm{N}(n) \leq 2U} a_{n} \chi_{D}(n) \mathrm{N}(n)^{-s} e^{-\frac{n}{Y}} \right|^{2}+\int_{U}^{2U} \ \sideset{}{^{\flat}} \sum_{\mathrm{N}(D) \leq Q} \left|\sum_{U < \mathrm{N}(n) \leq V} a_{n} \chi_{D}(n) \mathrm{N}(n)^{-s} e^{-\frac{n}{Y}} \right|^{2} \frac{dV}{V}\\
\gg N(\log QT)^{-3}.
\end{multline}
Nonetheless, \cref{lem:heath-brown-2} together with the inequality $|a_{n}| \leq \tau(n) \ll |n|^{\epsilon}$ for any $\epsilon > 0$ yields
\begin{equation}
\ \sideset{}{^{\flat}} \sum_{\mathrm{N}(D) \leq Q} \left|\sum_{U < \mathrm{N}(n) \leq V} a_{n} \chi_{D}(n) \mathrm{N}(n)^{-s} e^{-\frac{n}{Y}} \right|^{2} \ll_{\epsilon} (Q+U) U^{1-2\sigma}(QT)^{\epsilon} e^{-\frac{U}{Y}}
\end{equation}
for some $U$ in the range $X \leq U \leq Y^{2}$. Consequently, we deduce
\begin{equation}
N \ll_{\epsilon} (QX^{1-2\sigma}+Y^{2-2\sigma})(QT)^{\epsilon}.
\end{equation}

Gathering everything together from the cases~\eqref{item:first} and~\eqref{item:second} concludes for any $\epsilon > 0$ that
\begin{equation}
N \ll_{\epsilon} ((QT^{2})^{\frac{1}{3}}(Q+X)^{\frac{2}{3}} Y^{\frac{2(1-2\sigma)}{3}}+QX^{1-2\sigma}+Y^{2-2\sigma})(QT)^{\epsilon}.
\end{equation}
To balance the above three terms, we choose the parameters $X$ and $Y$ optimally by taking
\begin{equation}
X = Q, \qquad Y = (Q^{3} T^{2})^{\frac{1}{4-2\sigma}},
\end{equation}
from which the desired assertion follows.
\end{proof}

The proof of \cref{prop:density} is thereby reduced to the task of establishing the following~two lemmata.
\begin{lemma}\label{lem:heath-brown-2}
Let $N, Q \geq 1$, and let $(a_{n})$ be an arbitrary sequence of complex numbers. Then we have for any $\epsilon > 0$ that
\begin{equation}
\ \sideset{}{^{\flat}} \sum_{\mathrm{N}(D) \leq Q} \left|\sum_{\mathrm{N}(n) \leq N} a_{n} \chi_{D}(n) \right|^{2} \ll_{\epsilon} (N+Q) N^{1+\epsilon} Q^{\epsilon} \max_{\mathrm{N}(n) \leq N} |a_{n}|^{2}.
\end{equation}
\end{lemma}

\begin{proof}
Following the approach outlined in~\cites[Pages~265--267]{HeathBrown1995}, the claim follows from~\eqref{eq:Onodera} without any substantive modifications.
\end{proof}

\begin{lemma}\label{lem:heath-brown-1}
Let $Q \geq 1$ and $\tau \in \mathbb{R}$. Then we have for any fixed $\sigma \in [\frac{1}{2}, 1]$ and any $\epsilon > 0$~that
\begin{equation}
\ \sideset{}{^{\flat}} \sum_{\mathrm{N}(D) \leq Q} |L(\sigma+i\tau, \chi_{D})|^{4} \ll_{\epsilon} (Q+(Q(1+|\tau|)^{2})^{2-2\sigma})(Q(1+|\tau|))^{\epsilon}.
\end{equation}
\end{lemma}

\begin{proof}
By dyadic subdivisions, we write
\begin{equation}
\mathcal{F}(Q, s) \coloneqq \ \sideset{}{^{\flat}} \sum_{Q < \mathrm{N}(D) \leq 2Q} |L(s, \chi_{D})|^{4},
\end{equation}
and we denote by $\nu(\sigma)$ the infimum of those exponents $\nu$ for which
\begin{equation}
\mathcal{F}(Q, \sigma+i\tau) \ll (Q+(QT^{2})^{2-2\sigma})(QT^{2})^{\nu}
\end{equation}
holds uniformly in $Q$ and $\tau$. Here and henceforth, we abbreviate $T \coloneqq 1+|\tau|$ for the sake of notational simplicity. It follows from~\cites[Equation~(26)]{HeathBrown1995}, the Cauchy--Schwarz inequality, and the crude upper bound $\Gamma(x+iy) \ll_{x} e^{-|y|}$ that
\begin{equation}
\mathcal{F}(Q, s) \ll \ \sideset{}{^{\flat}} \sum_{Q < \mathrm{N}(D) \leq 2Q} \left|\sum_{0 \ne n \in \mathbb{Z}[i]} \tau(n) \chi_{D}(n) \mathrm{N}(n)^{-s} e^{-\frac{n}{U}} \right|^{2}+U^{2(\alpha-\sigma)} \int_{-\infty}^{\infty} \mathcal{F}(Q, \alpha+iu) e^{-|u-\tau|} \, du
\end{equation}
for fixed $\alpha$ and $\sigma$, where $s = \sigma+i\tau$ and $0 \leq \alpha < \sigma \leq 1$. We invoke the functional equation for Hecke $L$-functions~\cites[Equation~(5.14)]{Mordell1931} and Stirling's formula~\eqref{eq:Stirling}, obtaining
\begin{equation}
|L(\alpha+iu, \chi_{D})|^{4} \ll (Q(1+|u|)^{2})^{2-4\alpha} |L(1-\alpha+iu, \chi_{D})|^{4},
\end{equation}
where the exponent of $1+|u|$ differs from that over $\mathbb{Q}$, as can be inferred from the size of~the associated analytic conductor. It therefore follows that
\begin{equation}
\mathcal{F}(Q, \alpha+iu) \ll (Q(1+|u|)^{2})^{2-4\alpha}(Q+(Q(1+|u|)^{2})^{2\alpha})(Q(1+|u|))^{\nu(1-\alpha)+\epsilon},
\end{equation}
thereby yielding
\begin{multline}\label{eq:synthesise-1}
\mathcal{F}(Q, s) \ll_{\epsilon} \ \sideset{}{^{\flat}} \sum_{Q < \mathrm{N}(D) \leq 2Q} \left|\sum_{0 \ne n \in \mathbb{Z}[i]} \tau(n) \chi_{D}(n) \mathrm{N}(n)^{-s} e^{-\frac{n}{U}} \right|^{2}\\
 + U^{2(\alpha-\sigma)} (QT^{2})^{2-4\alpha}(Q+(QT^{2})^{2\alpha})(QT)^{\nu(1-\alpha)+\epsilon}.
\end{multline}
The inner sum on the right-hand side may be truncated at $\mathrm{N}(n) \leq N_{0} \coloneqq U(\log QT)^{2}$, since otherwise the contribution is negligible. We break the remaining terms into $O(\log N_{0})$ dyadic ranges $N < \mathrm{N}(n) \leq 2N$ with $N \ll N_{0}$ and apply the Cauchy--Schwarz inequality once again. Since \cref{lem:heath-brown-2} demonstrates
\begin{equation}
\ \sideset{}{^{\flat}} \sum_{Q < \mathrm{N}(D) \leq 2Q} \left|\sum_{N < \mathrm{N}(n) \leq 2N} \tau(n) \chi_{D}(n) \mathrm{N}(n)^{-s} e^{-\frac{n}{U}} \right|^{2} \ll_{\epsilon} (N+Q) N^{1-2\sigma+\epsilon} Q^{\epsilon},
\end{equation}
we obtain
\begin{equation}\label{eq:synthesise-2}
\ \sideset{}{^{\flat}} \sum_{Q < \mathrm{N}(D) \leq 2Q} \left|\sum_{0 \ne n \in \mathbb{Z}[i]} \tau(n) \chi_{D}(n) \mathrm{N}(n)^{-s} e^{-\frac{n}{U}} \right|^{2} \ll_{\epsilon} (Q+U^{2-2\sigma})(QTU)^{\epsilon},
\end{equation}
where $\sigma \in [\frac{1}{2}, 1]$. Substituting~\eqref{eq:synthesise-2} into~\eqref{eq:synthesise-1} implies
\begin{equation}\label{eq:synthesise-3}
\mathcal{F}(Q, s) \ll_{\epsilon} (Q+U^{2-2\sigma})(QTU)^{\epsilon}+U^{2(\alpha-\sigma)} (QT^{2})^{2-4\alpha}(Q+(QT^{2})^{2\alpha})(QT)^{\nu(1-\alpha)+\epsilon},
\end{equation}
where $\sigma \in [\frac{1}{2}, 1]$ and $\alpha \in [0, \sigma)$.

To verify the required claim, we first consider the case where $\sigma \in (\frac{1}{2}, 1]$. Setting $\alpha = 1-\sigma$ and $U = (QT^{2})^{1+\delta}$ for some $\delta \in (0, 1)$ in~\eqref{eq:synthesise-3} yields
\begin{align}
\mathcal{F}(Q, s) &\ll_{\epsilon} (Q+U^{2-2\sigma}+(QT^{2} U^{-1})^{4\sigma-2}(Q+(QT^{2})^{2-2\sigma}) (QT)^{\nu(\sigma)})(QTU)^{\epsilon}\\
&\ll_{\epsilon} (Q+(QT^{2})^{2-2\sigma})((QT^{2})^{\delta}+(QT^{2})^{\nu(\sigma)-(4\sigma-2) \delta})(QTU)^{\epsilon}
\end{align}
uniformly in $Q$ and $\tau$. It then follows from the definition of $\nu(\sigma)$ that
\begin{equation}
\nu(\sigma) \leq \max(\delta, \nu(\sigma)-(4\sigma-2) \delta) = \delta.
\end{equation}
Taking $\delta$ arbitrarily small confines $\nu(\sigma) \leq 0$, thereby establishing \cref{lem:heath-brown-1} for $\sigma \in (\frac{1}{2}, 1]$.

On the other hand, the case of $\sigma = \frac{1}{2}$ can be handled by choosing $\alpha = \frac{1}{2}-\epsilon$ in~\eqref{eq:synthesise-3}. Since the preceding argument shows that $\nu(\beta) \leq 0$ whenever $\beta \in (\frac{1}{2}, 1]$, it follows that $\nu(1-\alpha) \leq 0$. This leads to the desideratum upon setting $U = QT^{2}$ and taking $\epsilon$ arbitrarily small, ensuring that $\nu(\frac{1}{2}) \leq 0$, as required. The proof of \cref{lem:heath-brown-1} is thus complete for~$\sigma = \frac{1}{2}$.
\end{proof}

\section{Sparse averages of Zagier \texorpdfstring{$L$}{}-series}\label{sect:Zagier}
This section aims to address a sparse first moment of the Zagier $L$-series on the critical line, drawing inspiration from Balkanova, Frolenkov, and Risager~\cites[Theorem~1.2]{BalkanovaFrolenkovRisager2022},~which in turn pertains to double character sums analogous to those in the prequel~\cites[Lemma~3.2]{Kaneko2024}.

Prior to undertaking such pursuits, we shall provide remarks on some phenomena identified in the $2$-dimensional framework. Throughout this paragraph, let $\Gamma = \mathrm{PSL}_{2}(\mathbb{Z})$. Balkanova, Frolenkov, and Risager~\cites[Theorem~1.1]{BalkanovaFrolenkovRisager2022} proved uniformly for $|t| \ll x^{\epsilon}$ that
\begin{equation}\label{eq:BFR-1.1}
\sum_{3 \leq n \leq x} L \left(\frac{1}{2}+it, n^{2}-4 \right) = \int_{3}^{x} m_{t}(u) \, du+O_{\epsilon}(x^{\frac{2(1+\vartheta)}{3}+\epsilon}),
\end{equation}
where $m_{t}(u)$ denotes a density function given by
\begin{equation}
m_{t}(u) \coloneqq 
	\begin{dcases}
	\frac{1}{2\zeta(\frac{3}{2})} \left(\log(u^{2}-4)-\frac{\pi}{2}+3\gamma-2\frac{\zeta^{\prime}(\frac{3}{2})}{\zeta(\frac{3}{2})}-\log 8\pi \right) & \text{if $t = 0$},\\
	\frac{\zeta(1+2it)}{\zeta(\frac{3}{2}+it)}+\frac{2^{\frac{1}{2}+it} \sin(\frac{\pi}{4}+\frac{i\pi t}{2})}{\pi^{it}} \frac{\zeta(it)}{\zeta(\frac{3}{2}-it)} \Gamma(it)(u^{2}-4)^{-it} & \text{if $t \ne 0$}.
	\end{dcases}
\end{equation}
Proceeding further, another result of theirs~\cites[Theorem~1.2]{BalkanovaFrolenkovRisager2022} asserts that the exponent $2(2\delta-1)+\epsilon$ in~\eqref{eq:BFR-1.1} corresponds to $\mathcal{E}_{\Gamma}(x) \ll_{\epsilon} x^{\delta+\epsilon}$, provided $\delta \in [\frac{5}{8}, \frac{3}{4}]$. This statement may in principle be seen from the first display on~\cites[Page~116]{SoundararajanYoung2013} by partial summation,~truncating the integral up to $|\Im(s)| \ll x^{\epsilon}$, and invoking $X \coloneqq x^{\frac{1}{2}}+x^{-\frac{1}{2}}$ in the notation therein. Hence, we obtain heuristically
\begin{equation}
\mathcal{E}_{\Gamma}(x) \ll_{\epsilon} x^{\frac{1}{2}} V^{\frac{1}{2}+\epsilon}+x^{\frac{1}{2}+\frac{2(2\delta-1)}{2}+\epsilon} V^{-\frac{1}{2}} \ll_{\epsilon} x^{\delta+\epsilon},
\end{equation}
where we assumed that the contribution of the main term in~\eqref{eq:BFR-1.1} becomes negligibly small. Nonetheless, the converse assertion follows in a rather straightforward manner, provided the density function is allowed to remain inexplicit; see~\cites[Lemma~3.1]{Kaneko2024} for further details.

We are prepared to establish the following reduction principle in the $3$-dimensional setting.
\begin{lemma}\label{lem:BFR}
Let $\Gamma = \mathrm{PSL}_{2}(\mathbb{Z}[i])$. Suppose that for some fixed $\delta \in [\frac{5}{4}, \frac{3}{2}]$ and for any $\epsilon > 0$,
\begin{equation}\label{eq:delta}
\mathcal{E}_{\Gamma}(x) \ll_{\epsilon} x^{\delta+\epsilon}.
\end{equation}
Then there exists an effectively computable polynomial $P_{t}$ of degree $1$ such that for any $\epsilon > 0$,
\begin{equation}\label{eq:asymptotic}
\sum_{2 \ne \mathrm{N}(n) \leq x} L \left(\frac{1}{2}+it, n^{2}-4 \right) = xP_{t}(\log x)+O_{\epsilon}(x^{2(\delta-1)+\epsilon})
\end{equation}
uniformly for $|t| \ll x^{\epsilon}$.
\end{lemma}

\begin{remark}
Sarnak's barrier $\delta = \frac{5}{3}$ as in~\eqref{eq:Sarnak} corresponds to the Weyl-strength subconvex exponent $\vartheta = \frac{1}{6}$ for the summand on the left-hand side of~\eqref{eq:asymptotic} by \cref{lem:Szmidt-2}. In conjunction with the rational case~\cites[Theorem~1.3]{BalkanovaFrolenkovRisager2022}, it would be reasonable to predict that the~error term is $\Omega_{\pm}(x^{\frac{1}{2}})$ at least when $t = 0$. We refrain from undertaking such supplementary efforts.
\end{remark}

\begin{proof}
It suffices to focus on the error term in~\eqref{eq:asymptotic}, since the size of the error term is ensured by random matrix theory. Our machinery is fundamentally based on reductio ad absurdum. Let $V > 0$ be a parameter to be determined later. From~\cites[Equations~(39) and~(40)]{BalogBiroCherubiniLaaksonen2022}, \cref{prop:Bessel}, and partial summation, we deduce
\begin{align}
\Psi_{\Gamma}(x) &= \frac{1}{\pi} \sum_{2 \ne \mathrm{N}(n) \leq x} \mathrm{N}(n) L(1, n^{2}-4)+O(x^{1+\epsilon})\\
& = \frac{x^{2}}{2}+O_{\epsilon}(xV^{\frac{1}{2}+\epsilon}+x^{\frac{5}{4}} V^{\frac{1}{4}+\epsilon})-\frac{1}{\pi} \int_{(\frac{1}{2})} \sum_{2 \ne \mathrm{N}(n) \leq x} \mathrm{N}(n) L(s, n^{2}-4) \Gamma(s-1) V^{s-1} \, \frac{ds}{2\pi i}.
\end{align}
After making a change of variables $s \mapsto \frac{1}{2}+it$, the rapid decay of the gamma factor $\Gamma(-\frac{1}{2}+it)$ enables the truncation of the integral up to the range $|t| \ll x^{\epsilon}$ at the expense of introducing a negligibly small error term, thereby yielding
\begin{multline}\label{eq:expense}
\mathcal{E}_{\Gamma}(x) = -\frac{1}{2\pi^{2}} \int_{|t| \ll x^{\epsilon}} \Gamma \left(-\frac{1}{2}+it \right) V^{-\frac{1}{2}+it} \sum_{2 \ne \mathrm{N}(n) \leq x} \mathrm{N}(n) L \left(\frac{1}{2}+it, n^{2}-4 \right) \, dt\\
 + O_{\epsilon}(xV^{\frac{1}{2}+\epsilon}+x^{\frac{5}{4}} V^{\frac{1}{4}+\epsilon}).
\end{multline}
For any $1 \leq y \leq x$, subtracting~\eqref{eq:expense} from the one with $x \mapsto x+y$ implies
\begin{multline}\label{eq:expense-2}
\mathcal{E}_{\Gamma}(x+y)-\mathcal{E}_{\Gamma}(x) = -\frac{1}{2\pi^{2}} \int_{|t| \ll x^{\epsilon}} \Gamma \left(-\frac{1}{2}+it \right) V^{-\frac{1}{2}+it} \sum_{x < \mathrm{N}(n) \leq x+y} \mathrm{N}(n) L \left(\frac{1}{2}+it, n^{2}-4 \right) \, dt\\
 + O_{\epsilon}(xV^{\frac{1}{2}+\epsilon}+x^{\frac{5}{4}} V^{\frac{1}{4}+\epsilon}).
\end{multline}
On the other hand, upon taking $1 \leq y \leq x^{2(\delta-1)}$ sufficiently small, the assumption~\eqref{eq:delta}~along with partial summation necessitates the existence of an effectively computable polynomial~$P_{t}^{\prime}$ of degree $1$, depending on $|t| \ll x^{\epsilon}$, such that for any $\epsilon > 0$,
\begin{equation}\label{eq:impose}
\sum_{x < \mathrm{N}(n) \leq x+y} L \left(\frac{1}{2}+it, n^{2}-4 \right) = yP_{t}^{\prime}(\log y)+O_{\epsilon}(x^{2(\delta-1)+\epsilon}),
\end{equation}
since otherwise~\eqref{eq:expense-2} contradicts the assumption~\eqref{eq:delta} for any value of $V > 0$. Choosing now $y = x$ in~\eqref{eq:impose} and performing dyadic subdivisions completes the proof of \cref{lem:BFR}.
\end{proof}

We proceed to estimate certain double character sums involving quadratic characters of a special type; cf.~\cites[Lemma~3.2]{Kaneko2024}.
\begin{proposition}\label{prop:substitute}
Let $1 \leq A \leq |B|$ and $R \geq 1$. Let $\delta \in [\frac{5}{4}, \frac{3}{2}]$ be such that~\eqref{eq:delta} holds. Then we have for any $\epsilon > 0$ that
\begin{equation}\label{eq:nontrivial}
\sum_{B < \mathrm{N}(a) \leq A+B} \sum_{R < \mathrm{N}(r) \leq 2R} \left(\frac{a^{2}-4}{r} \right) \ll_{\epsilon} (A^{1+\epsilon}+|B|^{2(\delta-1)+\epsilon}) R^{\frac{1}{2}}.
\end{equation}
\end{proposition}

\begin{proof}
It follows from~\eqref{eq:S-chi-w} with $\sigma = \frac{1}{2}$ and $h = x = R$ that
\begin{equation}
\sum_{0 \ne r \in \mathbb{Z}[i]} \left(\frac{a^{2}-4}{r} \right) w(r) = \int_{|t| \ll R^{\epsilon}} \tilde{w} \left(\frac{1}{2}+it \right) L \left(\frac{1}{2}+it, \chi_{a^{2}-4} \right) \frac{ds}{2\pi i}+O(1).
\end{equation}
On the other hand, the application of \cref{lem:Szmidt-1,lem:Szmidt-2} translates the Dirichlet $L$-function on the right-hand side to the Zagier $L$-series $L(\frac{1}{2}+it, a^{2}-4)$. Executing the summation~over $a$ then enables the utilisation of \cref{lem:BFR}, yielding the desired claim by bounding the main term trivially. The proof of \cref{prop:substitute} is thus complete.
\end{proof}

\section{Mean values of \texorpdfstring{$\rho(c, a)$}{} over lattices}\label{sect:bilinear}
A result in this section constitutes the crux of the present paper, in harmony with~\cites[Section~5]{Kaneko2024}. Let $1 \leq A \leq |B|$ and $C \geq 1$. We intend to evaluate the bilinear form
\begin{equation}
\mathcal{F}(A, B, C) \coloneqq \sum_{B < \mathrm{N}(a) \leq A+B} \sum_{\mathrm{N}(c) \leq C} \rho(c, a).
\end{equation}
It follows from~\cref{lem:rho-phi}, M\"{o}bius inversion, and a fundamental exponent towards the Gau{\ss} circle problem that the trivial bound for the error term is given by
\begin{equation}\label{eq:trivial}
\mathcal{F}(A, B, C) = \sum_{\mathrm{N}(c) \leq C} \left(\frac{A}{\mathrm{N}(c)}+O(1) \right) \varphi(c) = \frac{\pi}{\zeta_{\mathbb{Q}(i)}(2)} AC+O(AC^{\frac{1}{2}}+C^{2}),
\end{equation}
which emulates~\cites[Equation~(42)]{Iwaniec1984}, yet falls well shy of meeting our requirements. As~a preliminary step to improving~\eqref{eq:trivial}, the following asymptotic formula akin to \cref{prop:Bessel} plays a pivotal role; cf.~\cites[Lemma~2*]{Iwaniec1984} over the rationals $\mathbb{Q}$, which relies fundamentally on the Weil bound and standard Fourier analysis, with no further nontrivial manipulations required.
\begin{lemma}\label{lem:Iwaniec-lemma-2}
If $r$ is squarefree and $(\ell, r) = 1$, then we have for any $\epsilon > 0$ that
\begin{equation}
\sum_{\substack{B < \mathrm{N}(a) \leq A+B \\ a \equiv b \tpmod{\ell}}} \left(\frac{a^{2}-4}{r} \right) = \pi A \frac{\mu(r)}{\mathrm{N}(\ell r)}+O_{\epsilon}(\mathrm{N}(r)^{\frac{1}{2}+\epsilon}+|B|^{\frac{1}{4}+\epsilon} \mathrm{N}(r)^{\frac{1}{4}+\epsilon}),
\end{equation}
where
\begin{equation}
\mu(r) = \sum_{a \tpmod{r}} \left(\frac{a^{2}-4}{r} \right).
\end{equation}
\end{lemma}

\begin{proof}
Without loss of generality, we assume that $\ell \sim 1$ so that the congruence $a \equiv b \tpmod{\ell}$ may henceforth be neglected for technical brevity. If $0 \ne c \in \mathbb{Z}[i]$ is squarefree so that $c_{\star} \sim 1$, then it follows from~\eqref{eq:multiplicativity}, \eqref{eq:incongruent}, \cref{prop:Bessel}, and M\"{o}bius inversion that
\begin{equation}
\sum_{B < \mathrm{N}(a) \leq A+B} \rho(c, a) = \sum_{r \mid c} \sum_{B < \mathrm{N}(a) \leq A+B} \left(\frac{a^{2}-4}{r} \right) = \pi A \sum_{r \mid c} \frac{\mu(r)}{\mathrm{N}(r)}+O_{\epsilon}(\mathrm{N}(c)^{\frac{1}{2}+\epsilon}+|B|^{\frac{1}{4}+\epsilon} \mathrm{N}(c)^{\frac{1}{4}+\epsilon}).
\end{equation}
M\"{o}bius inversion ensures the isolation of the term corresponding to a fixed $r \mid c$ on both~sides, thereby leading to the desideratum.
\end{proof}

The evaluation of the number of squarefree Gaussian integers of norm not exceeding~a~given point likewise plays a crucial role in the subsequent discussion.
\begin{lemma}\label{lem:squarefree}
For $S \geq 1$, we have that
\begin{equation}
\sum_{\substack{\mathrm{N}(s) \leq S \\ (s, q) = 1}} \mu^{2}(s) = \frac{S}{\zeta_{\mathbb{Q}(i)}(2)} \prod_{p \mid q} \left(1+\frac{1}{\mathrm{N}(p)} \right)^{-1}+O(\tau(q) S^{\frac{1}{2}}).
\end{equation}
\end{lemma}

\begin{proof}
The argument aligns with that in the rational case; see the fourth display on~\cites[Page~155]{Iwaniec1984}.
\end{proof}

The following result serves as a $3$-dimensional counterpart of~\cites[Lemma~2.2]{Kaneko2024}.
\begin{proposition}\label{prop}
Let $1 \leq A \leq |B|$ and $C \geq 1$. Suppose that $(\alpha, \beta)$ is a quadratic exponent pair over $\mathbb{Q}(i)$. Then we have for some fixed $\delta \in [\frac{5}{4}, \frac{3}{2}]$ and for any $\sigma \in [\frac{1}{2}, 1]$ and $\epsilon > 0$ that
\begin{multline}\label{eq:delta-included}
\mathcal{F}(A, B, C) = \frac{\pi}{\zeta_{\mathbb{Q}(i)}(2)} AC+O_{\epsilon}((|B|^{\frac{2(3+2\beta-(3+\beta)\sigma)}{(3-2\alpha)(2-\sigma)}} C+|B|^{\frac{1}{4}+\frac{3+2\beta-(3+\beta)\sigma}{(3-2\alpha)(2-\sigma)}} C\\
 + |B|^{2(\delta-1)-\frac{2(3+2\beta-(3+\beta)\sigma)}{(3-2\alpha)(2-\sigma)}} C+A^{\frac{1}{3}} |B|^{\frac{1}{6}} C+AC^{\frac{1}{2}})(|B|C)^{\epsilon}).
\end{multline}
\end{proposition}

\begin{proof}
Let $\mathscr{L}$ stand for the set of $\ell \in \mathbb{Z}[i]$ such that $p \mid \ell$ implies $p^{2} \mid \ell$ for any prime element $p \in \mathbb{Z}[i]$. It is now convenient to decompose
\begin{equation}
\mathcal{F}(A, B, C) = \sum_{\substack{\mathrm{N}(\ell rs) \leq C \\ \ell \in \mathscr{L}, \, (rs, 4\ell) = 1}} \mu(rs)^{2} \sum_{B < \mathrm{N}(a) \leq A+B} \rho(\ell, a) \left(\frac{a^{2}-4}{r} \right) \eqqcolon \mathcal{F}_{0}(A, B, C)+\mathcal{F}_{\infty}(A, B, C),
\end{equation}
where, for an auxiliary parameter $R$ to be determined later, $\mathcal{F}_{0}(A, B, C)$ (resp. $\mathcal{F}_{\infty}(A, B, C)$) denotes the summation restricted to $\mathrm{N}(\ell r) \leq R$ (resp. $\mathrm{N}(\ell r) > R$).

For the initial contribution $\mathcal{F}_{0}(A, B, C)$, the application of \cref{lem:rho-phi,lem:Iwaniec-lemma-2,lem:squarefree} with $S = \frac{C}{\mathrm{N}(\ell r)}$ and $q \sim 4\ell r$ yields
\begin{align}
\mathcal{F}_{0}(A, B, C) &= \sum_{\substack{\mathrm{N}(\ell) \leq C \\ \ell \in \mathscr{L}}} \sum_{b \tpmod{\ell}} \rho(\ell, b) \sum_{\substack{\mathrm{N}(\ell rs) \leq C \\ \mathrm{N}(\ell r) \leq R \\ \ell \in \mathscr{L}, \, (rs, 4\ell) = 1}} \mu(rs)^{2}\\
& \quad \times \left(\pi A \frac{\mu(r)}{\mathrm{N}(\ell r)}+O_{\epsilon}(\mathrm{N}(r)^{\frac{1}{2}+\epsilon}+|B|^{\frac{1}{4}+\epsilon} \mathrm{N}(r)^{\frac{1}{4}+\epsilon}) \right)\\
& = \frac{\pi}{\zeta_{\mathbb{Q}(i)}(2)} AC \sum_{\substack{\mathrm{N}(\ell) \leq C \\ \ell \in \mathscr{L}}} \frac{\varphi(\ell)}{\mathrm{N}(\ell)^{2}} \sum_{\substack{\mathrm{N}(\ell r) \leq R \\ (r, 4\ell) = 1}} \frac{\mu(r)}{\mathrm{N}(r)^{2}} \prod_{p \mid 4 \ell r} \left(1+\frac{1}{\mathrm{N}(p)} \right)^{-1}\\
& \quad + O_{\epsilon}((AC^{\frac{1}{2}}+CR^{\frac{1}{2}}+|B|^{\frac{1}{4}} CR^{\frac{1}{4}})(BC)^{\epsilon}).
\end{align}
Moreover, extending the summation over $\ell$ and $r$ to infinity introduces an error of at most $\ll_{\epsilon} (ACR^{-\frac{1}{2}}+AC^{\frac{1}{2}})(|B|C)^{\epsilon}$. Hence, there exists an absolute and effectively computable~constant $\Cl[constant]{mean-value} > 0$ such that
\begin{equation}\label{eq:F-0}
\mathcal{F}_{0}(A, B, C) = \Cr{mean-value} AC+O_{\epsilon}((CR^{\frac{1}{2}}+|B|^{\frac{1}{4}} CR^{\frac{1}{4}}+ACR^{-\frac{1}{2}}+AC^{\frac{1}{2}})(|B|C)^{\epsilon}).
\end{equation}

The problem is now reduced to estimating $\mathcal{F}_{\infty}(A, B, C)$, which we intend to balance against the first term in the error term in~\eqref{eq:F-0}. To proceed, it is convenient to break summation~over $r$ in $\mathcal{F}_{\infty}(A, B, C)$ into intervals of the type $(R_{1}, R_{2}]$ with $\mathrm{N}(\ell)^{-1} R \leq R_{1} < R_{2} \leq 2R_{1}$, and to decouple the summation over $a$ according as $a \in \mathcal{M}$ or $a \in \overline{\mathcal{M}}$, where
\begin{equation}
\mathcal{M} \coloneqq \{B < \mathrm{N}(a) \leq A+B: \text{$L(s, \chi_{a^{2}-4})$ has a zero in $[\sigma, 1] \times [-\log R, \log R]$} \},
\end{equation}
and $\overline{\mathcal{M}}$ stands for the complement of $\mathcal{M}$. By \cref{prop:density}, the cardinality of $\mathcal{M}$ does not exceed
\begin{equation}\label{eq:card}
\mathrm{Card}(\mathcal{M}) \ll_{\epsilon} |B|^{\frac{6(1-\sigma)}{2-\sigma}+\epsilon}.
\end{equation}
If the restriction to $\mathcal{M}$ is indicated as a superscript by $\mathcal{F}^{\mathcal{M}}_{\infty}(A, B, C)$, then one may impose on the left-hand side of~\eqref{eq:alpha-beta} the condition that $0 \ne n \in \mathbb{Z}[i]$ be squarefree and coprime to a fixed Gaussian integer $0 \ne q \in \mathbb{Z}[i]$ in the spirit of Iwaniec and Szmidt~\cites[Section~6]{IwaniecSzmidt1985}. It then follows from~\eqref{eq:rho-lambda},~\eqref{eq:prime-powers},~\eqref{eq:card}, and \cref{lem:squarefree} that
\begin{align}
\mathcal{F}^{\mathcal{M}}_{\infty}(A, B, C) &\ll_{\epsilon} \sum_{\substack{\mathrm{N}(\ell) \leq R \\ \ell \in \mathscr{L}}} \sum_{a \in \mathcal{M}} \rho(\ell, a) \sum_{R_{1} \geq \mathrm{N}(\ell)^{-1} R} \frac{C}{\mathrm{N}(\ell) R_{1}} R_{1}^{\alpha} |B|^{2\beta}(|B|C)^{\epsilon}\\
&\ll_{\epsilon} |B|^{2\beta} CR^{\alpha-1}(|B|C)^{\epsilon} \sum_{a \in \mathcal{M}} \sum_{\substack{\mathrm{N}(\ell) \leq R \\ \ell \in \mathscr{L}}} \frac{\rho(\ell, a)}{\mathrm{N}(\ell)^{\alpha}}\\
&\ll_{\epsilon} |B|^{\frac{2(3+2\beta-(3+\beta)\sigma)}{2-\sigma}} CR^{\alpha-1}(|B|C)^{\epsilon}.
\end{align}
On the other hand, if $a \in \overline{\mathcal{M}}$, then one may utilise an unconditional result in \cref{prop:substitute}. The proof of \cref{prop:substitute} manipulates the summations over $a$ and $r$ in a separate~manner, and partial summation allows the inclusion and exclusion of $\rho(\ell, a)$. It is bounded on average by virtue of~\eqref{eq:prime-powers} and~\eqref{eq:uniformly}, and $R_{1} \geq \mathrm{N}(\ell)^{-1} R$ as above, thereby rendering the summation over $\ell$ essentially bounded as in the second display on~\cites[Page~156]{Iwaniec1984}. Hence, we deduce
\begin{equation}\label{eq:complement}
\mathcal{F}^{\overline{\mathcal{M}}}_{\infty}(A, B, C) \ll_{\epsilon} (A+|B|^{2(\delta-1)}) CR^{-\frac{1}{2}} (|B|C)^{\epsilon}.
\end{equation}
Consequently, if we optimise
\begin{equation}
R = \max(A^{\frac{4}{3}} |B|^{-\frac{1}{3}}, |B|^{\frac{4(3+2\beta-(3+\beta)\sigma)}{(3-2\alpha)(2-\sigma)}}),
\end{equation}
then there exists an absolute and effectively computable constant $\Cr{mean-value} > 0$ such that
\begin{multline}\label{eq:F-final}
\mathcal{F}(A, B, C) = \Cr{mean-value} AC+O_{\epsilon}((|B|^{\frac{2(3+2\beta-(3+\beta)\sigma)}{(3-2\alpha)(2-\sigma)}} C+|B|^{\frac{1}{4}+\frac{3+2\beta-(3+\beta)\sigma}{(3-2\alpha)(2-\sigma)}} C\\
 + |B|^{2(\delta-1)-\frac{2(3+2\beta-(3+\beta)\sigma)}{(3-2\alpha)(2-\sigma)}} C+A^{\frac{1}{3}} |B|^{\frac{1}{6}} C+AC^{\frac{1}{2}})(|B|C)^{\epsilon}).
\end{multline}
The comparison of the main terms in~\eqref{eq:trivial} and~\eqref{eq:F-final} with $|B| = A^{2}$ and $C = A^{\frac{1}{2}}$ as $A \to \infty$ identifies $\Cr{mean-value} = \frac{\pi}{\zeta_{\mathbb{Q}(i)}(2)}$, which now completes the proof of \cref{prop}.
\end{proof}

\section{Bypassing the Weil--Gundlach bound}\label{sect:spectral}
In anticipation of the proof of \cref{thm:conditional}, we now integrate the automorphic techniques of Iwaniec~\cites[Section~10]{Iwaniec1984} with the framework of Koyama~\cites[Section~4]{Koyama2001}, the latter constituting a $3$-dimensional analogue of the framework of Luo and Sarnak~\cites[Appendix]{LuoSarnak1995}. A difference between these techniques lies in their treatment of Kloosterman sums: the former features an expression of Kloosterman sums in terms of $\rho(c, a)$, thereby reducing the problem to estimating nontrivially certain quadratic character sums, while the latter relies~exclusively on the Weil--Gundlach bound~\eqref{eq:Weil}. Among other crucial observations, our argument unveils an oscillatory factor in a certain well-behaved integral transform of Bessel kernels, appearing in the spherical Kuznetsov formula of Motohashi~\cites[Theorem]{Motohashi1996}[Theorem]{Motohashi1997}.

Utilising the explicit formula~\cites[Theorem~5.2]{Nakasuji2000}[Theorem~4.1]{Nakasuji2001} and its smoothed version~\cites[Lemma~4.1]{BalogBiroCherubiniLaaksonen2022}, we are naturally led to the problem of determining nontrivial bounds for the exponentially weighted spectral exponential sum of the shape
\begin{equation}\label{eq:normalised}
\sum_{t_{j} \leq T} X^{it_{j}} e^{-\frac{t_{j}}{T}}.
\end{equation}
To approximate the summand, the commonly chosen test function in the spherical Kuznetsov formula~\cites[Theorem]{Motohashi1996}[Theorem]{Motohashi1997}, whose Bessel--Kuznetsov transform behaves like $x$ as $x \to 0$, proves sufficient in our scenario. This phenomenon stands in contrast to the application of the Hecke trick in~\cites[Section~5]{BalkanovaFrolenkov2019-2} in the $2$-dimensional setting, where some convergence issues emerge when shifting the contour of integration to the desired region.

For $x > 0$ and $T, X \geq 2$, we define
\begin{equation}\label{eq:varphi}
\varphi(x) \coloneqq \frac{\sinh \beta}{\pi} x \exp(ix \cosh \beta), \qquad 2\beta = \log X+\frac{i}{T},
\end{equation}
where
\begin{equation}\label{eq:c-a-b}
i \cosh \beta = -a+ib, \qquad 
	\begin{cases}
	a \coloneqq \sinh(\log X^{\frac{1}{2}}) \sin((2T)^{-1}),\\
	b \coloneqq \cosh(\log X^{\frac{1}{2}}) \cos((2T)^{-1}).
	\end{cases}
\end{equation}
The Bessel--Kuznetsov transform given by\footnote{We here employ a less common accent symbol to distinguish it from the Fourier transform. Furthermore, a typographical error exists in the definition of $\breve{\varphi}(t)$ in~\cite[Page~68, Line~$-1$]{DeshouillersIwaniec1986} and~\cite[Page~233, Line~$-3$]{LuoSarnak1995}, wherein the imaginary unit $i$ was erroneously placed in the denominator rather than the numerator.}
\begin{equation}\label{eq:Bessel-Kuznetsov}
\breve{\varphi}(t) \coloneqq \frac{\pi i}{2\sinh \pi t} \int_{0}^{\infty} (J_{2it}(x)-J_{-2it}(x)) \varphi(x) \frac{dx}{x}
\end{equation}
admits the following approximation due to Deshouillers and Iwaniec~\cites[Lemma~7]{DeshouillersIwaniec1986}.
\begin{lemma}\label{lem:Heckes-trick}
Keep the notation as above. Let $T, X \geq 2$, and let $t > 0$. Then
\begin{equation}
\breve{\varphi}(t) = \frac{\sinh((\pi+2i \beta)t)}{\sinh \pi t} = X^{it} e^{-\frac{t}{T}}+O(e^{-\pi t}).
\end{equation}
\end{lemma}

We fix a smooth compactly supported weight function $h: (0, \infty) \to \mathbb{R}$ having holomorphic Mellin transform $\tilde{h}: \mathbb{C} \to \mathbb{C}$. It is convenient to choose $h$ to be supported on a dyadic interval $[\sqrt{N}, \sqrt{2N}]$, whose derivatives satisfy
\begin{equation}\label{eq:support}
h^{(\ell)}(\xi) \ll_{\ell} N^{-\frac{\ell}{2}}, \qquad \ell \in \mathbb{N}_{0},
\end{equation}
and whose mass equals
\begin{equation}\label{eq:mass}
\int_{-\infty}^{\infty} h(\xi) \xi \, d\xi = \tilde{h}(2) = N.
\end{equation}
Integrating by parts $\ell$ times and using~\eqref{eq:support} implies
\begin{equation}\label{eq:h-tilde}
\tilde{h}(s) = \frac{(-1)^{\ell}}{s(s+1) \cdots (s+\ell-1)} 
\int_{0}^{\infty} h^{(\ell)}(x) x^{s+\ell} \frac{dx}{x} \ll_{\sigma, \ell} N^{\frac{\sigma}{2}}(|s|+1)^{-\ell},
\end{equation}
where the implicit constant depends continuously on $\sigma \coloneqq \Re(s)$. The identity~\eqref{eq:h-tilde} is meant initially for $s \not\in \mathbb{Z}_{\leq 0}$, although it remains valid even at these exceptional points.

The complete unconditional resolution of the mean Lindel\"{o}f hypothesis (\cref{thm:mean-Lindelof})~now ensures that~\cites[Lemma~4.3]{Koyama2001} holds unconditionally, namely there exists an absolute and effectively computable constant $\Cl[constant]{Koyama-lemma} > 0$ such that
\begin{equation}
1 = \frac{\Cr{Koyama-lemma}}{N} \left(\sum_{n \in \mathbb{Z}[i]} \frac{t_{j}}{\sinh \pi t_{j}} h(|n|) |\rho_{j}(n)|^{2}-\mathfrak{r}(t_{j}, N) \right),
\end{equation}
where $\mathfrak{r}(t_{j}, N)$ serves as an error term satisfying
\begin{equation}
\sum_{t_{j} \leq T} |\mathfrak{r}(t_{j}, N)| \ll_{\epsilon} T^{3+\epsilon} N^{\frac{1}{2}}.
\end{equation}
Using \cref{lem:Heckes-trick} and the spherical Kuznetsov formula~\cites[Theorem]{Motohashi1996}[Theorem]{Motohashi1997} with $m = n$, $h = \breve{\varphi}$ and estimating trivially the diagonal and Eisenstein terms by
\begin{equation}
2\pi \int_{-\infty}^{\infty} \frac{\sigma_{it}(n)^{2}}{|n|^{2it} |\zeta_{\mathbb{Q}(i)}(1+it)|^{2}} \breve{\varphi}(t) \, dt \ll \tau(n)^{2} T(\log T)^{2}, \qquad \pi^{-2} \int_{-\infty}^{\infty} t^{2} \breve{\varphi}(t) \, dt \ll T^{2},
\end{equation}
understanding the behaviour of the spectral exponential sum~\eqref{eq:normalised} boils down to a nontrivial treatment of the spectral-arithmetic average
\begin{align}
\sum_{t_{j} \leq T} X^{it_{j}} e^{-\frac{t_{j}}{T}} &= \frac{c_{6}}{N} \sum_{n \in \mathbb{Z}[i]} h(|n|) \sum_{j = 1}^{\infty} \frac{t_{j}}{\sinh \pi t_{j}} \breve{\varphi}(t_{j}) |\rho_{j}(n)|^{2}+O(T^{3+\epsilon} N^{-\frac{1}{2}+\epsilon})\\
& = \frac{c_{6}}{N} \sum_{n \in \mathbb{Z}[i]} h(|n|) \sum_{0 \ne c \in \mathbb{Z}[i]} \frac{S(n, n, c)}{\mathrm{N}(c)} \mathring{\varphi} \left(\frac{2\pi \overline{n}}{c} \right)+O(T^{3+\epsilon} N^{-\frac{1}{2}+\epsilon}),\label{eq:spectral-arithmetic}
\end{align}
where~\cites[Theorem~3.2]{BalkanovaFrolenkov2020}[Theorem~2]{Motohashi2001}
\begin{equation}
\mathring{\varphi}(z) \coloneqq \int_{-\infty}^{\infty} \frac{it^{2}}{\sinh \pi t} \mathcal{J}_{it}(z) \breve{\varphi}(t) \, dt = \frac{1}{2} \int_{-\infty}^{\infty} t^{2} \mathcal{K}_{it}(z) \breve{\varphi}(t) \, dt.
\end{equation}
In particular, the Bessel kernels are defined by
\begin{equation}
\mathcal{J}_{\nu}(z) \coloneqq 2^{-2\nu} |z|^{2\nu} J_{\nu}^{\ast}(z) J_{\nu}^{\ast}(\overline{z}),
\qquad \mathcal{K}_{\nu}(z) \coloneqq \frac{\mathcal{J}_{-\nu}(z)-\mathcal{J}_{\nu}(z)}{\sin \pi \nu},
\end{equation}
where $J_{\nu}^{\ast}(t)$ is the entire function equal to $J_{\nu}(t) (\frac{t}{2})^{-\nu}$ when $t > 0$.

We are now prepared to establish the following result.
\begin{proposition}\label{prop:extra}
Let $N, T, X \geq 2$. Suppose that $(\alpha, \beta)$ is a quadratic exponent pair over $\mathbb{Q}(i)$. Then we have for some fixed $\delta \in [\frac{5}{4}, \frac{3}{2}]$ and for any $\sigma \in [\frac{1}{2}, 1]$ and $\epsilon > 0$ that
\begin{multline}\label{eq:extra}
\sum_{t_{j} \leq T} X^{it_{j}} \ll_{\epsilon} TX^{\frac{2(3+2\beta-(3+\beta)\sigma)}{(3-2\alpha)(2-\sigma)}+\epsilon}+TX^{\frac{1}{4}+\frac{3+2\beta-(3+\beta)\sigma}{(3-2\alpha)(2-\sigma)}+\epsilon}\\
 + TX^{2(\delta-1)-\frac{2(3+2\beta-(3+\beta)\sigma)}{(3-2\alpha)(2-\sigma)}+\epsilon}+T^{\frac{2}{3}} X^{\frac{1}{2}+\epsilon}+T^{2+\epsilon}.
\end{multline}
\end{proposition}

\begin{proof}
It follows from~\cites[Lemmata~3.2 and~3.3]{Kaneko2022-2} or~\cites[Lemmata~4.1 and~4.3]{ChatzakosCherubiniLaaksonen2022} that\footnote{Using an asymptotic form of the stationary phase analysis~\cites[Section~8]{BlomerKhanYoung2013}[Main Theorem]{KiralPetrowYoung2019}, the expression~\eqref{eq:Kaneko-CCL} is essentially equivalent to~\cites[Equation~(4.7)]{Koyama2001}. Nonetheless, Koyama addresses the right-hand side of~\eqref{eq:Kaneko-CCL} with absolute values and analyse the sum over $c$ via the Weil--Gundlach bound~\eqref{eq:Weil}.}
\begin{multline}\label{eq:Kaneko-CCL}
\sum_{0 \ne c \in \mathbb{Z}[i]} \frac{S(n, n, c)}{\mathrm{N}(c)} \mathring{\varphi} \left(\frac{2\pi \overline{n}}{c} \right)\\
 = 2\pi i\Xi^{2} \mathrm{N}(n) X \sum_{C_{1} \leq \mathrm{N}(c) \leq C_{2}} \frac{S(n, n, c)}{\mathrm{N}(c)^{2}} K_{0} \bigg(\frac{2\pi \Xi |n| X^{\frac{1}{2}}}{|c|} \bigg)+O_{\epsilon}(N^{\frac{1}{2}+\epsilon} X^{\frac{1}{2}+\epsilon}),
\end{multline}
with $\Xi \coloneqq e^{-i(\pi-1/T)/2}$, $C_{1} \coloneqq NX(T \log T)^{-2}$, and $C_{2} \coloneqq NX$. Here, it is convenient~to~exploit an additional oscillation from $K_{0}$ to detect nontrivial cancellations. Applying the asymptotic expansion of $K_{0}$, invoking \cref{lem:counting-function}, and rewriting $\check{e}(\cdot)$ in terms of $e(\cdot)$ demonstrates that the right-hand side~of~\eqref{eq:Kaneko-CCL} averaged over $n$ as in~\eqref{eq:spectral-arithmetic} boils down essentially to
\begin{equation}\label{eq:akin-to-Cai}
\frac{X^{\frac{3}{4}}}{N} \sum_{C_{1} \leq \mathrm{N}(c) \leq C_{2}} \frac{1}{\mathrm{N}(c)^{\frac{7}{4}}} \sum_{a \tpmod{c}} \rho(c, a) \sum_{n \in \mathbb{Z}[i]} h(|n|) \mathrm{N}(n)^{\frac{3}{4}} e \bigg(\frac{\Re(an \overline{c})+|nc| X^{\frac{1}{2}}}{|c|^{2}} \bigg)+O(X^{\frac{1}{2}}).
\end{equation}
By dyadic subdivisions, one may break up the sum over $c$ into at most $O(\log X)$ subsums~over $C < \mathrm{N}(c) \leq 2C$ with $C_{1} \leq C < 2C \leq C_{2}$. By applying $2$-dimensional Poisson summation in $n$, isolating zero frequency, and eliding the extra factor of $2\pi$ from the polar coordinates, it follows that
\begin{multline}
\sum_{n \in \mathbb{Z}[i]} h(|n|) \mathrm{N}(n)^{\frac{3}{4}} e \bigg(\frac{\Re(an \overline{c})+|nc| X^{\frac{1}{2}}}{|c|^{2}} \bigg) = \int_{0}^{\infty} h(x) x^{\frac{5}{2}} e \bigg(\frac{xX^{\frac{1}{2}}}{|c|} \bigg) J_{0} \left(\frac{x \, \Re(a \overline{c})}{|c|^{2}} \right) \, dx\\
 + \sum_{0 \ne n \in \mathbb{Z}[i]} \int_{0}^{\infty} h(x) x^{\frac{5}{2}} e \bigg(\frac{xX^{\frac{1}{2}}}{|c|} \bigg) J_{0} \left(|n|x+\frac{x \, \Re(a \overline{c})}{|c|^{2}} \right) \, dx.
\end{multline}
Integration by parts now establishes that the contribution of the second term is bounded~and that the first term remains bounded unless $|a+X^{\frac{1}{2}}| \leq C^{\frac{1}{2}} N^{-\frac{1}{2}+\epsilon}$ by the second line of~\eqref{eq:J-0}. Making a change of variables $x \mapsto |c|x$ and using partial summation leads to the expression
\begin{equation}
\frac{X^{\frac{1}{2}}}{N} \int_{0}^{\infty} \sum_{C < \mathrm{N}(c) \leq 2C} \sum_{|a+X^{\frac{1}{2}}| \leq C^{\frac{1}{2}} N^{-\frac{1}{2}+\epsilon}} \rho(c, a) e(x(\Re(a \overline{c}) |c|^{-1}+X^{\frac{1}{2}})) h(|c|x) x^{2} \, dx.
\end{equation}
Furthermore, it is convenient to make a change of variables in \cref{prop}, namely
\begin{equation}
A \mapsto C^{\frac{1}{2}} N^{-\frac{1}{2}} X^{\frac{1}{2}+\epsilon}, \qquad B \mapsto X,
\end{equation}
obtaining (cf.~\cites[Page~70]{Cai2002})
\begin{multline}
\sum_{C < \mathrm{N}(c) \leq 2C} \sum_{|a+X^{\frac{1}{2}}| \leq C^{\frac{1}{2}} N^{-\frac{1}{2}+\epsilon}} \rho(c, a) e(x(\Re(a \overline{c}) |c|^{-1}+X^{\frac{1}{2}})) h(|c|x)\\
\ll_{\epsilon} CX^{\frac{2(3+2\beta-(3+\beta)\sigma)}{(3-2\alpha)(2-\sigma)}+\epsilon}+CX^{\frac{1}{4}+\frac{3+2\beta-(3+\beta)\sigma}{(3-2\alpha)(2-\sigma)}+\epsilon} +CX^{2(\delta-1)-\frac{2(3+2\beta-(3+\beta)\sigma)}{(3-2\alpha)(2-\sigma)}+\epsilon}+C^{\frac{7}{6}} N^{-\frac{1}{6}} X^{\frac{1}{3}+\epsilon},
\end{multline}
where the contribution of the main term becomes negligibly small by repeated integration by parts. Invoking $\mathrm{supp}(h) \subseteq [\sqrt{N}, \sqrt{2N}]$, inserting $C = \frac{NX}{(T \log T)^{2}}$, and integrating trivially~over $x$ implies that the entire contribution of the error term amounts to
\begin{equation}\label{eq:entire-contribution}
\ll_{\epsilon} TX^{\frac{2(3+2\beta-(3+\beta)\sigma)}{(3-2\alpha)(2-\sigma)}+\epsilon}+TX^{\frac{1}{4}+\frac{3+2\beta-(3+\beta)\sigma}{(3-2\alpha)(2-\sigma)}+\epsilon}+TX^{2(\delta-1)-\frac{2(3+2\beta-(3+\beta)\sigma)}{(3-2\alpha)(2-\sigma)}+\epsilon}+T^{\frac{2}{3}} X^{\frac{1}{2}+\epsilon},
\end{equation}
where $N = T^{2} X^{\epsilon}$. Inserting~\eqref{eq:entire-contribution} into~\eqref{eq:spectral-arithmetic} thus completes the proof of \cref{prop:extra}.
\end{proof}

To enable the ensuing application of partial summation, we establish the following result.
\begin{proposition}\label{prop:key}
Let $N, T, X \geq 2$. Suppose that $(\alpha, \beta)$ is a quadratic exponent pair over $\mathbb{Q}(i)$. Then we have for some fixed $\delta \in [\frac{5}{4}, \frac{3}{2}]$ and for any $\sigma \in [\frac{1}{2}, 1]$ and $\epsilon > 0$ that
\begin{multline}
\sum_{t_{j} \leq T} X^{it_{j}} \ll_{\epsilon} T^{\frac{16}{13}} X^{\frac{1}{13}+\frac{18(3+2\beta-(3+\beta)\sigma)}{13(3-2\alpha)(2-\sigma)}+\epsilon}+T^{\frac{16}{13}} X^{\frac{1}{4}+\frac{9(3+2\beta-(3+\beta)\sigma)}{13(3-2\alpha)(2-\sigma)}+\epsilon}\\
 + T^{\frac{16}{13}} X^{\frac{18\delta-17}{13}-\frac{18(3+2\beta-(3+\beta)\sigma)}{13(3-2\alpha)(2-\sigma)}+\epsilon}+TX^{\frac{11}{26}+\epsilon}+T^{\frac{25}{13}} X^{\frac{1}{13}+\epsilon}.
\end{multline}
\end{proposition}

\begin{proof}
The claim follows upon interpolating~\eqref{eq:extra} and~\cites[Equation~(3-48)]{Kaneko2022-2} with $\eta = 0$ (\cref{thm:mean-Lindelof}) via the standard inequality $\min(A, B) \leq A^{\gamma} B^{1-\gamma}$ with $\gamma = \frac{4}{13}$.
\end{proof}

\section{Endgame: proofs of \texorpdfstring{\cref{thm:unconditional,thm:second-moment,thm:conditional}}{}}\label{sect:proof}
It is convenient to borrow the smooth explicit formula of Balog et al.~\cites[Lemma~4.1]{BalogBiroCherubiniLaaksonen2022} owing to the barrier $O_{\epsilon}(X^{\frac{3}{2}+\epsilon})$ inherent in the nonsmooth explicit formula of Nakasuji~\cites[Theorem~5.2]{Nakasuji2000}[Theorem~4.1]{Nakasuji2001}. In conjunction with the machinery of Soundararajan~and Young~\cites[Section~3]{SoundararajanYoung2013}, let $Y \in [x^{\frac{1}{2}}, x]$ be an auxiliary parameter to be chosen later.~Note that $\frac{1+2\vartheta}{3} \leq \frac{1}{2}$ is equivalent to $\vartheta \leq \frac{1}{4}$, thereby ensuring the subsequent usage of \cref{thm:Brun-Titchmarsh}.

Fix a smooth function $k: (0, \infty) \to [0, \infty)$ supported on $[Y, 2Y]$, satisfying $\tilde{k}(0) = 1$ and
\begin{equation}
\int_{-\infty}^{\infty} |k^{(\ell)}(u)| \, du \ll_{\ell} Y^{-\ell}, \qquad \ell \in \mathbb{N}_{0}.
\end{equation}
Then~\cites[Lemma~4.1]{BalogBiroCherubiniLaaksonen2022} coupled with \cref{thm:Brun-Titchmarsh} implies (cf.~\cites[Equation~(3-64)]{Kaneko2022-2})
\begin{equation}\label{eq:approximation}
\mathcal{E}_{\Gamma}(x) = 2 \, \Re \left(\sum_{t_{j} \leq \frac{x^{1+\epsilon}}{Y}} \int_{Y}^{2Y} \frac{(x+u)^{s_{j}}}{s_{j}} k(u) \, du \right)+O_{\epsilon}(x^{1+\vartheta+\epsilon} Y^{\frac{1}{2}}+x^{\frac{5}{4}+\frac{\vartheta}{2}+\epsilon} Y^{\frac{1}{4}}),
\end{equation}
where $s_{j} = 1+it_{j}$. By estimating the integral on the right-hand side using the supremum~of the integrand, we obtain
\begin{equation}\label{eq:sup}
\sum_{t_{j} \leq \frac{x^{1+\epsilon}}{Y}} \int_{Y}^{2Y} \frac{(x+u)^{s_{j}}}{s_{j}} k(u) \, du 
\ll x \sup_{Y \leq u \leq 2Y} \left|\sum_{t_{j} \leq \frac{x^{1+\epsilon}}{Y}} \frac{(x+u)^{it_{j}}}{1+it_{j}} \right|.
\end{equation}
It remains to address the spectral exponential sum within the absolute value. To this end,~let $\mathcal{S}(T, X)$ denote the spectral exponential sum on the left-hand side of~\eqref{eq:extra}. If $T^{\ast} \coloneqq \frac{x^{1+\epsilon}}{Y}$~and $X \coloneqq x+u$ with $u \in [Y, 2Y]$, then partial summation shows
\begin{equation}\label{eq:partial-summation}
\sum_{t_{j} \leq \frac{x^{1+\epsilon}}{Y}} \frac{(x+u)^{it_{j}}}{1+it_{j}} = \int_{1}^{T^{\ast}} \frac{d\mathcal{S}(T, X)}{1+iT} = \frac{\mathcal{S}(T^{\ast}, X)}{1+iT^{\ast}}+i \int_{1}^{T^{\ast}} \frac{\mathcal{S}(T, X)}{(1+iT)^{2}} dT,
\end{equation}
thereby reducing the problem to the estimation of $\mathcal{S}(T, X)$. We are now prepared to establish \cref{thm:unconditional,thm:second-moment,thm:conditional} in one fell swoop.

\subsection{Proof of \cref{thm:unconditional}}
We follow the proof strategy in~\cites[Section~4D]{Kaneko2022-2} \textit{mutatis mutandis}. The complete resolution of the mean Lindel\"{o}f hypothesis given in \cref{thm:mean-Lindelof} now guarantees $\eta = 0$ in accordance with the notation therein. Gathering together~\eqref{eq:approximation},~\eqref{eq:partial-summation}, and~\cites[Equation~(3-48)]{Kaneko2022-2} and optimising $Y = \min(x^{\frac{3}{4}-\frac{\vartheta}{2}}, x^{\frac{4(1-\vartheta)}{5}})$ yields
\begin{equation}\label{eq:combination}
\mathcal{E}_{\Gamma}(x) \ll_{\epsilon} x^{2+\epsilon} Y^{-\frac{3}{4}}+x^{1+\vartheta+\epsilon} Y^{\frac{1}{2}}+x^{\frac{5}{4}+\frac{\vartheta}{2}+\epsilon} Y^{\frac{1}{4}} \ll_{\epsilon} x^{\frac{23+6\vartheta}{16}+\epsilon}+x^{\frac{7+3\vartheta}{5}+\epsilon},
\end{equation}
from which the first claim of \cref{thm:unconditional} follows. The second claim follows from substituting into~\eqref{eq:combination} the best known subconvex exponent $\vartheta = \frac{1}{6}$ due to Nelson~\cites[Theorem~1.1]{Nelson2020}.

\subsection{Proof of \cref{thm:second-moment}}
A salient difference from the proof of \cref{thm:unconditional} lies in our interpretation of the parameter $\eta$, which cannot be assumed to be $0$, given the current state of knowledge. Let $\eta \in [0, \min(\frac{1}{4}, 2\vartheta)]$ denote an additional exponent towards \cref{conj:square-mean}. The antepenultimate display on~\cites[Page~1871]{Kaneko2022-2} then translates to
\begin{align}
\frac{1}{\Delta} \int_{V}^{V+\Delta} |\mathcal{E}_{\Gamma}(x)|^{2} \, dx &\ll_{\epsilon} \Delta^{-1} V^{4+\eta+\epsilon} Y^{-\frac{1}{2}-\eta}+V^{2(1+\vartheta)+\epsilon} Y+V^{\frac{5}{2}+\vartheta+\epsilon} Y^{\frac{1}{2}}\\
&\ll_{\epsilon} \Delta^{-\frac{1}{2(1+\eta)}} V^{\frac{13+2\vartheta+4(3+\vartheta) \eta}{4(1+\eta)}}+\Delta^{-\frac{2}{3+2\eta}} V^{\frac{2(5+\vartheta+(3+2\vartheta) \eta)}{3+2\eta}+\epsilon},
\end{align}
where
\begin{equation}
V^{\frac{1}{2}} \leq Y = \min(\Delta^{-\frac{1}{1+\eta}} V^{\frac{3-2\vartheta+2\eta}{2(1+\eta)}}, \Delta^{-\frac{2}{3+2\eta}} V^{\frac{2(2(1-\vartheta)+\eta)}{3+2\eta}}) \leq V, \qquad V^{\frac{1}{2}-\vartheta} \leq \Delta \leq V^{1-\vartheta+\frac{\eta}{2}}.
\end{equation}
Now \cref{thm:second-moment} follows from renaming the variables.

\subsection{Proof of \cref{thm:conditional}}
It follows from \cref{prop:key} and~\eqref{eq:partial-summation} that the first~term in~\eqref{eq:approximation} is bounded by
\begin{multline}\label{eq:E}
\ll_{\epsilon} x^{\frac{17}{13}+\frac{18(3+2\beta-(3+\beta)\sigma)}{13(3-2\alpha)(2-\sigma)}+\epsilon} Y^{-\frac{3}{13}}+x^{\frac{77}{52}+\frac{9(3+2\beta-(3+\beta)\sigma)}{13(3-2\alpha)(2-\sigma)}+\epsilon} Y^{-\frac{3}{13}}\\
 + x^{\frac{18\delta-1}{13}-\frac{18(3+2\beta-(3+\beta)\sigma)}{13(3-2\alpha)(2-\sigma)}+\epsilon} Y^{-\frac{3}{13}}+x^{\frac{37}{26}+\epsilon}+x^{2+\epsilon} Y^{-\frac{12}{13}}.
\end{multline}
The first term dominates the last term whenever $Y \gg x^{1-\frac{2(3+2\beta-(3+\beta)\sigma)}{(3-2\alpha)(2-\sigma)}+\epsilon}$. It is now convenient to minimise $Y$ so that
\begin{equation}
Y = x^{1-\frac{2(3+2\beta-(3+\beta)\sigma)}{(3-2\alpha)(2-\sigma)}+\epsilon}.
\end{equation}
Hence, we conclude from~\eqref{eq:approximation} that
\begin{multline}
\mathcal{E}_{\Gamma}(x) \ll_{\epsilon} x^{\frac{14}{13}+\frac{24(3+2\beta-(3+\beta)\sigma)}{13(3-2\alpha)(2-\sigma)}+\epsilon}+x^{\frac{5}{4}+\frac{15(3+2\beta-(3+\beta)\sigma)}{13(3-2\alpha)(2-\sigma)}+\epsilon}+x^{\frac{18\delta-4}{13}-\frac{12(3+2\beta-(3+\beta)\sigma)}{13(3-2\alpha)(2-\sigma)}+\epsilon}\\
 + x^{\frac{3}{2}+\vartheta-\frac{3+2\beta-(3+\beta)\sigma}{(3-2\alpha)(2-\sigma)}+\epsilon}+x^{\frac{3}{2}+\frac{\vartheta}{2}-\frac{3+2\beta-(3+\beta)\sigma}{2(3-2\alpha)(2-\sigma)}+\epsilon}+x^{\frac{37}{26}+\epsilon}.\label{eq:altogether}
\end{multline}
By~\eqref{eq:alpha-beta} and \cref{prop:exponent-pair}, the optimal choice of $\sigma$, based on the second and penultimate terms, boils down to
\begin{equation}\label{eq:optimal-choice}
0.88793 \cdots = \frac{103}{116} \leq \sigma = \frac{206+(68-52\vartheta^{\prime}) \vartheta+232\vartheta^{\prime}}{232+(34-26\vartheta^{\prime}) \vartheta+245\vartheta^{\prime}} \leq \frac{446}{481} = 0.92723 \cdots
\end{equation}
for any $\vartheta^{\prime} \geq 0$, which obeys the assumption in \cref{prop:key}, thereby justifying its validity. It is straightforward to verify from~\eqref{eq:altogether} and~\eqref{eq:optimal-choice} that the resulting bound is independent of $\vartheta^{\prime}$, namely
\begin{equation}
\mathcal{E}_{\Gamma}(x) \ll_{\epsilon} x^{\frac{245}{172}+\frac{15\vartheta}{43}+\epsilon}+x^{\frac{18\delta}{13}-\frac{250}{559}-\frac{12\vartheta}{43}+\epsilon}.
\end{equation}
Furthermore, by the definition~\eqref{eq:delta} of $\delta \in [\frac{5}{4}, \frac{3}{2}]$, one must equate
\begin{equation}
\delta = \max \left(\frac{245}{172}+\frac{15\vartheta}{43}, \frac{18\delta}{13}-\frac{250}{559}-\frac{12\vartheta}{43} \right) \quad \Longleftrightarrow \quad \delta = \frac{245}{172}+\frac{15\vartheta}{43}.
\end{equation}
Consequently, gathering the above observations together implies
\begin{equation}
\mathcal{E}_{\Gamma}(x) \ll_{\epsilon} x^{\frac{245}{172}+\frac{15\vartheta}{43}+\epsilon},
\end{equation}
as required. The proof of \cref{thm:conditional} is thus complete.


\newcommand{\etalchar}[1]{$^{#1}$}
\providecommand{\bysame}{\leavevmode\hbox to3em{\hrulefill}\thinspace}
\providecommand{\MR}{\relax\ifhmode\unskip\space\fi MR }
\providecommand{\MRhref}[2]{%
  \href{http://www.ams.org/mathscinet-getitem?mr=#1}{#2}
}
\providecommand{\Zbl}{\relax\ifhmode\unskip\space\thinspace\fi Zbl }
\providecommand{\MR}[2]{%
  \href{https://zbmath.org/?q=an:#1}{#2}
}
\providecommand{\doi}{\relax\ifhmode\unskip\space\thinspace\fi DOI }
\providecommand{\MR}[2]{%
  \href{https://doi.org/#1}{#2}
}
\providecommand{\SSNI}{\relax\ifhmode\unskip\space\thinspace\fi ISSN }
\providecommand{\MR}[2]{%
  \href{#1}{#2}
}
\providecommand{\ISBN}{\relax\ifhmode\unskip\space\thinspace\fi ISBN }
\providecommand{\MR}[2]{%
  \href{#1}{#2}
}
\providecommand{\arXiv}{\relax\ifhmode\unskip\space\thinspace\fi arXiv }
\providecommand{\MR}[2]{%
  \href{#1}{#2}
}
\providecommand{\href}[2]{#2}

\end{document}